\documentclass{TPmod}

\DeclareMathOperator{\NE}{\mathrm{NE}}

\newcommand{\power}[1]{[\![#1]\!]}

\newcommand{\hooklongrightarrow}{\lhook\joinrel\longrightarrow}
\newcommand{\hooklongleftarrow}{\longleftarrow\joinrel\rhook}

\newcommand{\Rbar}{\overline{\EuR}}
\newcommand{\Sbar}{\overline{\EuS}}
\newcommand{\Mbar}{\overline{\mathcal{M}}}
\newcommand{\CP}{\mathbb{CP}}

\def\fm{\mathfrak{m}}
\def\bc{\mathsf{bc}}

\def\snc{{$\mathsf{snc}$}}
\def\cM{\mathcal{M}}

\def\bY{\mathbf{Y}}
\def\ff{\mathfrak{f}}
\def\vv{\mathsf{q}}
\def\tang{\mathsf{t}}

\def\Nef{\mathsf{N}}
\def\fuk{\EuF}
\def\sph{\mathsf{sph}}
\def\disc{\mathsf{disc}}
\def\Sym{\mathrm{Sym}}
\def\bL{\mathbf{L}}
\def\by{\mathbf{y}}
\def\num{\mathrm{num}}
\def\tang{\mathsf{t}}
\def\nov{T}
\def\bK{\mathbf{K}}
\def\bP{\mathbf{P}}

\def\bi{\mathbf{i}}
\def\pun{\mathsf{pun}}
\def\for{\mathsf{for}}
\def\can{\mathsf{can}}

\def\Ymax{\EuY^{\mathsf{max}}}
\def\bN{\mathbf{N}}

\begin{document}

\title{Constructing the relative Fukaya category}

\author[Perutz and Sheridan]{Timothy Perutz and Nick Sheridan}

\begin{abstract}
We give a definition of Seidel's `relative Fukaya category', for a smooth complex projective variety relative to a simple normal crossings divisor, under a semipositivity assumption.
We use the Cieliebak--Mohnke approach to transversality via stabilizing divisors.  Two features of our construction are noteworthy: that we work relative to a normal crossings divisor which supports an effective ample divisor but need not have ample components;  and that our relative Fukaya category is linear over a certain ring of multivariate power series with \emph{integer} coefficients. 
\end{abstract}

\maketitle

\section{Introduction}

The main construction of this paper is the relative Fukaya category of a smooth complex projective variety relative to a simple normal crossings divisor, under a semipositivity assumption. 
We introduce it in Section \ref{sec:alg_geom}. 
It is a special case of a more general, purely symplectic construction, which we introduce first in Section \ref{sec:symp_rel_fuk}.

\subsection{Symplectic relative Fukaya category}\label{sec:symp_rel_fuk}

Let $X$ be a compact $2n$-dimensional manifold, and $W \subset X$ a subdomain (by which we mean a compact codimension-$0$ submanifold with boundary). 
Let $\omega$ be a symplectic form on $X$, and $\theta$ a primitive for $\omega|_W$, making $W$ into a Liouville domain. 

\begin{defn}\label{def:sod}
A \emph{system of divisors} $V$ for $W \subset X$ is a finite set $\{V_q\}_{q \in Q}$ of codimension-2 closed submanifolds of $X \setminus W$, such that: 
\begin{itemize}
\item the intersection $\cap_{q \in I} V_q$ is transverse for each $I \subset Q$;
\item the classes $PD([V_q])$ span $H^2(X,W;\Q)$;
\item for each $q \in Q$, the set of $r \in Q$ such that $PD([V_r]) = PD([V_{q}])$ in $H^2(X,W)$ has size $\ge n+1$.
\end{itemize}
\end{defn}

Given the data $(W \subset X, \omega,\theta,V)$ described above, we define $\Nef(V) \subset H^2(X,W;\R)$ to be the convex cone spanned by the classes $PD([V_q])$. 
We denote by $\kappa \in H^2(X,W;\R)$ the de Rham cohomology class represented by $[\omega;\theta]$.
We assume that $(W \subset X,\omega,\theta,V)$ satisfies the following conditions:
\begin{enumerate}
\item [] \textbf{(Ample$'$)} The class $\kappa$ lies in the interior of $\Nef(V)$; 
\item []\textbf{(Semipositive$'$)} There exists a class $\tilde{c}_1 \in \Nef(V)$ which is a lift of $c_1(TX) \in H^2(X)$ along $H^2(X,W) \to H^2(X)$. 
\end{enumerate}

Let $J_0$ be an $\omega$-compatible almost-complex structure on $X$, which makes each $V_q$ into an almost complex submanifold, and satisfies a certain convexity condition near the boundary of $W$. (Precisely, there should exist a `convex collar' for $W$, see Definition \ref{def:conv_coll}.)

All holomorphic curves $u: \Sigma \to X$ we consider will have boundary in $W$, and non-negative intersection number with the divisors $V_q$ (Lemma \ref{lem:posint}), so the class $[u]$ will lie in the monoid $\NE(V) \subset H_2(X,W)$ dual to $\Nef(V)$. 

We consider the monoid ring $\tilde{R}(V) := \Z[\NE(V)]$. 
Our assumption that the classes $PD([V_q])$ span $H^2(X,W;\Q)$ implies that $\Nef(V)$ is full-dimensional, and hence $\NE(V)$ is strictly convex.\footnote{Recall that this means $\NE(V) \cap -\NE(V) = \{0\}$. Note that $\NE(V)$ need not be full-dimensional, although it will be in the algebro-geometric context of Section \ref{sec:alg_geom}.}  
This implies that the submodule $\tilde{\fm} \subset \tilde{R}(V)$, spanned by $\nov^\beta$ for $\beta \in \NE(V) \setminus \{0\}$, is an ideal. 
Because $\kappa$ lies in the interior of $\Nef(V)$ by the \textbf{(Ample$'$)} assumption, and $\Nef(V)$, hence $\NE(V)$, are finitely-generated, the $\tilde{\fm}$-adic filtration on $\tilde{R}(V)$ is equivalent to the filtration $F_{\ge \lambda} \tilde{R}(V) = \langle T^\beta\rangle_{\kappa(\beta) \ge \lambda}$; and the completion with respect to either is the ring of infinite series
$$R(V) := \Z\power{\NE(V)} = \left\{\sum_{\beta \in \NE(V)} n_\beta \cdot T^\beta: n_\beta \in \Z\right\}.$$
We denote the completion of $\tilde{\fm}$ by $\fm \subset R(V)$; it is an ideal, and the $\fm$-adic filtration on $R(V)$ is complete. 
We think of $R(V)$ as roughly analogous to the filtered ring of non-negative Novikov series used in \cite{FO3}. 
If $V$ is understood, we will drop it from the notation and write $R:=R(V)$. 

We define the \emph{symplectic relative Fukaya category} $\fuk(W \subset X,\omega,\theta,V,J_0)$ in Section \ref{sec:def}.
It is a $\Z/2$-graded curved filtered $R$-linear $A_\infty$ category (this notion is defined precisely in Section \ref{subsec:curvfilt}), whose objects are exact closed Lagrangian branes in $W$. 
It is independent of all auxiliary choices made in its construction, up to curved filtered quasi-equivalence (this notion is defined precisely in Section \ref{subsec:bc}). 
It is also invariant in the same sense under deformations of the tuple $(W,\omega,\theta)$:

\begin{prop}[cf. Proposition \ref{prop:indep2}]\label{prop:indep_def}
Suppose we have a smoothly-varying family of data $(W_t \subset X,\omega_t,\theta_t,V,J_0)_{t \in [0,1]}$ as above: in particular, $W_t$ remains disjoint from $V$, $J_0$ remains $\omega_t$-compatible, and there exists a smoothly-varying family of convex collars for $(W_t \subset X,J_0)$.
Then there is a curved filtered quasi-equivalence
$$\fuk(W_0 \subset X,\omega_0,\theta_0,V,J_0) \simeq \fuk(W_1 \subset X,\omega_1,\theta_1,V,J_0).$$
\end{prop} 

The following result concerns the dependence of the symplectic relative Fukaya category on the system of divisors $V$:

\begin{prop}[cf. Proposition \ref{prop:indep}]\label{prop:indepV0V1}
If $V_0 \subset V_1$ is a sub-system of divisors, then $\Nef(V_0) \subset \Nef(V_1)$, and there is a curved filtered quasi-equivalence 
\begin{equation}
\label{eq:V0V1}
\fuk(W \subset X,\omega,\theta,V_1,J_0) \otimes_{R(V_1)} R(V_0) \simeq \fuk(W \subset X,\omega,\theta,V_0,J_0).
\end{equation}
\end{prop}

We have not addressed the dependence of the symplectic relative Fukaya category on $J_0$, although we would expect it to be invariant under deformations of $J_0$.

\begin{rmk}
We have not mentioned the \textbf{(Semipositive$'$)} condition yet. 
It ensures that all closed holomorphic curves $u$ satisfy $c_1(u) \ge 0$, which is used to rule out sphere bubbling; this allows us to work over $\Z$, see Section \ref{subsec:whyZ}.
\end{rmk}

\subsection{Algebro-geometric relative Fukaya category}\label{sec:alg_geom}

Let $X$ be a smooth complex projective variety of dimension $n$, and $D \subset X$ a simple normal crossings (\snc) divisor, i.e., a normal crossings divisor whose irreducible components $\{D_p\}_{p \in P}$ are smooth. 
We consider $\mathrm{Div}(X,D) \simeq \Z^P$, the free abelian group of divisors supported on $D$. 
Let $\Nef \subset \mathrm{Div}(X,D)_\Q :=\mathrm{Div}(X,D) \otimes_\Z \Q$ be a rational polyhedral cone---so $\Nef = \{ \sum_{i=1}^s \lambda_i \delta_i  : \lambda_i \geq 0\}$ for a finite set of generators $\delta_i\in \mathrm{Div}(X,D)$---satisfying the following conditions:
\begin{enumerate}
\item [] \textbf{(Ample)} $\Nef$ contains an ample class in its interior; 

\item [] \textbf{(Semiample)} $\Nef$ is contained in the cone of effective semiample divisors supported on $D$;

\item []\textbf{(Semipositive)} $\Nef$ contains a `homological anticanonical divisor', i.e. one homologous to the anticanonical divisor modulo torsion. 
\end{enumerate}

We recall that a $\Q$-divisor $L$ is called semiample if there is a natural number $m$ for which $|m L|$ is a basepoint-free linear system, and that semiample implies nef.  
Declare $(X,D)$ to satisfy the \emph{non-negativity condition} if such a cone $\Nef$ exists.

\begin{rmk}
If $X$ is simply connected, the non-negativity condition implies that $-K_X$ is semiample. Hence $-K_X$ is nef; by \cite{CH}, $X$ is therefore a product $Y\times F$, where $Y$ has trivial canonical bundle and $F$ is rationally connected.
\end{rmk}

We define $\NE \subset \Z^P$ to be the monoid dual to $\Nef$. 
Because $\Nef$ is full-dimensional by the \textbf{(Ample)} condition, we have the ideal $\tilde{\fm} \subset \tilde{R}(\Nef) := \Z[\NE]$ defined as before; and because $\Nef$ is furthermore finitely generated, the $\tilde{\fm}$-adic completion of $\tilde{R}(\Nef)$ is $R(\Nef) = \Z\power{\NE}$ as before.\footnote {The \textbf{(Semiample)} assumption implies that $\Nef$ is strictly convex, and hence that $\NE$ is full-dimensional.}

\begin{defn}
Given $(X,D,\Nef)$ as above, we define \emph{auxiliary symplectic data} to consist of:
\begin{itemize}
\item a subdomain $W \subset X \setminus D$, which is a deformation retract of $X \setminus D$;
\item a K\"ahler form $\omega$ on $X$, and a primitive $\theta$ for $\omega|_W$ which makes $W$ into a Liouville domain, such that $[\omega;\theta] \in H^2(X,W;\R)$ lies in the interior of $\Nef$, under the identification 
$$H^2(X,W;\R) \simeq H^2(X,X \setminus D;\R) \simeq \mathrm{Div}(X,D)_\R.$$
(The first isomorphism exists because $W \subset X \setminus D$ is a deformation retract, the second identifies the class of a divisor $D_p$ in $\mathrm{Div}(X,D)$ with its Poincar\'e dual in $H^2(X,X \setminus D;\R)$.)
\end{itemize}
\end{defn}

We equip $X$ with its integrable complex structure $J_0$.

Using the \textbf{(Ample)} condition, we construct a non-empty path-connected space of auxiliary symplectic data in Section \ref{sec:aux_data}, which come equipped with convex collars. 

Using the \textbf{(Semiample)} assumption, there exists a system of divisors $V \subset X \setminus W$ such that $\Nef(V) = \Nef$ under the above identification, see Section \ref{sec:sod}. 
(The idea is to perturb normal-crossings divisors supported on $D$ to smooth ones, using Bertini's theorem.)

Thus we have the data $(W \subset X,\omega,\theta,V,J_0)$ needed to define the relative Fukaya category. 
The \textbf{(Ample$'$)} assumption is satisfied by construction, and the \textbf{(Semipositive$'$)} assumption is an immediate consequence of the \textbf{(Semipositive)} assumption.

We then define the algebro-geometric relative Fukaya category, 
$$\fuk(X,D,\Nef):=\fuk(W \subset X,\omega,\theta,V,J_0).$$
It is independent of all auxiliary choices made in its construction, up to curved filtered quasi-equivalence. 
In particular, it is independent of the choices of $\omega,\theta,W$, by the invariance of $\fuk(W \subset X,\omega,\theta,V,J_0)$ under deformations of these structures and the path-connectedness of the space of choices. 
To show that it is independent of the choice of the system of divisors $V$, we first observe that if $V$ and $V'$ are different systems of divisors, whose union $V \cup V'$ is transverse and hence defines a system of divisors, we have curved filtered quasi-equivalences
$$\fuk(W \subset X,\omega,\theta,V,J_0) \simeq \fuk(W \subset X,\omega,\theta,V \cup V',J_0) \simeq \fuk(W \subset X,\omega,\theta,V',J_0)$$
by \eqref{eq:V0V1}. 
To treat the general case, it suffices to observe that there exists a perturbation $V''$ of $V$ which is transverse to both $V$ and $V'$, by another application of Bertini's theorem.

More generally, if $\Nef_0 \subset \Nef_1$, a similar argument shows that there is a curved filtered quasi-equivalence 
\begin{equation}
\label{eq:N0N1}
\fuk(X,D,\Nef_1) \otimes_{R(\Nef_1)} R(\Nef_0) \simeq \fuk(X,D,\Nef_0).
\end{equation}

\begin{rmk}
We do not know of a purely symplectic criterion for the existence of a system of divisors analogous to the \textbf{(Semiample)} condition (which comes from algebraic geometry and uses Bertini's theorem). 
\end{rmk}

\begin{rmk}
As pointed out to us by Marco Castronovo, it should be possible to implement the above strategy to define `algebro-geometric' relative Fukaya categories in more general contexts, e.g. when $D$ is not simple normal crossings. One just needs to argue for the existence of auxiliary symplectic data and systems of divisors in the relevant context.
\end{rmk}

\subsection{Properties}

The symplectic relative Fukaya category $\fuk(W \subset X,\omega,\theta,V,J_0)$ is a deformation of the exact Fukaya category of the Liouville domain $W$, as defined in \cite{Seidel:FCPLT}. 
In order to state this result clearly, we briefly outline our conventions concerning the exact Fukaya category $\fuk(W,\theta)$.  
It is a $\Z/2$-graded, $\Z$-linear $A_\infty$ category. 
Its objects are closed, exact Lagrangian submanifolds $L \subset W$, equipped with Pin structures.  
The morphism spaces $hom^*(L_0,L_1)$ are free, finite-rank, $\Z/2$-graded $\Z$-modules generated by the intersection points betwen $L_0$ and a transverse Hamiltonian perturbation of $L_1$. 
The $A_\infty$ structure maps count certain pseudoholomorphic discs $u$ in $W$. 

\begin{prop}\label{prop:5.1}
There is a strict isomorphism 
\[ \fuk(W \subset X,\omega,\theta,V,J_0) \otimes_R R/\fm \simeq \fuk(W,\theta).\]
(This implies that $\fuk(X,D,\Nef)$ satisfies \cite[Assumption 5.1]{Sheridan2016}.)
\end{prop}

In particular, the objects of $\fuk(W \subset X,\omega,\theta,V,J_0)$ are identical to those of $\fuk(W,\theta)$. 
The morphism spaces of $\fuk(W \subset X,\omega,\theta,V,J_0)$ are obtained from those of $\fuk(W,\theta)$ by tensoring with $R$: in particular they are free, finite-rank, $\Z/2$-graded $R$-modules. 
The $A_\infty$ structure maps count certain pseudoholomorphic discs $u$ in $X$, weighted by $\nov^{[u]} \in R$. 
The content of Proposition \ref{prop:5.1} is essentially that all the pseudoholomorphic discs $u$, with $[u] = 0$ in $H_2(X,W)$, are contained in $W$ by the maximum principle (Lemma \ref{lem:maxprin}).

We remark that the algebro-geometric relative Fukaya category $\fuk(X,D,\Nef)$ also satisfies \cite[Assumptions 5.2 and 5.3]{Sheridan2016}, see Section \ref{sec:other_ass} for a brief discussion.

\subsection{Bounding cochains}

For any curved filtered $A_\infty$ category $\EuA$, we can form the associated uncurved $A_\infty$ category $\EuA^\bc$ of c-unital bounding cochains in $\EuA$;  and if two curved filtered $A_\infty$ categories are curved filtered quasi-equivalent, then their associated categories of bounding cochains are quasi-equivalent (see Theorem \ref{thm:cqiso}). 

Let $\BbK$ be a commutative ring,\footnote{All our rings, and homomorphisms thereof, are taken to be unital.} equipped with a complete, exhaustive, decreasing filtration, and a filtered ring homomorphism $R(\Nef) \to \BbK$. 
Then $\fuk(X,D,\Nef) \otimes_{R(\Nef)} \BbK$ is a curved filtered $\BbK$-linear $A_\infty$ category, which is independent of all choices made in its construction up to curved filtered quasi-equivalence. 
In particular, the corresponding category of c-unital bounding cochains is independent of choices up to quasi-equivalence. 

\begin{rmk}
We define an (uncurved) $A_\infty$ functor to be a quasi-equivalence if it induces an equivalence on the level of cohomology, and say that two $A_\infty$ categories are quasi-equivalent if they are connected by a zig-zag of quasi-equivalences (Definition \ref{defn:qeq}). 
This is a `good' equivalence relation if the morphism cochain complexes are cofibrant (see \cite[Section 3.1]{GPS2017}); this is true for example if the coefficient ring $\BbK$ is a field, but we have not checked it in general.  (Note that the argument of \cite[Remark 3.5]{GPS2017} does not apply here, as the action filtration is not a filtration by $R$-submodules.)
\end{rmk}

Using this construction, we can define a category which is comparable with the `usual' Fukaya category of the symplectic manifold $(X,\omega)$. 
Let $(X,D)$ be a pair such that there exists an effective homological anticanonical divisor supported on $D$, and $\kappa \in (\R_{>0})^P$ such that $\sum_p \kappa_p \cdot D_p$ is ample. 
There exists a rational polyhedral cone $\Nef$ satisfying the \textbf{(Ample)}, \textbf{(Semiample)}, and \textbf{(Semipositive)} conditions, and containing $\kappa$ in its interior (because $\kappa$ lies in the interior of the convex cone of effective semiample classes supported on $D$---using the fact that ampleness is an open condition, and ample implies semiample---and the homological anticanonical divisor lies in the same cone).

Let $\Bbbk$ be a commutative ring and $K \subset \R$ a submonoid, such that $\kappa(\NE) \subset K$. 
We define the corresponding Novikov-type ring $\Lambda_{\Bbbk,K}$ to be the completion of the group ring $\Bbbk[K]$ with respect to the filtration $F^a$, for any $a>0$, where $F^a_{\ge \lambda}$ is spanned by the monomials $\nov^{\lambda'}$ with $\lambda' \ge a \cdot \lambda$, for any $\lambda \in \Z_{\ge 0}$.
Then we have a filtered ring homomorphism 
\begin{align*}
R(\Nef) & \to \Lambda_{\Bbbk,K}, \qquad \text{sending}\\
\nov^\beta & \mapsto \nov^{\kappa(\beta)}.
\end{align*}
(We choose $a \le \min_{\beta \in \fm} \kappa(\beta)$, so that the homomorphism is filtered.)

\begin{defn}
We define 
$$
\fuk(X,D,\kappa;\Lambda_{\Bbbk,K}) := \fuk(X,D,\Nef) \otimes_{R(\Nef)} \Lambda_{\Bbbk,K}.
$$
\end{defn}

By our previous remark, $\fuk(X,D,\kappa;\Lambda_{\Bbbk,K})$ is independent of all auxiliary choices made in the definition of $\fuk(X,D,\Nef)$, up to curved filtered quasi-equivalence. It is also independent of $\Nef$ up to curved filtered quasi-equivalence, using equation \eqref{eq:N0N1}, and the fact that the relevant class of cones $\Nef$ is closed under intersection; that is why we omit $\Nef$ from the notation. 
It follows that $\fuk(X,D,\kappa;\Lambda_{\Bbbk,K})^\bc$ is independent of these choices up to quasi-equivalence. 
Note that this $A_\infty$ category is $G$-gapped, in the sense of \cite[Definition 3.2.26]{FO3}, where $G = \kappa(\NE) \subset K$ is discrete because $\NE$ is finitely generated.

\begin{conj}\label{conj:usual_fuk}
Suppose that $K \subset \R$ is a sub\emph{group}, and $(W \subset X,\omega,\theta)$ are auxiliary symplectic data. 
Then there should be an embedding 
$$\fuk(X,D,\kappa;\Lambda_{\Bbbk,K})^\bc \hookrightarrow \fuk(X,\omega;\Lambda_{\Bbbk,K})^\bc,$$
for any reasonable definition of the RHS (see, e.g., \cite{FO3,Fukaya2017} in the case $\Bbbk = \Q$ or $\R$, $K = \R$).
\end{conj}

\begin{rmk}
Of course, we can make the analogous construction and Conjecture in the context of the symplectic relative Fukaya category; but we have not established independence, e.g., of the system of divisors $V$ in that context.
\end{rmk}

\subsection{Why we can work over $\Z$}\label{subsec:whyZ}

In general, the Fukaya category can only be defined over a ring containing $\Q$, because the moduli spaces of holomorphic discs may be orbifolds. The orbifold points arise when the holomorphic disc admits a nontrivial finite symmetry group. A smooth holomorphic disc won't have such a symmetry group; however in general the definition of the Fukaya category involves virtual counts over moduli spaces of stable discs. A stable disc may admit such a symmetry group, arising from symmetries of a spherical component. In order to define the Fukaya category over $\Z$, one must rule out the appearance of such spherical components in the moduli spaces. This can be done when the symplectic manifold is exact \cite{Seidel:FCPLT}, or more generally, spherically positive  \cite{FO3:integers} (i.e., any $J$-holomorphic sphere $u$ with $c_1(u) \le 0$ is constant). 

In this paper we show that the relative Fukaya category can be defined over $\Z\power{\NE}$ when $(X,D)$ satisfies the non-negativity condition. This encompasses the important Calabi--Yau case ($c_1=0$), which is not spherically positive unless the only $J$-holomorphic spheres are constant. The \textbf{(Semipositive$'$)} condition ensures that all $J$-holomorphic spheres $u$ satisfy $c_1(u) \ge 0$. The technique of ruling out such spheres in this case goes back to \cite{Hofer1995}: all relevant moduli spaces of holomorphic discs have dimension $\le 1$. Bubbling off a sphere with $c_1>0$ subtracts at least two from the dimension of the disc moduli space, leaving it with negative dimension and hence empty; on the other hand spheres with $c_1=0$  sweep out a space of codimension $4$, and hence generically avoid all one-dimensional moduli spaces of discs as these sweep out a space of dimension $3$.

\begin{rmk}
A number of authors have studied Lagrangian Floer theory and homological mirror symmetry in finite characteristic in recent years, see e.g.  \cite{Perutz2012,Lekili2015,Lekili2017,Lekili2020}. 
It has not been possible to carry on this study for compact Calabi--Yau varieties of dimension $\ge 3$, because the available constructions of the Fukaya category relied on a characteristic-zero assumption (cf. \cite{FO3,Sheridan:CY,Fukaya2017,CharestWoodward,VWX}). Our construction in this paper removes this obstacle.
\end{rmk}

\begin{rmk}
Until recently, $\Z$-valued Gromov--Witten invariants were only defined under a semipositivity assumption analogous to ours here; the general definition was $\Q$-valued. Recently however, building on work of Abouzaid--McLean--Smith \cite{AMS21}, Bai--Xu \cite{BaiXu} have defined integer-valued genus-zero Gromov--Witten invariants of arbitrary compact symplectic manifolds.
\end{rmk}

\subsection{Applications}
  
Homological mirror symmetry was proved for `generalized Greene--Plesser mirrors' in \cite[Theorem D]{SS:GGP}, contingent on a definition of the Fukaya category satisfying certain properties (e.g., Abouzaid's split-generation criterion). 
This paper provides such a definition, namely $\fuk(X,D,\kappa;\Lambda_{\C,\R})^\bc$, and establishes some of the necessary properties; the remaining ones will be established in future work.\footnote{More precisely, because we only consider embedded Lagrangians in this paper, our construction is only adequate to the purposes of \cite[Theorem D]{SS:GGP} under the `embeddedness condition' [\emph{op. cit.}, Definition 1.3]. Furthermore, the application in \cite{SS:GGP} uses certain gradings on the relative Fukaya category (in particular, a $\Z$-grading on $\fuk(X,D,\kappa;\Lambda_{\C,\R})^\bc$), whereas here we have worked exclusively with $\Z/2$-gradings for simplicity of exposition; however it is routine to endow the relative Fukaya category as defined here with the necessary grading structures, see \cite{Sheridan2016}.}

The proof of homological mirror symmetry for generalized Greene--Plesser mirrors goes through almost verbatim over a coefficient ring $\Lambda_{\Bbbk,K}$ where $\Bbbk$ is a field of characteristic zero and $K \subset \R$ is a subgroup containing the image of $\kappa(\NE)$. The obstruction to extending the proof to the case of finite characteristic is that the deformation theory techniques used in \cite{Sheridan2016}, on which \cite{SS:GGP} depends, use the characteristic-zero assumption in a fundamental way; however this could possibly be circumvented by using a different approach to the deformation theory (for example, \cite[Lemma 3.9]{Seidel:HMSquartic} works for a field of arbitrary characteristic). 
The characteristic-zero assumption is also used in \cite{SS:GGP} in the verification of smoothness of the mirror, but this can presumably be remedied by modifying the `MPCP condition' depending on the characteristic.

It was shown in \cite{Ganatra2015} that homological mirror symmetry for Calabi--Yau varieties implied closed-string (Hodge-theoretic) mirror symmetry. 
The proof required a definition of the relative Fukaya category satisfying certain properties (e.g., the existence of a cyclic open--closed map respecting connections). 
Again, this paper provides the necessary definition, namely $\fuk(X,D,\kappa;\Lambda_{\C,\Z})^\bc$, but the relevant properties are relegated to future work.

\subsection{Relation to other work on relative Fukaya categories}
The idea of constructing a Fukaya category for a smooth projective variety relative to an ample divisor is due to Seidel \cite{Seidel:FukDef}, who writes ``The intended role of [the relative Fukaya category] $\EuF(M \subset X)$ is to interpolate between [the exact category] $\EuF(M)$ ... and the Fukaya category $\EuF(X)$ of the closed symplectic manifold''.  This idea was fully worked out for the case of quartic K3 surfaces in \cite{Seidel:HMSquartic}, which appeared in 2003. In that setting, sphere-bubbling  is essentially a non-issue.  The second-named author constructed higher-dimensional examples in \cite{Sheridan:CY, Sheridan:Fano}, using the geometry of the divisor to control sphere-bubbling. The arguments that exclude sphere bubbles from Gromov closures of low-dimensional  moduli spaces of polygons  depends on the fact that the normal crossings divisor is not merely ample, but has ample components. Relative Fukaya categories, typically with a smooth divisor, also figure in symplectic Picard--Lefschetz theory, as in \cite{Seidel:LefV}. 

Homological mirror symmetry for generalized Greene--Plesser mirrors \cite{SS:GGP} presents a situation where one has a natural choice of normal crossings divisor which supports an ample divisor, but which does not have ample components, and in which the coefficient rings are naturally multivariate.  The present paper allows for such situations. 

We have already mentioned that our use of power series with integer coefficients is one distinctive feature of our approach.  Another, in which we follow a suggestion of the referee, is that our construction is formulated under symplectic hypotheses, not only algebro-geometric ones. A third is the use of auxiliary `systems of divisors'. 

A slightly different use of stabilizing divisors is to construct symplectic invariants by means of Donaldson hypersurfaces. Cieliebak--Monhke \cite{Cieliebak2007}, in work which we draw on for a number of technical points, showed that high-degree Donaldson hypersurfaces could be used to set up Gromov--Witten invariants in a rational symplectic manifold.  Charest--Woodward \cite{CW:floer,CharestWoodward} use a similar mechanism to define the $A_\infty$-algebra of Floer cochains $CF(L,L)$ for a rational Lagrangian $L$.  Their construction works for a large class of symplectic manifolds and a large class of Lagrangians therein---considered one at a time---but does not appear to allow for situations involving the geometry of specific divisors such as those that appear in HMS for generalized Greene--Plesser mirror pairs (\emph{op. cit.}).

\subsection{Outline}

Section \ref{sec:bc} defines the relevant notion of `curved filtered quasi-equivalence' of curved filtered $A_\infty$ categories, gives the construction of the uncurved $A_\infty$ category $\EuA^\bc$ of c-unital bounding cochains on a curved filtered $A_\infty$ category $\EuA$, and proves that if $\EuA$ and $\EuB$ are curved filtered quasi-equivalent then $\EuA^\bc$ and $\EuB^\bc$ are quasi-equivalent. 
Section \ref{sec:targ} introduces the class of almost-complex structures we will work with (namely, those $J$ which preserve a chosen system of divisors, and are sufficiently close to $J_0$). Section \ref{sec:spheres} establishes basic results about the moduli spaces of stable $J$-holomorphic spheres, which will allow us to rule out their bubbling off a moduli space of pseudoholomorphic discs by the argument sketched in Section \ref{subsec:whyZ} above. Section \ref{sec:discs} introduces the relevant moduli spaces of pseudoholomorphic discs. Section \ref{sec:def} gives the definition of the symplectic relative Fukaya category, and proves Proposition \ref{prop:5.1}. Section \ref{sec:indep} proves Proposition \ref{prop:indepV0V1} (independence of the symplectic relative Fukaya category of the auxiliary choices made in its construction, and system of divisors).  
Section \ref{sec:dep2} proves Proposition \ref{prop:indep_def} (independence of deformation of auxiliary symplectic data). 
Section \ref{sec:a_geom} gives the definition of the algebro-geometric relative Fukaya category, and proves its basic properties.

\textbf{Acknowledgments:} T.P. is partially supported by NSF CAREER grant 1455265. 
N.S. is supported by a Royal Society University Research Fellowship, ERC Starting Grant 850713 -- HMS, the Simons Collaboration on Homological Mirror Symmetry, the Leverhulme Prize, and a Simons Investigator award. 
N.S. is grateful to Paul Seidel for very helpful conversations, including suggesting the crucial notion of a system of divisors.  We thank the referee for a close reading, and for encouraging us to formulate our construction in purely symplectic terms.

\section{Curved filtered $A_\infty$ categories} \label{sec:bc}

\subsection{Definition}\label{subsec:curvfilt}

All of our rings will be unital. In this section we will work over $\BbK$, which will be a commutative ring equipped with an exhaustive, complete, decreasing filtration $F_{\ge \bullet} \BbK$ indexed by the integers. 

One key example of such a ring $\BbK$ will be $\Z\power{\NE}$, equipped with the $\fm$-adic filtration. The other key class of examples will be the Novikov-type rings $\Lambda_{\Bbbk,K}$ defined in the Introduction.

We will regularly pass from $\BbK$ to its associated graded ring $\mathrm{Gr}_* \BbK$; for instance, from $\Z\power{\NE}$ to $\mathrm{Gr}_* \Z\power{\NE} \cong \Z [\NE]$.

\begin{defn}
A \emph{curved filtered $\BbK$-linear $A_\infty$ algebra} is a $\Z/2$-graded $\BbK$-module $A$, equipped with an exhaustive, complete, decreasing filtration $F_{\ge \bullet}$, compatible with the module structure in the sense that
$$ F_{\ge \lambda_1} \BbK \cdot F_{\ge \lambda_2} A \subset F_{\ge \lambda_1 + \lambda_2} A.$$
This comes with $\BbK$-linear $A_\infty$ operations 
$$\mu^k: A^{\otimes_\BbK k} \to A$$
for all $k \ge 0$, of degree $2-k$. 
These are required to respect the filtration in the sense that
$$ \mu^k(F_{\ge \lambda_1} A,\ldots,F_{\ge \lambda_k}A) \subset F_{\ge \lambda_1 + \ldots + \lambda_k} A,$$
furthermore have $\mu^0 \in F_{\ge 1} A$, and satisfy the curved $A_\infty$ equations. 
\end{defn}

Associated to a curved filtered $\BbK$-linear $A_\infty$ algebra $A$, there is an uncurved $\Z/2 \oplus \Z$-graded $\mathrm{Gr}_* \BbK$-linear $A_\infty$ algebra $\mathrm{Gr}_*A$ whose underlying $\mathrm{Gr}_* \BbK$-module is the associated graded $\mathrm{Gr}_* A$, with the structure maps induced by those for $A$. 
The structure map $\mu^k$ has degree $(2-k,0)$. 
The notion of c(ohomological)-unitality \cite[Section 2a]{Seidel:FCPLT} makes sense for $\mathrm{Gr}_*A$, as it is uncurved. 
If a c-unit exists, it necessarily has degree $(0,0)$. 

\begin{defn}
We say that $A$ is c-unital if $\mathrm{Gr}_*A$ is.
\end{defn}

We similarly define the notion of a curved filtered $A_\infty$ category, and of c-unitality of such a category. Henceforth we will implicitly assume that our curved filtered $A_\infty$ algebras and categories are c-unital.

\begin{example}
The relative Fukaya category $\fuk(X,D,\Nef)$, equipped with the $\fm$-adic filtration, is a curved filtered $\Z\power{\NE}$-linear $A_\infty$ category. 
The fact that it is c-unital follows from the fact that $\mathrm{Gr}_* \fuk(X,D,\Nef) \simeq \fuk(X \setminus D;\Z) \otimes_\Z \mathrm{Gr}_* \Z\power{\NE}$, and the argument for c-unitality of $\fuk(X \setminus D)$ (see, e.g., \cite[Section 2.4]{Sheridan:Fano}) works with coefficients in an arbitrary commutative ring.
\end{example}

\begin{example}
Let $\Z\power{\NE} \to \BbK$ be a filtered ring homomorphism; then $\fuk(X,D,\Nef) \otimes_{\Z\power{\NE}} \BbK$ is naturally a curved filtered $\BbK$-linear $A_\infty$ category. The argument for c-unitality is the same as before.
\end{example}

\subsection{Bounding cochains and c-unitality}\label{subsec:bc}

\begin{defn}
Let $A$ be a curved filtered $A_\infty$ algebra; we say that $a \in A$ is a \emph{Maurer--Cartan element} if $a \in F_{\ge 1} A$ has degree $1$ and satisfies the Maurer--Cartan equation:
\[ \sum_{k\ge 0} \mu^k(a,\ldots,a) = 0.\]
(Note that the infinite sum converges by completeness.)
\end{defn}

Given a Maurer--Cartan element $a \in A$, we can use it to deform the $A_\infty$ algebra $A$ to an uncurved, not-necessarily-c-unital $A_\infty$ algebra $A^a = (A,\mu^*_a)$, in a standard way (see \cite[Section 2.7]{Sheridan2016}). 

\begin{defn}
We say that a Maurer--Cartan element $a \in A$ is c-unital if $A^a$ is c-unital and furthermore, there exists a cocycle $e^a \in F_{\ge 0}A^a$ representing the c-unit, whose class in $\mathrm{Gr}_0A$ represents the c-unit of $\mathrm{Gr}_*A$.
\end{defn}

\begin{defn}\label{def:cunbc}
Let $\EuA$ be a curved filtered $A_\infty$ category. 
We define an uncurved $A_\infty$ category $\EuA^\bc$ whose objects are pairs $(A,a)$, where $a\in hom^1(A,A)$ is a c-unital Maurer--Cartan element.
The morphism spaces are 
\[ hom_{\EuA^\bc} ((A,a),(B,b)) := hom_\EuA(A,B).\]
The formula for the $A_\infty$ structure maps, which get deformed by the Maurer--Cartan elements, can be found in \cite[Equation (3.20)]{Seidel:FCPLT}; it converges by completeness.
\end{defn}

Since our convention is that all $A_\infty$ categories should be c-unital, the following Lemma is required to show that Definition \ref{def:cunbc} is valid:

\begin{lem}
The $A_\infty$ category $\EuA^\bc$ is c-unital.
 \end{lem}
 \begin{proof}
By definition, each object $(A,a)$ comes equipped with $e^a \in F_{\ge 0}hom_{\EuA^\bc}((A,a),(A,a))$ which represents a c-unit for that algebra, and whose class in $\mathrm{Gr}_0 hom_{\EuA}(A,A)$ represents a c-unit for $hom_{\mathrm{Gr}_*\EuA}(A,A)$. 
It suffices to show that the maps
\[ hom_{H^*(\EuA^\bc)}((A,a),(B,b)) \xrightarrow{[e^a] \bullet (-)} hom_{H^*(\EuA^\bc)}((A,a),(B,b)) \xleftarrow{(-)\bullet [e^b]}  hom_{H^*(\EuA^\bc)}((A,a),(B,b))
\]
are equal to the identity.
 
We will focus on the leftmost map, as the argument for the rightmost one is identical. 
We start by showing that the map is an isomorphism. 
This follows from a spectral sequence comparison argument, for the spectral sequence induced by the filtration. 
Indeed, the map induced on the $E_1$ page is
\[ hom_{H^*(\mathrm{Gr}_*\EuA)}(A,B) \xrightarrow{[e^a] \bullet (-)}  hom_{H^*(\mathrm{Gr}_*\EuA)}(A,B).\] 
This map is an isomorphism, in fact the identity, because $[e^a]$ represents the c-unit in $\mathrm{Gr}_*\EuA$. 
Since the filtration is exhaustive and complete, the Eilenberg--Moore comparison theorem \cite[Theorem 5.5.11]{Weibel1994} implies that the original map $[e^a]\bullet(-)$ is an isomorphism.

Now observe that $[e^a]\bullet[e^a] = [e^a]$, because $e^a$ is a c-unit for $H^*(hom_{\EuA^\bc}((A,a),(A,a)))$ by definition. 
For any $f \in  hom_{H^*(\EuA^\bc)}((A,a),(B,b))$, we have
\[
[e^a] \bullet \left([e^a] \bullet f - f\right) = [e^a] \bullet f - [e^a]\bullet f = 0,\]
and therefore $[e^a] \bullet f - f = 0$, because $[e^a]\bullet(-)$ is an isomorphism by our previous argument, and in particular injective. 
This completes the proof.
 \end{proof}

\subsection{Filtered curved quasi-equivalence}

\begin{defn}
A curved filtered $A_\infty$ functor between two curved filtered $A_\infty$ categories is a curved $A_\infty$ functor $G$, respecting the filtration in the obvious way, having curvature $G^0 \in F_{\ge 1}$, such that the uncurved $A_\infty$ functor $\mathrm{Gr}_*G$ is c-unital. 
(Note that the curved $A_\infty$ functor equations converge by completeness.)
\end{defn}

\begin{defn}\label{defn:qeq}
We say that an uncurved $\Z/2 \oplus \Z$-graded $A_\infty$ functor is a quasi-equivalence if it induces an equivalence on the level of $\Z/2 \oplus \Z$-graded cohomological categories.\footnote{Note that the grading plays a role in the notion of isomorphism of objects, and hence in the notion of essential surjectivity. Specifically, a quasi-isomorphism is required to have degree $(0,0)$.}   
We say that two uncurved $\Z/2\oplus \Z$-graded $A_\infty$ categories are quasi-equivalent if they are connected by a zig-zag of quasi-equivalences.
\end{defn}

\begin{defn}
A curved filtered $A_\infty$ functor $G$ is called a curved filtered quasi-equivalence if $\mathrm{Gr}_*G$ is a $\Z/2\oplus \Z$-graded quasi-equivalence. 
We say that two curved filtered $A_\infty$ categories are curved filtered quasi-equivalent if they are connected by a zigzag of curved filtered quasi-equivalences.
\end{defn}

A curved filtered $A_\infty$ functor $G:\EuA \to \EuB$ induces a not-necessarily-c-unital uncurved $A_\infty$ functor $G^\bc$, which maps from $\EuA^\bc$ to the not-necessarily-c-unital category of bounding cochains on $\EuB$, cf. \cite[Section 2.7]{Sheridan2016}. 

\begin{thm}\label{thm:cqiso}
If $G:\EuA \to \EuB$ is a curved filtered quasi-equivalence, then $G^\bc$ defines an uncurved $A_\infty$ functor $\EuA^\bc \to \EuB^\bc$ (in particular, the image of a c-unital bounding cochain under $G^\bc$ is c-unital). 
Furthermore, this functor is a quasi-equivalence.
\end{thm}

The proof of Theorem \ref{thm:cqiso} will be given in Section \ref{subsec:cqiso}, but we record here an immediate corollary:

\begin{cor}
If $\EuA$ and $\EuB$ are curved filtered quasi-equivalent, then $\EuA^\bc$ and $\EuB^\bc$ are quasi-equivalent.
\end{cor}

\subsection{The triangle algebra}

In this section we introduce a well-known construction (see \cite[Section 4.5]{Keller:DIH}) which will be used in the proof of Theorem \ref{thm:cqiso}. 

\begin{defn}\label{defn:tm} 
Given $A_\infty$ algebras $A$, $B$, and an $A$-$B$ bimodule $M$, we can form a new $A_\infty$ algebra 
\[ \mathcal{T}(M) := A \oplus M[-1] \oplus B\]
by defining
\[ \mu^k(c_1,\ldots,c_k) := \left\{\begin{array}{ll}
									\mu^k_A(c_1,\ldots,c_k) & \text{if all $c_i \in A$;} \\
									\mu^k_B(c_1,\ldots,c_k) & \text{if all $c_i \in B$;} \\
									\mu^{i|1|j}_M(c_1,\ldots,c_{i-1};c_i;c_{i+1},\ldots,c_k) & \text{whenever the expression makes sense;} \\
									0 & \text{otherwise}.
									\end{array}\right.\]
We can think of elements of $\mathcal{T}(M)$ as arranged in a lower-triangular two-by-two matrix:
\[ \mathcal{T}(M) = \left(\begin{array}{cc}
						A & 0 \\ M & B \end{array} \right).\]
\end{defn}

\begin{defn}\label{defn:tmf}
Let $M$ be an $A$-$B$ bimodule, and $f \in M[-1] \subset \mathcal{T}(M)$ have degree $1$. 
Observe that the Maurer--Cartan equation for $f$ reduces to $\mu^{0|1|0}(f)=0$, i.e., the equation for $f$ to be closed. 
Thus we can form the $A_\infty$ algebra $\mathcal{T}(M)^f$, by deforming $\mathcal{T}(M)$ with the Maurer--Cartan element $f$. 
It is clear that there is no issue with convergence, as $f$ can appear at most once in any $A_\infty$ product. 
It is clear that the projections to $A$ and $B$ define strict $A_\infty$ homomorphisms 
\begin{equation}
\label{eq:projAB} A \xleftarrow{\pi_A} \mathcal{T}(M)^f \xrightarrow{\pi_B} B.
\end{equation}
\end{defn}

\begin{defn}
If $M$ is an $A$-$B$ bimodule, we say that $f \in M$ is a \emph{quasi-isomorphism} if $\mu^{0|1|0}(f)=0$, and the chain maps
\[ A \xrightarrow{\mu^{1|1|0}(-;f;)} M \xleftarrow{\mu^{0|1|1}(;f;-)} B \]
are quasi-isomorphisms of cochain complexes. 
\end{defn}

\begin{lem}\label{lem:projqiso}
If $f \in M$ is a quasi-isomorphism, then the projections \eqref{eq:projAB} both induce isomorphisms on cohomology; in particular, $A$ is quasi-isomorphic to $B$.
\end{lem}
\begin{proof}
The kernel of $\pi_B$ is the cochain complex
\[ Cone \left(A \xrightarrow{\mu^{1|1|0}(-;f;)} M\right)\]
which is acyclic, because $f$ is a quasi-isomorphism. Similarly for $\pi_A$.
\end{proof}
		
\begin{cor}
Suppose that $X$ and $Y$ are quasi-isomorphic objects in an $A_\infty$ category (i.e., they become isomorphic objects in the cohomological category). 
Then the $A_\infty$ algebras $hom(X,X)$ and $hom(Y,Y)$ are quasi-isomorphic.
\end{cor}
\begin{proof}
This follows from Lemma \ref{lem:projqiso}, where we take $A = hom(X,X)$, $B=hom(Y,Y)$, $M=hom(X,Y)$, and $f \in M$ a cochain-level representative of the isomorphism between $X$ and $Y$ in the cohomological category.
\end{proof}

\subsection{Proof of Theorem \ref{thm:cqiso}}
\label{subsec:cqiso}

Let $A$ and $B$ be curved filtered $A_\infty$ algebras and $M$ a filtered $A$-$B$ bimodule. 
If $a \in A$ and $b \in B$ are solutions of the respective Maurer--Cartan equations, we denote by $A^a$ and $B^b$ the corresponding uncurved $A_\infty$ algebras; we can also define an $A^a$-$B^b$ bimodule $^a M ^b$ in a similar way. 

Note that $\mathrm{Gr}_*M$ is a $\Z/2\oplus\Z$-graded $\mathrm{Gr}_* A$-$\mathrm{Gr}_*B$ bimodule. 

\begin{lem}\label{lem:mcdef}
Suppose that $m_0 \in \mathrm{Gr}_0 M$ is a quasi-isomorphism, and $b \in B$ a solution of the Maurer--Cartan equation. 
Then there exists $a \in A$ a solution of the Maurer--Cartan equation, and a closed element $m \in F_{\ge 0}{}^a M^b$ representing $m_0$. 
If $b$ is c-unital, then so is $a$.
\end{lem}
\begin{proof}
Observe that the conditions required of $a$ and $m$ are equivalent to 
\[t := \left(\begin{array}{cc} a & 0 \\ m &b\end{array}\right) \in \left( \begin{array}{cc} F_{\ge 1} A & 0 \\ F_{\ge 0}M & F_{\ge 1} B \end{array} \right) \subset \mathcal{T}(M)\]
 satisfying the Maurer--Cartan equation. 
This is not quite the same as being a Maurer--Cartan element for $\mathcal{T}(M)$, as $m \notin F_{\ge 1} M$. 
Nevertheless, the Maurer--Cartan equation still converges, because $m$ is lower-triangular. 

We solve for $a$ and $m$, inductively in the filtration. 
Suppose we have solved to order $k-1$; i.e., the Maurer--Cartan equation for $t$ holds modulo $F_{\ge k}$. 
Let 
\[ \mathbb{T}_k = \left(\begin{array}{cc} \mathbb{A}_k & 0 \\ \mathbb{M}_k & \mathbb{B}_k \end{array}\right) := \sum_s \mu^s(t,\ldots,t) \in \mathrm{Gr}_k \mathcal{T}(M).\]
The fact that $b$ is a bounding cochain implies that $\mathbb{B}_k = 0$. 
The $A_\infty$ equations 
\[ \sum \mu^*(t,\ldots,t,\mu^*(t,\ldots,t),t,\ldots,t)=0\]
imply that
\[ \left(\begin{array}{cc}\mu^1\left(\mathbb{A}_k\right) & 0 \\ \mu^{0|1|0}\left(\mathbb{M}_k\right) + \mu^2_0(m_0,\mathbb{A}_k) & 0 \end{array} \right) = 0,\]
which is equivalent to
\[ \mu^1_{\mathrm{Gr}_*\mathcal{T}(M)^{m_0}}(\mathbb{T}_k) = 0.\]
Similarly, adding $a_k \in F_{\ge k} A$ to $a$ and $m_k \in F_{\ge k} M$ to $m$ changes $\mathbb{T}_k$ by
\[ \mathbb{T}_k \mapsto \mathbb{T}_k + \mu^1_{\mathrm{Gr}_k\mathcal{T}(M)^{m_0}} \left(\begin{array}{cc} a_k & 0 \\ m_k & 0 \end{array}\right).\]

We now observe that because $\mathbb{B}_k  = 0$, $\mathbb{T}_k$ lies in $\ker(\pi_{\mathrm{Gr}_k B})$. 
This complex is acyclic, because $m_0$ is a quasi-isomorphism; we have shown that $\mathbb{T}_k$ is closed, and that we are free to modify it by adding exact terms. 
Therefore we can choose $a_k$ and $m_k$ so that $\mathbb{T}_k$ vanishes. 
The inductive construction of $a$ and $m$ converges, by completeness.

It remains to check that if $b$ is c-unital, then so is $a$. For this we consider the strict $A_\infty$ homomorphisms
$$ A^a \xleftarrow{\pi_{A^a}} \mathcal{T}(M)^m \xrightarrow{\pi_{B^b}} B^b.$$
These are filtered quasi-isomorphisms, by a spectral sequence comparison argument.  Therefore the unit $e^b \in F_{\ge 0} B^b$ admits a lift to a closed element in $F_{\ge 0}\mathcal{T}(M)^m$, whose projection to $F_{\ge 0}A^a$ we denote by $e^a$. This represents a c-unit for $A^a$, using the fact that $\pi_{A^a}$ and $\pi_{B^b}$ induce algebra isomorphisms on the level of cohomology.
\end{proof}

\begin{proof}[Proof of Theorem \ref{thm:cqiso}]
Because $\mathrm{Gr}_*G$ is fully faithful, so is $G^\bc$, by a comparison argument for the spectral sequence induced by the filtration. 
It follows immediately that $G^\bc$ is c-unital. 
It remains to check that $G^\bc$ is essentially surjective.
 
In other words, given $(Y,y) \in \ob \EuB^\bc$, we must find $(X,x) \in \ob \EuA^\bc$ such that $G^\bc(X,x) \simeq (Y,y)$. 
Because $\mathrm{Gr}_* G$ is a $\Z/2 \oplus \Z$-graded quasi-equivalence, and in particular essentially surjective, there exists $X \in \ob \EuA$ such that $GX$ is quasi-isomorphic to $Y$ in $\mathrm{Gr}_*\EuB$. 
Let $f_0 \in \mathrm{Gr}_0 hom^*_{\EuB}(X,Y)$ be a quasi-isomorphism. 

We now apply Lemma \ref{lem:mcdef}, with $A=hom_{\EuA}(X,X)$, $B=hom_\EuB(Y,Y)$, $M=(G \otimes \id)^*(hom_\EuB(X,Y))$, $b=y$, $m_0=f_0$. 
The outcome is a c-unital Maurer--Cartan element $a \in A$, together with a closed element $m \in F_{\ge 0}{}^aM^b$ representing $m_0$. 
Setting $x=a$, and $f=m \in F_{\ge 0} hom_{\EuB^\bc}(G^\bc(X,x),(Y,y))$, we find that $f$ is closed and represents $f_0$. 

We now consider the maps
\begin{align*} 
hom((Y,y),G^\bc(X,x)) &\xrightarrow{(-) \bullet [f]} hom((Y,y),(Y,y)) \qquad \text{and}\\ 
hom((Y,y),G^\bc(X,x)) &\xrightarrow{[f] \bullet (-)} hom(G^\bc(X,x),G^\bc(X,x)),
\end{align*}
where all morphism spaces are taken in the cohomological category $H^*(\EuB^\bc)$. 
Using the fact that $f_0$ is a quasi-isomorphism, a comparison argument for the spectral sequence induced by the filtration shows that these maps are isomorphisms. 
In particular we can solve $[g] \bullet [f]=[e^y]$ and $[f] \bullet [h]=[G^\bc(e^x)]$, from which one easily deduces that $[h]=[g]$, so that $f$ defines a quasi-isomorphism between $G^\bc(X,x)$ and $(Y,y)$.
\end{proof}

\section{Target geometry}\label{sec:targ}

Let $(X,\omega)$ be a compact symplectic manifold. 
We denote the space of $\omega$-tame almost complex structures on $X$ by $\EuJ(X)$. 

\subsection{Maximum principle}

Let $X$ be a symplectic manifold, $J \in \EuJ(X)$, and $W \subset X$ a Liouville subdomain. 

\begin{defn}\label{def:conv_coll}
A \emph{convex collar} for $(W \subset X,J)$ is a neighbourhood $U$ of $\partial W$ in $X \setminus \mathring{W}$, together with a fibration $h:U \to [c,d)$ such that $h^{-1}(c) = \partial W$, and $\theta_U := -dh \circ J$ is a primitive for $\omega|_U$ which patches together smoothly with the primitive $\theta$ for $\omega|_W$ to endow $W \cup U$ with the structure of a Liouville subdomain.
\end{defn}

The name `convex collar' is justified by the following lemma:

\begin{lem}[cf. Lemma 7.2 of \cite{Abouzaid2007}]
\label{lem:maxprin}
Suppose that $(W \subset X,J)$ admits a convex collar, and 
\[ u:(\Sigma,\partial \Sigma) \to (X,W) \]
is a $J$-holomorphic curve such that $[u] = 0$ in $H_2(X,W)$. 
Then $u$ is contained in $W$.
\end{lem}
\begin{proof}
We choose a convex collar $h:U \to [c,d)$ for $(W \subset X,J)$. 
Find a sequence $(c_k)$ of numbers in $[c,d)$, converging to $c$, such that for each $k$, $u$ intersects $h^{-1}(c_k)$ transversely. 
Set $W_k = W \cup h^{-1}([c,c_k])$, and let $u_k:(\Sigma_k,\partial \Sigma_k) \to (X \setminus \mathring{W}_k,\partial W_k)$ be the part of $u$ mapping `outside' $\mathring{W}_k$.

Observe that $\theta(\partial u_k) \le 0$ because if $\xi$ is a positively-oriented tangent vector of $\partial \Sigma_k$, then $j(\xi)$ does not point out of $\Sigma_k$, hence $dh \circ du \circ j(\xi) \ge 0$ (as $dh$ evaluates non-negatively on tangent vectors which do not point `into' $W_k$); so
$$\theta(du(\xi)) = -dh \circ J \circ du(\xi) = -dh \circ du \circ j(\xi) \le 0.$$
On the other hand, we have $[u] = 0$ in $H_2(X,W)$, by assumption; hence, $[u_k] = 0$ and we have $[\omega;\theta](u_k) = 0$. 
It follows by definition that
\[ \omega(u_k) = \theta(\partial u_k) \le 0.\]
But now, $\omega(u_k) \ge 0$ because $J$ is $\omega$-tame, so we must have $\omega(u_k) = 0$, hence $u_k$ is constant, so $u_k$ is contained in $W_k$ for all $k$; hence $u$ is contained in $W$.
\end{proof}

\subsection{Positivity of intersection}

Let $V$ be a system of divisors for $W \subset X$.

\begin{defn}[cf. Definition 4.1 of \cite{Sheridan2016}]
We denote the space of $\omega$-tame almost complex structures which make each $V_q$ into an almost complex submanifold by $\EuJ(X,V) \subset \EuJ(X)$.
\end{defn}

\begin{lem}[Lemma 4.2 of \cite{Sheridan2016}]\label{lem:posint}
If $J \in \EuJ(X,V)$, then any $J$-holomorphic curve $u:(\Sigma,\partial \Sigma) \to (X,X \setminus D)$ satisfies $[u] \cdot V_q  \ge 0$ for all $q \in Q$. 
In particular, $[u] \in \NE(V)$.
\end{lem}
\begin{proof}
If $u$ is not contained inside $V_q$, then $u \cdot V_q \ge 0$ by positivity of intersection (cf. \cite[Exercise 2.6.1]{mcduffsalamon} or \cite[Prop. 7.1]{Cieliebak2007}). 
There are $\ge n+1$ divisors $V_r$ with $[V_r] = [V_q]$, and $u$ can't be contained in all of them because their common intersection is empty by transversality.
\end{proof}

We recall that $J \in \EuJ(X)$ is called \emph{semipositive} if every $J$-holomorphic sphere has nonnegative first Chern number \cite[Definition 6.4.5]{mcduffsalamon}. 

\begin{cor}
\label{cor:semipos}
Any $J \in \EuJ(X,V)$ is semipositive.
\end{cor}
\begin{proof}
Let $u: \CP^1 \to X$ be a $J$-holomorphic sphere. 
By the \textbf{(Semipositive$'$)} condition on $V$, we have
\[ c_1(TX)(u) = \tilde{c}_1(u) \ge 0,\]
because $\tilde{c}_1 \in \Nef(V)$ by assumption and $[u] \in \NE(V)$ by Lemma \ref{lem:posint}.
\end{proof}
 
We introduce notation, for any subset $K \subset Q$: 
\begin{equation} \label{eqn:VK}
V_K := \bigcap_{q \in K} V_q, \qquad V^K := \bigcup_{q \notin K} V_q.
\end{equation}
(We set  $V_\emptyset := X$.) 

\begin{lem}[Adjunction formula]\label{lem:adj}
If $u \in H_2(V_K)$, then
\begin{equation}
\label{eq:adj}
 c_1(TV_K)(u) = c_1(TX)(u) - \sum_{q \in K}u \cdot V_q.
 \end{equation}
\end{lem}
\begin{proof}
Follows from the adjunction formula, together with the isomorphism
\[ NV_K \simeq \bigoplus_{q \in K} NV_q|_{V_K}\]
which holds by transversality. 
\end{proof}
 
The following lemma will be used to show that holomorphic curves contained inside $V$ do not have `excess dimension':

\begin{lem}[Lemma 4.5 of \cite{Sheridan2016}]
\label{lem:c1loss}
If $J \in \EuJ(X,V)$, then any $J$-holomorphic curve $u$ contained inside $V_K$ satisfies $ c_1(TV_K)(u) \le c_1(TX)(u)$.
\end{lem}
\begin{proof}
Follows from Lemmas \ref{lem:adj} and \ref{lem:posint}.
 \end{proof}

\begin{lem}[Lemma 4.4 of \cite{Sheridan2016}]\label{lem:intersectv}
Suppose that $J \in \EuJ(X,V)$, $W \subset X$ is a Liouville subdomain such that $(W \subset X,J)$ admits a convex collar, and $u:(\Sigma,\partial \Sigma) \to (X,L)$ is a non-constant $J$-holomorphic curve with boundary on a closed, exact Lagrangian submanifold $L \subset \mathring{W}$.
Then $u$ intersects $V$, and in particular $u \cdot V_q > 0$ for some $q$ by positivity of intersection.
\end{lem}
\begin{proof}
We prove the contrapositive: suppose that $u$ does not intersect $V$, so $u \cdot V_q = 0$ for all $q$. 
Because the classes $PD([V_q])$ span $H^2(X,W)$ (by definition of a system of divisors), we have $[u] = 0$, and hence $u \subset W$ by Lemma \ref{lem:maxprin}. 
Because $L$ is exact, this implies that $u$ has zero energy and hence is constant.
\end{proof}

\subsection{Space of almost complex structures}

Let $J_0 \in \EuJ(X,V)$ be an $\omega$-compatible almost complex structure. 
We introduce a class of perturbations of $J_0$, which will be sufficiently large to achieve transversality for all moduli spaces of holomorphic curves we consider. 
Set
\begin{align} \label{eqn:perts}
 \EuY= \EuY(X,V) &:= \left\{   Y\in C^\infty \left( \End TX \right) :     YJ_0 + J_0 Y = 0;\, Y(TV_q) \subset TV_q \text{ for all $q \in Q$}\right\}, \\
\EuY_*=  \EuY_*(X,V) &:= \left\{Y \in \EuY: \| Y\|_{C^0} < \log \frac{3}{2}\right\}. 
 \end{align}
Here $ \| Y\|_{C^0}$ means $\sup_{x \in X} \| Y_x\|$, where $\| Y_x\|$ is the operator norm, taken with respect to the almost K\"ahler metric on $TX$. We will usually fix $V$ and drop it from the notation.
Note that $\EuY_*$ is a convex (hence contractible) subset of the vector space $\EuY$, and is open with respect to the $C^\infty$ topology.

Any $Y \in \EuY_*$ determines an almost complex structure $J_Y := J_0 \exp(Y)$, where $\exp (Y)$ denotes the pointwise exponential of an endomorphism. The two conditions in the definition of $\EuY$ imply respectively that $J_Y$ is an almost complex structure on $X$; and that it makes each component $V_q$ into an almost complex submanifold. 
The bound on $\|Y\|_{C^0}$ implies that $J_Y$ tames $\omega$ (to see this, write $e^Y = I + (e^Y-I)$ and observe that $\| e^Y-I\| \leq e^{\|Y\|}-1 \leq 1/2$). 
Therefore $J_Y \in \EuJ(X,V)$.

For positive integers $\ell$, we write $ \EuY^\ell$ and $\EuY_*^\ell$ for the counterparts of these spaces in which $Y$ is $C^\ell$ rather than $C^\infty$.

\subsection{Extending almost complex structures}

Observe that $V^K\cap V_K$ (notation from (\ref{eqn:VK})), the points of $V_K$ which lie in $V_q$ for some $q \notin K$, is a \snc{} divisor in $V_K$. 
We define $\EuY_{K}  = \EuY(V_K, V^K\cap V_K)$, and observe that there are restriction maps $r_K\colon \EuY \to \EuY_K$; similarly define $\EuY_K^\ell$ and the restriction map $r_K\colon \EuY^\ell\to \EuY_K^\ell$.

\begin{lem}
\label{lem:extendJlocally}
Take $Y\in \EuY_{*,K}^\ell$ (where $1\leq \ell \leq \infty$), and $x \in V_K$. Then there exists $\widetilde{Y} \in \EuY_*^{\ell}$ such that $r_K(\widetilde{Y}) = Y$ on a neighborhood of $x$ in $V_K$.
 \end{lem}
\begin{proof}
We may choose local complex coordinates $(z_1,\ldots,z_n)$ in a neighborhood $U$ of $x$, such that $U \cap V$ is identified with $\{z_1\cdots z_{i+j} = 0\}$, $V_K$ with $\{z_1= \ldots = z_i = 0\}$, and $x$ with $0$. The complex coordinates give a trivialization $(TX|_U,J_0) \simeq (\C^n \times U,J_{\C^n})$. 
The complex coordinates also define a splitting $TU \simeq TV_K \oplus N V_K \simeq \underline{\C^{n-i}}  \oplus \underline{\C^{i}}$.   On $U$, we define the infinitesimal almost complex structure $\widehat{Y}$ to equal $Y$ on the $TV_K$ summand and zero on the $NV_K$ summand.  Notice that $\widehat{Y}(T_y V_q) \subset T_yV_q$ for $y\in U \cap V_q$.  The $C^0$ norm of $\widehat{Y}$ (over $U$) is bounded by $\| Y \|_{C^0}$. Define $\widetilde{Y} = \chi \cdot \widehat{Y}$ on $U$, where $\chi \colon X\to [0,1]$ is a smooth function that is compactly supported in $U$ and equal to $1$ near $x$; extend $\widetilde{Y}$ by zero to $X$. Since the constraints on $\tilde{Y}$ are fibrewise linear, this extension has the desired properties.
\end{proof}

\begin{lem}
\label{lem:extendJ}
The restriction maps $r_K\colon \EuY_*^\ell \to \EuY_{*,K}^\ell$ are surjective for $1\leq \ell \leq \infty$.
\end{lem}

\begin{proof}
Given $Y_K\in  \EuY_{*,K}^\ell$, and given $x\in V_K$, there exists, by Lemma \ref{lem:extendJlocally}, $\widetilde{Y}_x \in \EuY_*^\ell$ such that $r_K(\widetilde{Y}_x)$ agrees with $Y_K$ over $U_x$, for some neighborhood $U_x$ of $x$ in $V_K$. Choose a finite set of points $x_\alpha$ such that $\bigcup_\alpha U_{x_\alpha} = V_K$. Let $\{ \rho_\alpha:V_K \to [0,1]\}$ be a smooth partition of unity subordinate to the open cover $\{ U_\alpha \}$, and let $\widetilde{\rho}_\alpha:X \to [0,1]$ be a smooth extension of $\rho_\alpha$ to $X$, for each $\alpha$. If we then set $Y = \sum_\alpha \widetilde{\rho}_\alpha \widetilde{Y}_{x_\alpha}$, then $Y\in \EuY_*^\ell$ and $r_K(Y)= Y_K$ as required.
\end{proof}

\section{Holomorphic spheres}\label{sec:spheres}

\subsection{Bubble trees}\label{sec:jcons}
For $J \in \EuJ(X,V)$, let us consider the topological space 
\begin{equation} \label{eqn:stable spheres}
\Mbar_k(A;J) 
\end{equation}
of stable $J$-holomorphic spheres with $k$ marked points, with homology class $A \in H_2(X;\Z)$, in Gromov's topology.
We follow the notation of \cite[Sections 5 and 6]{mcduffsalamon} with minor changes.

The topological type of the domain of a stable curve in $\Mbar_k(A;J)$ is prescribed by a $k$-labelled tree $(T,E,\Lambda)$: $T$ is the set of vertices $\alpha$ corresponding to components $\Sigma_\alpha$ of the domain, $E \subset T \times T$ the set of edges $\alpha E \beta$ corresponding to points $z_{\alpha \beta} \in \Sigma_\alpha$ at which $\Sigma_\alpha$ is attached to $z_{\beta \alpha} \in \Sigma_\beta$, and $\Lambda: \{1,\ldots,k\} \to T$ is a function which prescribes the distribution of marked points among the components. 
We associate a function
\[ T \to \mathsf{subsets}(Q): \quad \alpha \mapsto K_\alpha\]
to each stable curve, where $K_\alpha \subset Q$ is defined to be the set of divisors $V_q$ in which the image of the component $u_\alpha: \Sigma_\alpha \to X$ of the stable curve is contained. 
(Because $J \in \EuJ(X,V)$, every smooth $J$-holomorphic curve is either contained in $V_q$ or intersects it positively in finitely many points, for all $q \in Q$.) 
The homology class of the components of the stable curve is the collection of spherical homology classes
\[ A_\alpha \in \im \left( \pi_2(V_{K_\alpha}) \to H_2(V_{K_\alpha};\Z) \right). \]
The images of these classes in $H_2(X;\Z)$ sum to $A$, and they satisfy a stability condition: if $A_\alpha = 0$ then $\Sigma_\alpha$ has at least three special points. 

We define the \emph{combinatorial type} of an element of $\Mbar_k(A;J)$ to be the five-tuple $\Gamma = (T,E,\Lambda,\{K_\alpha\},\{A_\alpha\})$. 
We decompose the moduli space into a disjoint union according to combinatorial type:
\begin{equation}
\label{eqn:Mbar stratification}
\Mbar_k(A;J) = \coprod_\Gamma \cM_\Gamma(J).
\end{equation}
We can describe $\cM_\Gamma(J)$ as the quotient of a space of stable \emph{maps} by a reparametrization group. 
Explicitly, let 
\[
\cM(\Gamma,J) = \{ \mathbf{u} = \{u_\alpha\}_{\alpha \in T} \colon  u_\alpha:\CP^1_\alpha \to V_{K_\alpha} \; J\text{-holomorphic, }[u_\alpha] = A_\alpha\} .
\]
Let 
\begin{equation}\label{eq:ZGamma}
Z(\Gamma) \subset (\CP^1)^E \times (\CP^1)^k
\end{equation}
denote the set of tuples $\mathbf{z}:=\left(\{z_{\alpha \beta} \in \CP^1_\alpha\}_{\alpha E \beta},\{z_i \in \CP^1_{\alpha_i}\}_{1 \le i \le k}\right)$ of marked points, where the marked points on each sphere $\CP^1_\alpha$ are distinct. 
We denote
\begin{equation} \label{eqn:def product over edges}
(X,V)^\Gamma := \prod_{\alpha E \beta} V_{K_\alpha},
\end{equation}
so that there is an evaluation map
\begin{eqnarray}
 \mathrm{ev}^E: \mathcal{M}\left(\Gamma,J\right) \times Z(\Gamma) & \to & (X,V)^\Gamma, \label{eqn:def ev E} \\
(\mathbf{u},\mathbf{z}) & \mapsto & \{u_\alpha(z_{\alpha \beta})\}. \notag 
\end{eqnarray}
We denote
\begin{equation} 
\label{def Delta K}
\Delta^{\Gamma} := \left\{ (x_{\alpha \beta}) \in (X,V)^\Gamma: x_{\alpha \beta} = x_{\beta \alpha}  \in X  \right \}  \subset (X,V)^\Gamma.
\end{equation}
We then denote
\[ \widetilde{\mathcal{M}}_\Gamma\left(J\right) := \left(\mathrm{ev}^E\right)^{-1}\left(\Delta^\Gamma\right).\]
The evaluation map $\mathrm{ev}^E$ is invariant with respect to the left action of the reparametrization group $G_\Gamma  := \prod_\alpha \mathrm{PSL}(2,\C)$ on the domain, and we put $\cM_\Gamma(J) := \widetilde{\cM}_\Gamma(J)/G_\Gamma$. 

We denote by 
\begin{equation}\label{eqn: simple stable curves}
 \cM^*_k(A;J) \subset \Mbar_k(A;J)
 \end{equation} 
 the subset of \emph{simple} stable curves (see \cite[Definition 6.1.1]{mcduffsalamon}), and similarly $\cM^*_\Gamma(J) \subset \cM_\Gamma(J)$, $\widetilde{\cM}^*_\Gamma(J)\subset \widetilde{\cM}_\Gamma(J)$, and $\cM^*(\Gamma,J) \subset \cM(\Gamma,J)$.  
Explicitly, the latter is the subset of tuples $\mathbf{u}$ such that $u_\alpha(\CP^1) \neq u_\beta(\CP^1)$ for $\alpha \neq \beta$ such that $u_\alpha$ and $u_\beta$ are non-constant, and non-constant $u_\alpha$ are simple.

\begin{defn}
\label{defn:Jreg}
An element $u \in \widetilde{\mathcal{M}}^*_\Gamma(J)$ is called \emph{regular} if
\begin{itemize}
\item $u_\alpha:\CP^1 \to V_{K_\alpha}$ is regular, as a $J$-holomorphic curve inside $V_{K_\alpha}$. 
\item The evaluation map $\mathrm{ev}^E$ is transverse to $\Delta^\Gamma$.
\end{itemize}
The moduli space $\widetilde{\cM}^*_\Gamma(J)$ is called regular if every element in it is regular. 
An element $Y \in \EuY_*$ is called regular if $\widetilde{\cM}^*_\Gamma(J_Y)$ is regular for every $\Gamma$. 
The set of regular $Y$  is denoted $\EuY^{\mathrm{reg}}_* \subset \EuY_*$.
\end{defn}

\begin{lem}
\label{lem:dimifreg}
If $Y \in \EuY^{\mathrm{reg}}_*$, then $\mathcal{M}^*_\Gamma(J_Y)$ is a smooth manifold of dimension $\mathrm{dim}(\Gamma)$, where
\begin{align}
\half \mathrm{dim}(\Gamma) 
&= k  + \sum_{\alpha \in T} \left(n -|K_\alpha| - 3\right) + c_1(TV_{K_\alpha})(A_\alpha)+ \sum_{\alpha E \beta} \left(|K_\alpha \cap K_\beta| - n + 2\right)   \label{eq:dimifreg1} \\
& =  k + n - 2 +  \sum_{\alpha \in T} \left(c_1(TV_{K_\alpha})(A_\alpha)  - |K_\alpha| - 1\right) +   \sum_{\alpha E \beta}|K_\alpha \cap K_\beta|. \notag
\end{align}
In the sums $\sum_{\alpha E \beta}$, we take $E$ to be the set of \emph{undirected} edges. 
\end{lem}
\begin{proof}
Follows from standard dimension counting, cf. \cite[Theorem 6.2.6 (i)]{mcduffsalamon}. The second equality uses the Euler characteristic of the tree.
\end{proof}

\begin{lem}
\label{lem:dimineq}
The dimension \eqref{eq:dimifreg1} is bounded above by
\begin{eqnarray*}
 \mathrm{dim}(\Gamma) &\le & 2 \left( k + n + c_1(TX)(A) - 3 - |E| -  \max_\alpha |K_\alpha|\right)  \\
& = & \mathrm{dim}(\mathcal{M}_{k}(A)) - 2|E| - 2\max_\alpha |K_\alpha|
\end{eqnarray*}
(again, we recall that $|E|$ is the number of undirected edges).
\end{lem}
\begin{proof}
The inequality follows by combining two inequalities. 
First,
\[ c_1(TV_{K_\alpha})(A_\alpha) \le c_1(TX)(A_\alpha)\]
by Lemma \ref{lem:c1loss}. 
Second,
\[ \sum_\alpha |K_\alpha| - \sum_{\alpha E \beta} |K_\alpha \cap K_\beta| \ge \max_\alpha |K_\alpha|.\]
To see this, we first choose a vertex $\alpha'$ such that $|K_{\alpha'}|$ is maximal; then we pair each vertex $\alpha$ with the edge $\alpha E \beta$ that points towards the vertex $\alpha'$. 
The term $|K_{\alpha'}|$ in the sum does not have a partner, and the other terms pair off as $|K_\alpha| - |K_{\alpha} \cap K_\beta| \ge  0$, giving the result.
\end{proof}

\subsection{Parametrized moduli spaces}
\label{sec:param}

We will need to consider a parametric generalization of the framework of the previous section; cf. \cite[Section 9h]{Seidel:FCPLT}, \cite[Section 6.7]{mcduffsalamon}.  Suppose that $\EuS$ is a smooth manifold and $\bY: \EuS \to \EuY_*$ a smooth function. 
We denote
\[ \cM_\Gamma(J_\bY) := \coprod_{s \in \EuS} \cM_\Gamma(J_{\bY(s)}).\]
It comes with a map
\[ \pi_\EuS: \cM_\Gamma(J_\bY) \to \EuS\]
recording the parameter. In slightly more detail, we let $\mathcal{B}(V_K)$ be the Fr\'echet manifold of smooth maps $\CP^1\to V_K$, and define $\widetilde{\cM}(\Gamma, J_\bY) \subset \EuS \times  \prod_{\alpha\in T} \mathcal{B} (V_{K_\alpha})$ to be the subspace of tuples $\left(s, \{ u_\alpha\} \right)$ defined by the Cauchy--Riemann equation $\bar{\partial}_{J_{\mathbf{Y} }}  =0$ and the discrete condition $(u_\alpha)_*[\mathbb{CP}^1] = A_\alpha$.  Proceed to construct $\widetilde{\cM}_\Gamma(J_\bY)$, its quotient $\cM_\Gamma(J_\bY)$, and $\cM^*_\Gamma(J_\bY)\subset \cM_\Gamma(J_\bY)$, by analogy with their non-parametric counterparts. 

A first use of parametric moduli spaces (albeit in an infinite-dimensional context) is in proving that there are plenty of regular almost complex structures:

\begin{lem}
\label{lem:regJexists}
The subset $\EuY^{\mathrm{reg}}_* \subset \EuY_*$ is comeagre.
\end{lem}
\begin{proof}
The proof is that of \cite[Theorem 6.2.6 (ii)]{mcduffsalamon}, with adaptations.  The main adaptation is that Lemma \ref{lem:extendJ} shows we have enough perturbations along $V_K$ to achieve transversality for curves contained entirely inside $V_K$.

By the argument of \cite[proof of Theorem 6.2.6 (ii)]{mcduffsalamon}, it suffices to prove that the subset $\EuY_{*}^{\ell,\mathrm{reg}} \subset \EuY^\ell_*$ of regular elements is comeagre for each $\ell \geq 2$.

Fix $\ell$ and a set of indices $K$. The space $\EuY_K^\ell$ has a comeagre subset $\EuY_{K,*}^{\ell,\mathrm{reg}}\subset \EuY_{*,K}^\ell$ whose elements $Y$ are defined by the condition that all simple $J_Y$-spheres $u\colon \C P^1\to V_K$, not contained in $V^K\cap V_K$, are regular (standard adaptation of the argument of \cite[Theorem 3.1.5]{mcduffsalamon}). 

Define 
\[ \EuY_*^{\ell, \mathrm{reg\, for\, } K} = \{ Y \in \EuY_*^\ell: \, r_K(Y) \in \EuY_{*,K}^{\ell, \mathsf{reg}}\}.\]  
The restriction map $r_K \colon \EuY_*^\ell \to   \EuY_{*,K}^{\ell}$ is surjective by Lemma \ref{lem:extendJ}.  The preimage of a comeagre subset under a continuous linear surjection is again comeagre; hence $\EuY_*^{\ell, \mathrm{reg\, for \, } K}$ is comeager in $\EuY_*^\ell$. Thus $\bigcap_K \EuY_*^{\ell, \mathrm{reg\, for \, } K}$ is also comeager. This deals with the first bullet point in the definition of regularity (Definition \ref{defn:Jreg}). 

One next needs to prove that the evaluation map
\[  \ev^E \colon \cM^*(\Gamma,\EuJ_\ell ) \times Z(\Gamma)  \to (X,V)^\Gamma  \] 
(cf. (\ref{eqn:def ev E})) is transverse to the diagonal $\Delta^\Gamma$; here $\EuJ^\ell =\{ J_Y: Y \in \EuY^\ell_* \}$ is the family of almost complex structures parametrized by $\EuY^\ell_*$, and $\cM^*(\Gamma,\EuJ_\ell ) = \bigcup_{J\in \EuJ^\ell} \cM^*(\Gamma,J)$. 
This can be accomplished by an adaptation of \cite[Prop. 6.2.8]{mcduffsalamon}. The proof in \emph{op. cit.} hinges on \cite[Prop. 3.4.3]{mcduffsalamon}, and it suffices to use the tweak to the latter proposition indicated in \cite[Remark 3.4.8]{mcduffsalamon}. With that point in place, the proof of \cite[Theorem 6.2.6 (ii)]{mcduffsalamon} goes through.
\end{proof}

We return to the discussion of parametric moduli spaces. An element $u \in \widetilde{\cM}^*_\Gamma(J_\bY)$ is called regular if it is cut out transversely by its defining equations. That is, each $u_\alpha$ is parametrized-regular, i.e., a transverse zero of the Cauchy--Riemann operator on $\mathcal{B}(V_{K_\alpha})\times \EuS$, and the evaluation map 
\begin{eqnarray*}
 (\pi_\EuS,\mathrm{ev}^E): \mathcal{M}\left(\Gamma,J_\bY \right) \times Z(\Gamma) & \to & \EuS \times (X,V)^\Gamma
\end{eqnarray*}
is transverse to $\EuS \times \Delta^\Gamma$.

The moduli space $\widetilde{\cM}^*_\Gamma(J_\bY)$ is called regular if every element in it is regular; and $\bY$ is called regular if $\widetilde{\cM}^*_\Gamma(J_\bY)$ is regular for every $\Gamma$. 
The proof of Lemma \ref{lem:regJexists} generalizes easily to show that the set of regular $\bY$ is comeagre in $C^\infty(\EuS,\EuY_*)$. In this situation, $\cM^*_\Gamma(J_\bY)$ is a smooth manifold of dimension $\dim(\Gamma) + \dim(\EuS)$.

\section{Holomorphic discs}\label{sec:discs}

\subsection{Deligne--Mumford moduli space}

If $\ell \ge 3$, we consider the Deligne--Mumford moduli space $\Rbar_\ell$ of stable curves of genus zero (that  is, spheres) with  marked points labelled $z_1,\ldots,z_\ell$. 
It is a complex manifold, and carries a universal family $\Sbar_\ell \to \Rbar_\ell$ with sections $\{z_i\}_{i =1,\ldots,\ell}$ given by the marked points.
We denote the restriction of the universal family to the smooth locus by $\EuS_{\ell} \to \EuR_{\ell}$. 

If $k+2\ell \ge 2$, we consider the Deligne--Mumford moduli space $\Rbar_{k,\ell}$ of stable discs equipped with boundary marked points labelled $\zeta_0,\ldots,\zeta_k$ (in order) and interior marked points labelled $z_1,\ldots,z_\ell$. 
It is a smooth manifold with corners and carries a universal family $\Sbar_{k,\ell} \to \Rbar_{k,\ell}$ with sections $\{z_i\}_{i=1,\ldots,\ell}$ and $\{\zeta_i\}_{i=0,\ldots,k}$ given by the marked points.  
We denote the restriction of the universal family to the smooth locus by $\EuS_{k,\ell} \to \EuR_{k,\ell}$.

The strata of the Deligne--Mumford compactification 
\[ \Rbar_{k,\ell} = \coprod_\Gamma \EuR^\Gamma_{k,\ell}\]
are indexed by certain decorated trees $\Gamma$. 
Specifically, each $\Gamma$ consists of a tree with vertex set $V= V_\sph \sqcup V_\disc$ decomposed into `sphere' and `disc' vertices; a set of `finite' edges $E = E_\sph \sqcup E_\disc \subset V \times V$ decomposed into `disc' edges (those connecting disc vertices) and `sphere' edges (the rest); a set of `semi-infinite sphere edges' indexed by $\{1,\ldots,\ell\}$, which is equivalent to an $\ell$-marking of $(V,E)$; a set of `semi-infinite disc edges' indexed by $\{0,\ldots,k\}$ and attached to disc vertices, which is equivalent to a $k+1$-marking of $(V_\disc,E_\disc)$; the graph $(V_\disc,E_\disc)$ together with the semi-infinite disc edges is required to be a tree, and must come equipped with an isotopy class of embeddings in the plane so that the semi-infinite edges are labelled in order from $0$ to $k$ around the boundary.  
For each vertex $\alpha$, we denote by $\ell_\alpha$ the number of sphere edges incident to $\alpha$; and for each disc vertex $\alpha$, we denote by $k_\alpha$ the number of disc edges incident to $\alpha$, minus one. 
$\Gamma$ is required to be \emph{stable}, which means that $k_\alpha+2\ell_\alpha \ge 2$ when $\alpha$ is a disc vertex, and $\ell_\alpha\ge 3$ when $\alpha$ is a a sphere vertex.

The stratum indexed by such a $\Gamma$ has the form
\[ \EuR^\Gamma_{k,\ell} \simeq \prod_{\alpha \in V_\sph} \EuR_{\ell_\alpha} \times \prod_{\alpha \in V_\disc} \EuR_{k_\alpha,\ell_\alpha},\]
where the diffeomorphism is determined by a choice of ordering of the sphere edges incident to each vertex. 
Specifically, given such an ordering, each finite sphere edge $\alpha E \beta$ determines a section $z_{\alpha \beta}$ of the universal family $\EuS_\alpha \to \EuR_\alpha$ of spheres or discs; each finite disc edge determines a section $\zeta_{\alpha \beta}$ of the universal family $\EuS_\alpha \to \EuR_\alpha$ of discs. 
Letting $\mathrm{pr}_\alpha: \EuR^\Gamma \to \EuR_\alpha$ denote the projection, the restriction of the universal family to such a stratum has the form
\[ \EuS^\Gamma_{k,\ell} \simeq \left(\coprod_{\alpha \in V_\sph} \mathrm{pr}_\alpha^* \EuS_{\ell_\alpha} \sqcup \coprod_{\alpha \in V_\disc} \mathrm{pr}_\alpha^* \EuS_{k_\alpha,\ell_\alpha}\right)/\sim\]
where `$\sim$' identifies $\mathrm{pr}_\alpha^* z_{\alpha \beta} \sim \mathrm{pr}_\beta^* z_{\beta \alpha}$ for all finite spherical edges $\alpha E \beta$ and $\mathrm{pr}_\alpha^*\zeta_{\alpha \beta} \sim \mathrm{pr}_\beta^*\zeta_{\beta \alpha}$ for all finite disc edges $\alpha E \beta$. 

We now observe that the group $\Sym(\ell)$ acts on $\Sbar_{k,\ell} \to \Rbar_{k,\ell}$ by relabelling the interior marked points, and the action preserves the smooth locus $\EuS_{k,\ell} \to \EuR_{k,\ell}$. 

\begin{lem}\label{lem:free}
The action of $\Sym(\ell)$ on $\EuR_{k,\ell}$ is free, and therefore properly discontinuous.  
\end{lem}
\begin{proof}
Since $\Sym(\ell)$ is finite and $\EuR_{k,\ell}$ is Hausdorff, a free action is properly discontinuous.  To say that the action is free is to say that if $f: \Sigma \to \Sigma$ is an isomorphism which fixes the boundary markings and permutes the interior markings, then $f$ is the identity. Note that $f^{\ell!}$ must act as the identity on all marked points, hence be equal to the identity, since all domains we consider have trivial automorphism group. But the group of automorphisms of the disc fixing a single boundary point has no nontrivial roots of the identity;\footnote{To prove this claim, observe that if the boundary fixed-point set $P=\{ t \in \partial \mathbb{D} : \alpha (t) = t)\}$ of an automorphism  $\alpha$ of the disc $\mathbb{D}$ has cardinality $|P| \ge 3$, then $\alpha$ is the identity; if $|P| \in \{1, 2\}$ then $\alpha$ preserves each connected component $I$ of $\partial \mathbb{D} \setminus P$, because $\alpha$ preserves the boundary orientation. For each component $I$, either $\alpha(t) < t$ for all $t\in I$, or  $\alpha(t) >t$ for all $t\in I$ ($<$ is the order on an oriented interval); either property persists in iterates.} 
 therefore $f$ itself is the identity as required.
\end{proof}

\begin{rmk}
The argument extends to show that, for $\Sigma$ a stable disc, if $f: \Sigma \to \Sigma$ is an isomorphism which fixes the boundary markings and permutes the interior markings, then $f$ acts as the identity on each disc component of $\Sigma$. Thus the action of $\Sym(\ell)$ on strata of $\EuR_{k,\ell}$ without spherical components is again free and properly discontinuous.
\end{rmk}

\begin{rmk}
The action of $\Sym(\ell)$ on $\Rbar_{k,\ell}$ is not free, in general. For example, if there is a spherical component with exactly one node and exactly two marked points, there is an automorphism interchanging those marked points.  There is a quotient orbifold-with-corners $\Rbar_{k,\ell}/\Sym(\ell)$---locally the quotient of a manifold with corners by the action of a finite group---parametrizing stable discs whose interior marked points are unordered. The interior stratum $\EuR_{k,\ell}/\Sym(\ell)$ is both a manifold and a fine moduli space; and the same goes for the codimension-1 strata, since they parametrize stable discs with two disc components but no spherical components.
\end{rmk}

\subsection{Strip-like and cylindrical ends}\label{sec:slends}

For any $r \in \Rbar_\ell$ or $\Rbar_{k,\ell}$, we denote the corresponding fibre of the universal family by $\Sigma_r$. 
We denote by 
\begin{equation}\label{eqn:Sigma punctured}
\Sigma^\circ_r \subset \Sigma_r
\end{equation}
the complement of the boundary marked points and boundary nodes. (Thus $\Sigma^\circ_r = \Sigma_r$, when $r \in \Rbar_\ell$.)

A \emph{parametrized cylindrical end} for the $i$th internal marked point $z_i \in \Sigma_r$ is a holomorphic embedding
$$\epsilon: \mathbb{D} \to \Sigma^\circ_r$$
where $\mathbb{D} \subset \C$ is the closed unit disc, such that $\epsilon(0) = z_i$. 
We have an action of $S^1$ on the set of such cylindrical ends, by $\epsilon \mapsto \epsilon(e^{i\theta} z)$. 
An equivalence class for this action is called an \emph{unparametrized cylindrical end}; these are the ones we care about, so we will drop `unparametrized'.

A \emph{strip-like end} for the $i$th boundary marked point $\zeta_i \in \Sigma_r$ consists of a holomorphic embedding
$$
\epsilon: Z^\pm \to \Sigma^\circ_r
$$
where $Z^\pm := \R_\pm \times [0,1]$ and the sign is $-$ if $i=0$ and $+$ otherwise. The map is required to be proper, send the boundary intervals $\R_\pm \times \{0,1\}$ to $\partial \Sigma^\circ_r$, and limit to $\zeta_i$ as $s \to \pm \infty$ (see \cite[Section 8d]{Seidel:FCPLT}).

A \emph{choice of cylindrical and strip-like ends} for $\Sigma^\circ_r$ consists of a choice of cylindrical (respectively, strip-like) end for each interior (respectively, boundary) marked point, having disjoint images. 

We denote by $\Sbar^\circ_{k,\ell} \subset \Sbar_{k,\ell}$ the union of subsets $\Sigma^\circ_r$ over all $r$. 
Let $\Rbar_{k,\ell}^{(i)}$ denote the real oriented projectivization (space of half-rays) of the complex line bundle over $\Rbar_{k,\ell}$ with fibre $\hom(T_0 \mathbb{D},T_{z_i} \Sigma^\circ_r)$ over $r$. 
A \emph{universal choice of cylindrical end} for the $i$th interior marked point consists of a smooth $S^1$-equivariant map
$$\epsilon: \mathbb{D} \times \Rbar_{k,\ell}^{(i)} \to \Sbar^\circ_{k,\ell},$$
such that for any $\xi \in \Rbar_{k,\ell}^{(i)}$, the (equivalence class under positive real scaling of the) differential of the map $\epsilon(-,\xi)$ at $0 \in \mathbb{D}$ coincides with $\xi$, and which represents a cylindrical end on each fibre of the universal family. (Here $S^1 \subset \C$ acts by multiplication on $\mathbb{D}$ and $\Rbar_{k,\ell}^{(i)}$, and trivially on the target.)

We similarly define the notion of a universal choice of cylindrical end for the $i$th marked point on $\Sbar^\circ_\ell = \Sbar_\ell$.

A \emph{universal choice of strip-like end} for the $i$th boundary marked point consists of a smooth map
\[\epsilon: \Rbar_{k,\ell} \times Z^\pm \to \Sbar^\circ_{k,\ell},\]
which restricts to a strip-like end on each fibre of the universal family. 

A \emph{universal choice of cylindrical and strip-like ends} is a universal choice of cylindrical (respectively, strip-like) end for each interior (respectively, boundary) marked point in $\Rbar_{k,\ell}$ and $\Rbar_\ell$, having disjoint images. 
We require the universal choice to be \emph{consistent} with respect to gluing along cylindrical or strip-like ends (in the sense of \cite[Section 9g]{Seidel:FCPLT}), and furthermore $\Sym(\ell)$-equivariant.

\begin{lem}  \label{lem:cons_ends}
A consistent and equivariant universal choice of cylindrical and strip-like ends exists.
\end{lem}
\begin{pf} 
This is proved by a modification of the argument used to prove existence of consistent strip-like ends for the spaces $\Rbar_{k,0}$, cf. \cite[Lemma 9.3]{Seidel:FCPLT}.
We assume, inductively, that a consistent and equivariant choice has been made for all $\Rbar_{k',\ell'}$ and $\Rbar_{\ell'}$ of dimension $<N$, and show that it can be extended to $\Rbar_{k,\ell}$ and $\Rbar_\ell$ of dimension $N$. 
The base $N=0$ of the induction requires us to choose equivariant cylindrical and strip-like ends for the zero-dimensional moduli spaces $\Rbar_{0,1}$, $\Rbar_{2,0}$, and $\Rbar_3$; the requirement of consistency is empty, and the requirement of equivariance is only non-trivial for $\Rbar_3$, where it is easily arranged.
 
Now let $N>0$, and let us consider a disc moduli space $\Rbar_{k,\ell}$ (the sphere case is analogous). The requirement of consistency dictates the choice over some neighbourhood $U$ of the boundary strata $\Rbar_{k,\ell} \setminus \EuR_{k,\ell}$; because the previous choices were equivariant by construction, and the gluing construction is manifestly equivariant, this choice over $U$ is equivariant. 
Thus it descends to a choice of cylindrical and strip-like ends over $U/\Sym(\ell)$.\footnote{We abuse the terminology `choice of cylindrical ends' here, as there is no globally-defined ordering of the interior marked points, but the meaning should be clear. (Note that the problem of extending choices of cylindrical and strip-like ends is local in the base, and we can alway choose a local ordering of the interior marked points.)} 
We can extend this choice from $(U \cap \EuR_{k,\ell})/\Sym(\ell)$ to $\EuR_{k,\ell}/\Sym(\ell)$ (which is a manifold by Lemma \ref{lem:free}) as in \cite[Lemma 9.3]{Seidel:FCPLT}, then pull back to $\EuR_{k,\ell}$; the result is equivariant and consistent by construction.
\end{pf}

Given a consistent equivariant universal choice of cylindrical and strip-like ends, we may choose a \emph{universal thick-thin decomposition} (cf. \cite[Remark 9.1]{Seidel:FCPLT} -- note that this is an additional choice). 
This consists of an open set 
\begin{equation}\label{eqn:S thin}
\Sbar^{thin}_{k,\ell} \subset \Sbar^\circ_{k,\ell},
\end{equation} 
whose intersection with each fibre is parametrized as the union of strip-like ends for each disc component of the fibre, together with strip-like `gluing regions' $\epsilon_e([0,l_e] \times [0,1])$, which are also required to be $\Sym(\ell)$-equivariant and disjoint from the marked points. (In Seidel's terminology, we may take $\Sbar^{thin}_{k,\ell}$ to be the union of all strip-like ends together with the union of gluing regions lying over subsets where the gluing parameters are sufficiently small.) It is straightforward to make the choices of gluing regions $\Sym(\ell)$-equivariant, as in the proof of Lemma \ref{lem:cons_ends}.  Note that we do \emph{not} include the cylindrical ends and cylindrical gluing regions in the thin region, as we will not require a gluing argument along such a cylindrical region in this paper.

\subsection{Floer data}

Let $(W \subset X,\omega,\theta,V, J_0)$ be data as in Section \ref{sec:symp_rel_fuk}: $(X,\omega)$ is a compact symplectic manifold, $W \subset X$ is a Liouville subdomain with Liouville one-form $\theta$, $V \subset X \setminus W$ is a system of divisors, $J_0 \in \EuJ(X,V)$, and $(W \subset X,J_0)$ admits a convex collar. 
Recall the definition of $\EuY_* \subset \EuY$. 

We set
\[ \Ymax_* := \{Y \in \EuY_*: \mathrm{supp}(Y) \subset W\}.\]
Note that if $(U,h)$ is a convex collar for $(W \subset X,J_0)$, and $Y \in \Ymax_*$, then $(U,h)$ is also a convex collar for $(W \subset X,J_Y)$.

\begin{defn}\label{defn:obj}
A \emph{Lagrangian brane} is a closed, exact Lagrangian submanifold $L \subset \mathring{W}$, equipped with a Pin structure.
\end{defn}

Let $\EuH \subset C^\infty(X;\R)$ denote the set of smooth functions supported in $W$. 
For each pair of Lagrangian branes $(L_0,L_1)$ we choose a smooth function 
\[ H_{01}:[0,1] \to \EuH.\]
We assume that the time-1 flow of the corresponding Hamiltonian vector field $X_{H_{01}(t)}$, when applied to $L_0$, makes it transverse to $L_1$. 
Thus, the set of Hamiltonian chords $y$ from $L_0$ to $L_1$ is finite.

Next, for each pair of Lagrangian branes $(L_0,L_1)$ we consider a smooth function
\begin{align*}
Y_{01}: [0,1]& \to \EuY_*^{\max{}}, \qquad \text{giving rise to} \\
J_{01}:[0,1] & \to \EuJ^{\max{}}(X,V), \quad \text{via $J_{01} := J_{Y_{01}}$.}
\end{align*}
The corresponding equation for Floer trajectories between chords $y_0$ and $y_1$ in the Liouville domain $W$ is
\begin{align*}
u: \R \times [0,1] & \to W \\
\partial_s u + J_{01}(t) (\partial_t u - X_{H_{01}(t)} \circ u) &=0\\
u(s,i) & \in L_i \qquad \text{for $i=0,1$} \\
\lim_{s \to +\infty} u(s,\cdot) &= y_1  \\
\lim_{s \to -\infty} u(s,\cdot) &= y_0.
\end{align*}
For a comeagre subset of the space of choices of $Y_{01}$, the moduli space of such Floer trajectories is regular \cite{FHS:transversality}. 
We choose such a regular $Y_{01}$, for each pair of Lagrangian branes $(L_0,L_1)$ equipped with $H_{01}$.

Taken together, the choice of $(H_{01},Y_{01})$ for each pair $(L_0,L_1)$ is called a choice of \emph{Floer data}.

\subsection{Perturbation data}

A \emph{Lagrangian labelling} $\bL$ is a tuple of Lagrangian branes $L_0,\ldots,L_k$ (with $k \ge 0$). 
Given a family of stable discs $\Rbar_{k,\ell}$, a Lagrangian labelling $\bL$ will be thought of as a labelling of the boundary components of discs in the family by the branes $L_i$. 
Let $r \in \Rbar_{k,\ell}$, $\Sigma_r$ the fibre of the universal family over $r$, $\Sigma^\circ_r$ the complement of all boundary marked points and nodes, and $\tilde{\Sigma}^\circ_r$ the normalization of $\Sigma^\circ_r$. 
Following \cite[Section 8e]{Seidel:FCPLT}, we make the following:

\begin{defn}\label{def:perdat}
A \emph{perturbation datum} for $\Sigma_r^\circ$ equipped with boundary labelling $\bL$ consists of a pair $P=(Y,K)$ where 
\[Y \in C^\infty(\tilde{\Sigma}^\circ_r, \EuY_*),\quad K \in \Omega^1(\tilde{\Sigma}^\circ_r,\EuH),\]
 satisfying
 \begin{itemize}
 \item[]
 \textbf{(Constant on spheres)} $Y$ is constant, and $K$ vanishes, on each spherical component of $\tilde{\Sigma}^\circ_r$; and furthermore, $Y(z_{\alpha \beta}) = Y(z_{\beta \alpha})$ for every sphere edge $\alpha E \beta$;
\item[]
\textbf{(Thin regions)}
 over every thin region (strip-like end or strip-like gluing region) parametrized by $\epsilon$, and having boundary labels $L_\ell,L_r$, we have
\[ Y(\epsilon(s,t)) = Y_{\ell,r}(t)\qquad \epsilon^* K = H_{\ell,r}(t) dt;\]
\item[]
\textbf{(Boundary)} over each component $C$ of the boundary labelled by $L_C$, we have
\[ K(\xi)|_{L_C} =0 \quad \text{for all $\xi \in TC \subset T(\partial \Sigma^\circ_r)$}.\]

\item []
\textbf{(Maximum principle)} If $\ell=0$, then $Y \in \EuY^{\max{}}_*\subset \EuY_*$.   
\end{itemize}
\end{defn}

Let $\Sigma^{\circ \circ}_r \subset \Sigma^\circ_r$ denote the complement of all boundary \emph{and interior} marked points and nodes of $\Sigma_r$.  
Let $\Sbar^{\circ \circ}_{k,\ell} \subset \Sbar^{\circ}_{k,\ell}$ denote the union of subsets $\Sigma^{\circ \circ}_r \subset \Sigma^\circ_r$ over all fibres. 
Note that $\Sbar^{\circ \circ}_{k,\ell} \to \Rbar_{k,\ell}$ is a submersion, so the bundle of fibrewise one-forms $\Omega^1_{\Sbar^{\circ\circ}_{k,\ell}/\Rbar_{k,\ell}}$ makes sense.

\begin{defn}\label{defn:pert}
A \emph{universal choice of perturbation data} consists of pairs $\bP_{k,\ell,\bL} = (\bY_{k,\ell,\bL},\bK_{k,\ell,\bL})$ for all $k+2\ell \ge 2$ and all Lagrangian labellings $\bL$, where
\[ \bY_{k,\ell,\bL} \in C^\infty\left(\Sbar^\circ_{k,\ell}, \EuY_*\right),\qquad \bK_{k,\ell,\bL} \in \Omega^1_{\Sbar^{\circ\circ}_{k,\ell}/\Rbar_{k,\ell}}\left(\EuH\right)\]
(so $\bK_{k,\ell,\bL}$ is an $\EuH$-valued, fiberwise linear map defined on the vertical tangent bundle of the universal curve $\Sbar^{\circ \circ}_{k,\ell} \to \Rbar_{k,\ell}$). 
The restriction of $\bP_{k,\ell,\bL}$ to each fibre $\Sigma_r$ of the universal family should be a perturbation datum for that fibre, in the sense of Definition \ref{def:perdat}.\footnote{More precisely, the pullback of $\bY_{k,\ell,\bL}(r)$ to $\tilde{\Sigma}^\circ_r$ should form one part of the perturbation datum, and the one-form $\bK_{k,\ell,\bL}(r)$ should admit an extension from $\Sigma^{\circ \circ}_r$ to $\tilde{\Sigma}^\circ_r$, forming the other part of a perturbation datum. (Of course the extension is unique, if it exists.)} 
Furthermore, the $\bP_{k,\ell,\bL}$ should satisfy:
\begin{enumerate}
\item []
\textbf{(Equivariant)} $\bP_{k,\ell,\bL}$ is $\Sym(\ell)$-equivariant;

\item []
\textbf{(Consistent on discs)} The restriction of $\bP_{k,\ell,\bL}$ to $\mathrm{pr}_\alpha^* \Sbar^\circ_{k_\alpha,\ell_\alpha}$ is equal to $\mathrm{pr}_\alpha^* \bP_{k_\alpha,\ell_\alpha,\bL_\alpha}$ for any disc vertex $\alpha \in V_\disc$ of a tree $\Gamma$.
\end{enumerate}
\end{defn}

\begin{rmk}
Recall that the embedding $\Sbar^\circ_{k_\alpha,\ell_\alpha} \hookrightarrow \Sbar^\circ_{k,\ell}$  depends on a choice of ordering of the finite spherical edges incident to $\alpha$. By the \textbf{(Equivariant)} condition satisfied by $\bP_{k_\alpha,\ell_\alpha,\bL_\alpha}$, the \textbf{(Consistent on discs)} condition on $\bP_{k,\ell,\bL}$ does not depend on this ordering.
\end{rmk}

\begin{lem}\label{lem:pertinduct}
Suppose $k+2\ell=N$, and that $\bP_{k',\ell',\bL'}$ satisfying the conditions of Definition \ref{defn:pert} have been chosen for all $k'+2\ell' <N$. Then there exist $\bP_{k,\ell,\bL}$ satisfying the conditions of Definition \ref{defn:pert}. 
\end{lem}
\begin{proof}
The \textbf{(Consistent on discs)} condition uniquely specifies the restriction of $(\bY_{k,\ell,\bL},\bK_{k,\ell,\bL})$ to all boundary strata of the form $\mathrm{pr}_\alpha^*\Sbar^\circ_{k_\alpha,\ell_\alpha}$ for $\alpha \in V_\disc$; the \textbf{(Constant on spheres)} condition then uniquely specifies the restriction of $(\bY_{k,\ell,\bL},\bK_{k,\ell,\bL})$ to the remaining boundary strata of the form $\mathrm{pr}_\alpha^*\Sbar_{\ell_\alpha}$ for $\alpha \in V_\sph$, because each maximal tree of sphere components intersects a disc component in a unique point. 
Thus the restriction of $(\bY_{k,\ell,\bL},\bK_{k,\ell,\bL})$ to all boundary strata is uniquely specified by the preceding choices. 
It is clear that these uniquely-specified restrictions satisfy the \textbf{(Equivariant)} condition (using the fact that a constant map is equivariant) and the \textbf{(Thin regions)} condition. 

Having specified $(\bY_{k,\ell,\bL},\bK_{k,\ell,\bL})$ over the boundary strata, we need to show that it extends over all of $\Sbar^\circ_{k,\ell}$ (cf. \cite[Lemma 9.5]{Seidel:FCPLT}). 
Such an extension can be constructed using appropriate cutoff functions, because $\EuY_*$ and $\EuH$ are convex, and it can easily be arranged to satisfy the \textbf{(Thin regions)} condition, using the fact that our choice of strip-like ends is consistent. 
The \textbf{(Equivariant)} condition can then be achieved by averaging, again using convexity. 
The averaged perturbation data continues to respect the \textbf{(Thin regions)} condition, because we chose the parametrization of the thin regions to be equivariant.
\end{proof}

\subsection{Pseudoholomorphic curves}

Given a Lagrangian labelling $\bL=(L_0,\ldots,L_k)$, a \emph{choice of Hamiltonian chords} $\by=(y_0,\ldots,y_k)$ consists of chords $y_0$ from $L_0$ to $L_k$ and $y_i$ from $L_{i-1}$ to $L_i$ for $i=1,\ldots,k$. 
We will be considering maps $u:\Sigma^\circ \to X$, where $\Sigma^\circ$ is a disc with $k+1$ boundary punctures, each boundary component is constrained to lie on the Lagrangian by which it is labelled, and $u$ limits to the chords $y_i$ at the punctures. 
We denote the set of homotopy classes of such maps by $\pi_2(\by)$. 
To each $A \in \pi_2(\by)$ we can associate certain numerical invariants: a Maslov-type index $i(A)$, equal to the Fredholm index of the linearization of the pseudo-holomorphic curve equation in the form to be discussed shortly,\footnote{In the case that our Floer groups are $\Z$-graded, we have $i(A) = i(y_0) - \sum_{j=1}^k i(y_j)$, with $i(y)$ denoting the degree of the generator $y$.} and topological intersection numbers $A \cdot V_q$ for $q \in Q$. 
We denote $\pi_2^\num(\by):= \pi_2(\by)/\sim$, where the equivalence relation identifies two classes if they have the same numerical invariants. 
Note that given $A \in \pi_2^\num(\by)$ and a spherical homology class $A' \in \im \pi_2(X) \to H_2(X;\Z)$, we can form the connect sum $A + A' \in \pi_2^\num(\by)$: the connect sum has Maslov index $i(A) + 2c_1(A')$ and intersection numbers $A \cdot V_q + A' \cdot V_q$, so the resulting equivalence class is independent of the choices involved in forming the connect sum.

Because we choose our almost complex structures to respect the components $V_q$ of the system of divisors $V$, any holomorphic curve which is not contained in $V$ (such as a disc whose boundary is contained in Lagrangians disjoint from $V$) will intersect each $V_q$ positively in a finite set of points. 
We will record the local intersection number at $z$ of a holomorphic curve $u$ with the components of $V$ as a `tangency vector' $\iota(u,z) \in (\Z_{\ge 0})^Q$.

Now, we will only need to consider classes $A \in \pi_2^\num(\by)$ such that $A \cdot V_q \ge 0$ for all $q$ (because of Lemma \ref{lem:posint}). 
Given such an $A$, a choice of \emph{tangency data} is a number $\ell \ge 0$ and a function $\tang:\{1,\ldots,\ell\} \to \Z_{\ge 0}^Q$ such that $\sum_{i=1}^\ell \tang(i)_q = A \cdot V_q$.

Given a Lagrangian labelling $\bL$, a universal choice of perturbation data associates an almost complex structure $J_z := J_{\bY_{k,\ell,\bL}(z)} \in \EuJ(X,V)$ to every point $z$ on a fibre $\Sigma_r^\circ$ of the universal family $\Sbar^\circ_{k,\ell} \to \Rbar_{k,\ell}$. 
It also determines a one-form $\nu$ on $\Sigma_r^\circ$ with values in (Hamiltonian) vector fields on $X$: $\nu(\xi)$ is the Hamiltonian vector field corresponding to the Hamiltonian function $\bK_{k,\ell,\bL}(\xi)$.
 
Now, suppose we are given a Lagrangian labelling $\bL$, a choice of Hamiltonian chords $\by$, an element $A \in \pi_2^\num(\by)$, a choice of tangency data $\tang$, and a universal choice of perturbation data $\bP$. 
We define a moduli space $\cM(\by,A,\tang,\bP)$ if $k+2\ell \ge 1$; the moduli space is undefined if $k=\ell=0$. 
If $k+2\ell \ge 2$, we define the moduli space of pseudo-holomorphic discs with tangency conditions
\begin{equation}\label{eqn: tang moduli}
\cM(\by,A,\tang,\bP)
\end{equation}
to consist of pairs $(r,u)$ where $r \in \EuR_{k,\ell}$ and $u: \Sigma_r^\circ \to X$ satisfies the resulting pseudoholomorphic curve equation:
\begin{align*}
(Du-\nu)\circ j_z &= J_z \circ (Du-\nu)  \\
u(C_i) & \subset L_i \notag  \\
\lim_{s \to \pm \infty} u(\epsilon_i(s,\cdot)) &= y_i(\cdot)\notag  \\ 
\iota(u,z_i)  & \ge \tang(i) \text{ for $i=1,\ldots,\ell$} \notag \\
[u] &=A. \notag
\end{align*}

Note that $\Sym(\ell)$ acts on $\bigcup_{\tang} \cM(\by,A,\tang,\bP)$, by the \textbf{(Equivariant)} condition on the perturbation data. The action is free by Lemma \ref{lem:free}. 
We remark that the condition $\iota(u,z_i) \ge \tang(i)$ is equivalent to the condition $\iota(u,z_i) = \tang(i)$, by Lemma \ref{lem:posintCM} below; the inequality is equivalent to a jet condition \eqref{eqn:jet vanishing}, which fits within our functional-analytic framework for analyzing the moduli space.

In the remaining case $k+2\ell = 1$, we must have $(k,\ell) = (1,0)$ (so $\EuR_{k,\ell}$ is undefined). 
In this case we have $\by = (y_0,y_1)$, and we define $\cM(\by,A,\emptyset,\bP)$ to be the moduli space of Floer trajectories from $y_0$ to $y_1$, modulo translation. 

\begin{lem}\label{lem:posintCM}
If $u \in \cM(\by,A,\tang,\bP)$ then $u$ intersects $V_q$ at the marked points $z_i$ with $\tang(i)_q \ge 1$, with multiplicity $\tang(i)_q$, and nowhere else.
\end{lem}
\begin{proof}
By \cite[Proposition 7.1]{Cieliebak2007}, each intersection point $z$ of $u$ with $V_q$ contributes $\iota(u,z)_q \ge 1$ to $[u] \cdot V_q = A \cdot V_q$, so we have
\[ A \cdot V_q \ge \sum_i \iota(u,z_i)_q \ge \sum_i \tang(i)_q = A \cdot V_q.\]
Thus we have equality, which implies the result. 
\end{proof}

If we set $\ell = A \cdot V$ and let $\vv:\{1,\ldots,\ell\} \to Q$ be a function with $|\vv^{-1}(q)|=A \cdot V_q$, we can define a function $\tang^{\can}:\{1,\ldots,\ell\} \to \Z_{\ge 0}^Q$ by setting $\tang^{\can}(i)_q = 1$ if $\vv(i) = q$ and $0$ otherwise. 
We denote the resulting moduli space by 
\begin{equation}\label{eqn: can tang}
\cM(\by,A,\bP) := \cM(\by,A,\tang^{\can},\bP).
\end{equation}
It is independent of the choice of $\vv$ up to the action of $\Sym(\ell)$.
We denote by $\Sym(\vv) \subset \Sym(\ell)$ the subgroup preserving $\vv$, and observe that it acts freely on $\cM(\by,A,\bP)$.

\subsection{Transversality}\label{sec:transv}

In this section we will define what it means for our perturbation data to be `regular'; in the next we will show how this leads to the desired Gromov-type compactification of our moduli spaces. 

Let $(\by,A_0,\tang,\bP)$ be data as in the previous section. Let $I \subset \{1,\ldots,\ell\}$ be a subset containing all $i$ such that $\tang(i)=0$, and let $\{\Gamma_i\}_{i \in I}$ be combinatorial types of bubble trees with $1$ marked point, in the sense of Section \ref{sec:jcons}. 
We will denote the image of the total homology class of $\Gamma_i$ in $\pi_2^\num(\by)$ by $A_i$, and set
\[ A := A_0 + \sum_{i \in I} A_i.\]

\begin{figure}
\begin{center}
\includegraphics[width=0.5\textwidth]{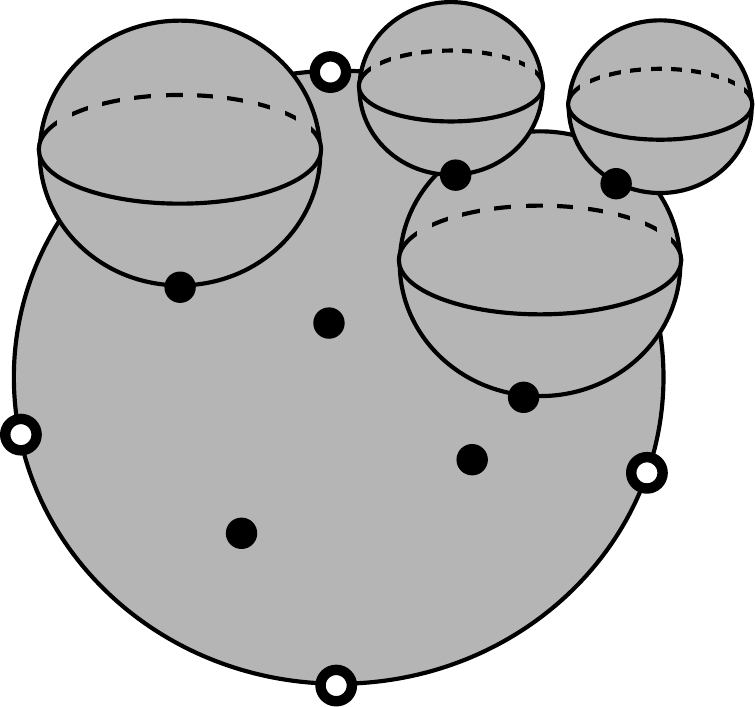}
\end{center}
\caption{\label{fig:disk3}A pseudoholomorphic disc with two holomorphic bubble trees attached. Boundary punctures are denoted by an open point, and interior marked points and nodes are denoted by a solid point. Note that the bubble trees carry no marked points.}
\end{figure}

Associated to this data we have a moduli space of discs with `$J_{z_i}$-holomorphic bubble trees of combinatorial type $\Gamma_i$, attached at their respective unique marked points to the points $z_i$, for $i \in I$' (see Figure \ref{fig:disk3}). 
To be precise, we introduce notation: $K_i \subset Q$ is the set of non-zero coordinates of $\tang(i)$; and $K'_i \subset Q$ is the subset attached to the component of $\Gamma_i$ containing the marked point.
Starting with the moduli spaces (\ref{eqn: tang moduli}), we consider the evaluation maps
\begin{align*}
\ev_i: \cM(\by,A_0,\tang,\bP) & \to \EuS^\circ_{k,\ell} \times V_{K_i} \\
(r,u) & \mapsto (z_i(r),u(z_i))
\end{align*}
and
\begin{align*} 
(\pi_{\EuS^\circ_{k,\ell}},\mathrm{ev}_1):\cM_{\Gamma_i}(\bY_{k,\ell}) & \to \EuS^\circ_{k,\ell} \times V_{K'_i}
\end{align*}
(see Section \ref{sec:param}). Here `$\mathrm{ev}_1$' denotes evaluation at the point $z_1$. Recall (cf. Definition \ref{defn:pert}) that $\bY_{k,\ell}$ is a family of infinitesimal almost complex structures parametrized by $\EuS^\circ_{k,\ell}$; the map $\pi_{\EuS^\circ_{k,\ell}}$ records the parameter.
We put these together to define
\[ \cM(\by,A,\tang,\bP,\{\Gamma_i\}) := \cM(\by,A_0,\tang,\bP) \times \prod_{i \in I} \cM_{\Gamma_i}(\bY_{k,\ell}),\]
with an evaluation map
\[ \ev: \cM(\by,A,\tang,\bP,\{\Gamma_i\}) \to \prod_{i \in I} V_{K_i} \times V_{K'_i} \times \EuS^\circ_{k,\ell} \times \EuS^\circ_{k,\ell}.\] 
We define the following subset of the RHS:
\[ \Delta := \prod_{i \in I} \left\{(x_i,x'_i, s ,s') \in V_{K_i} \times V_{K'_i} \times \EuS^\circ_{k,\ell} \times \EuS^\circ_{k,\ell}: x_i=x'_i, \; s = s' \right\}.\]
Now define the moduli space of pseudo-holomorphic discs with tangencies and unmarked bubble-trees
\begin{equation}\label{eqn:moduli with tangencies and bubbles} 
\cM_{\{\Gamma_i\}}(\by,A,\tang,\bP) := \ev^{-1}(\Delta).
\end{equation}
We denote by 
\begin{equation}\label{eqn:simple moduli with tangencies and bubbles} 
\cM^*_{\{\Gamma_i\}} \subset \cM_{\{\Gamma_i\}}
\end{equation}
the subset consisting of discs with unmarked bubble trees attached, each of which is simple.

\begin{rmk}\label{rmk:bubbtree}
Let us clear up a potential confusion. A disc with unmarked bubble trees attached will typically \emph{not} be the Gromov limit of a family of discs. 
Rather, a Gromov limit of discs will have spherical components contain interior marked points at which tangency conditions to $V$ are imposed; a disc with bubble trees attached is obtained by forgetting these marked points, which is possible by our \textbf{(Constant on spheres)} assumption from Definition \ref{defn:pert} (cf. Lemma \ref{lem:nospheresdim1}).
\end{rmk}

In the notation of \cite[Section 6]{Cieliebak2007}, the tangency constraint on an element $u\in \cM(\by,A_0,\tang,\bP)$ can be formulated as 
\begin{equation}\label{eqn:jet vanishing}   
d^{\tang(i)_q}u({z_i}) \in T_{u(z_i)}V_q 
\end{equation}
for all $q \in Q$.

\begin{defn}\label{defn:regularity}
A disc $u\in \cM(\by,A_0,\tang,\bP)$ is called parametrized-regular if it is cut out transversely by the parametric Cauchy--Riemann equation together with the jet-vanishing condition (\ref{eqn:jet vanishing}).  An element of $\cM^*_{\{\Gamma_i\}}(\by,A,\tang,\bP)$ is called regular if the disc is parametrized-regular, the bubble trees are parametrized-regular (cf. Section \ref{sec:param}), and the evaluation map is transverse to $\Delta$. 
The moduli space is called regular if all of its points are regular. 
A universal choice of perturbation data is called regular if all moduli spaces having negative virtual dimension are regular (i.e., empty); and furthermore, all moduli spaces $\cM(\by,A,\bP)$ of virtual dimension $\le 1$ are regular.
\end{defn}

\begin{lem}\label{lem:genreg}
Let $\by = (y_0,\ldots,y_k)$ be a choice of Hamiltonian chords corresponding to a Lagrangian labelling $\bL$, $A_0 \in \pi_2^\num(\by)$, $\tang:\{1,\ldots,\ell\} \to \Z_{\ge 0}^Q$ a choice of tangency data, and $\{\Gamma_i\}_{i \in I}$ a set of combinatorial types of bubble trees as above. Let $k+2\ell = N$. Suppose we have fixed $\bP_{k',\ell',\bL'}$ for all $2 \le k'+2\ell' < N$ and all Lagrangian labellings $\bL'$;  and suppose that these satisfy the conditions required of a universal choice of perturbation data. 
Then the set of $\bP_{k,\ell,\bL}$ satisfying those conditions, and such that $\cM^*_{\{\Gamma_i\}}(\by,A,\tang,\bP)$ is regular, is comeagre.
\end{lem}
\begin{proof}
The argument is essentially standard. 
In the absence of higher tangency conditions (i.e., if $\tang(i) \le 1$ for all $i$), it is a combination of \cite[Section 9k]{Seidel:FCPLT} with \cite[Section 6.7]{mcduffsalamon} (modified as in Section \ref{sec:spheres} to account for the divisor $V$). The necessary techniques to deal with higher tangency conditions were developed in \cite[Section 6]{Cieliebak2007}, and implemented in a setting very similar to ours in \cite[Theorem 4.19]{CharestWoodward} and \cite[Theorem 4.1]{CW:floer}. 
The only slightly non-standard aspect of our setup is our requirement that the perturbation data be chosen to be $\Sym(\ell)$-equivariant.

We start by choosing a countable open cover of $\EuR_{k,\ell}$ by open sets $C_i$, each of which comes equipped with a smooth trivialization $\EuS^\circ_{k,\ell}|_{C_i} \simeq C_i \times \Sigma_i^\circ$ with respect to which the strip-like ends and marked points are constant. 

We fix one of the open sets $C_i$ in our cover, and drop the $i$ from the notation. 
We will prove regularity of $\cM^*_{\{\Gamma_i\}}(\by,A,\tang,\bP)|_C$ (i.e., the set of all $(r,u,\{u_i\}) \in \cM^*_{\{\Gamma_i\}}(\by,A,\tang,\bP)$ such that $r \in C$), for a comeagre subset of the set of perturbation data $\bP$; taking the intersection over the countable set of sets in our cover yields the result. 

We will consider maps of Sobolev class $W^{q,p}$, and perturbation data of class $C^s$, where $q-\frac{2}{p} \ge \max_{i,j}\{\tang(i)_j\}$ and $s \ge q+1$. 
We define $\mathrm{Map}^{q,p}(\Sigma^\circ,X,\bL,V)$ to be the set of maps $u:\Sigma^\circ \to X$ of Sobolev class $W^{q,p}$, satisfying the boundary constraints $\bL$, converging to the Hamiltonian chords $\by$ in $W^{q,p}$-sense along the strip-like ends, and having the prescribed orders of tangency at the marked points.  
We set 
$$
\EuB^{q,p}_C := \mathrm{Map}^{q,p}(\Sigma^\circ,X,\bL,V) \times \prod_{1 \le i \le \ell, \alpha \in T(\Gamma_i)} \mathrm{Map}^{q,p}(\CP^1_\alpha,V_{K_\alpha}).$$
By \cite[Lemma 6.5]{Cieliebak2007}, it carries a natural structure of Banach manifold.
Let $\mathbb{P}^s=\mathbb{Y}^s \times \mathbb{K}^s$ denote the space of all $\bP_{k,\ell,\bL}=(\bY_{k,\ell,\bL},\bK_{k,\ell,\bL})$ satisfying the conditions required of a universal choice of perturbation data, but having class $C^s$ instead of $C^\infty$. 
We have a $C^{s-q}$ Banach vector bundle $\EuE_C \to C \times \EuB_C^{q,p} \times \mathbb{P}^s$, together with a $C^{s-q}$ Fredholm section $\bar{\partial}$. There is also an evaluation map 
$$\mathrm{ev}^E: \EuB_C^{q,p} \times \prod_{1 \le i \le \ell} Z(\Gamma_i) \to \prod_{1 \le i \le \ell} V_{K_i} \times \prod_{1 \le i \le \ell} (X,V)^{\Gamma_i},$$
together with a natural diagonal subset $\Delta$ of the target. 
We define 
$$\widetilde{\cM}_{\{\Gamma_i\}}^{q,p}(\by,A,\tang,\mathbb{P}^s)|_C \subset C \times \EuB^{q,p}_C \times \mathbb{P}^s \times \prod_{1 \le i \le \ell} Z(\Gamma_i)$$
(recalling the notation $Z(\Gamma)$ from \eqref{eq:ZGamma}) 
to be the subset cut out by the equations $\bar{\partial}(r,\mathbf{u},\bP) = 0$ and $\mathrm{ev}^E(\mathbf{u},\mathbf{z}) \in \Delta$, and 
$\cM^{q,p}_{\{\Gamma_i\}}(\by,A,\tang,\mathbb{P}^s)|_C $ to be its quotient by the reparametrization group $\prod_{1 \le i \le \ell} G_{\Gamma_i}$.
The fibre of $\cM_{\{\Gamma_i\}}^{q,p}(\by,A,\tang,\mathbb{P}^s)|_C$ over $\bP \in \mathbb{P}^s$ is precisely $\cM_{\{\Gamma_i\}}^{q,p}(\by,A,\tang,\bP)|_C$. 
By definition, it is regular if and only if the corresponding fibre $\widetilde{\cM}_{\{\Gamma_i\}}^{q,p}(\by,A,\tang,\bP)|_C$ is cut out transversely.

We now argue that the universal moduli space $\widetilde{\cM}_{\{\Gamma_i\}}^{q,p,*}(\by,A,\tang,\mathbb{P}^s)|_C$ (of pseudoholomorphic discs with \emph{simple} bubble trees attached) is cut out transversely, for each $C$. 
We focus on the only slightly non-routine step, which is regularity of the disc component. 

Transversality is a local condition at each point $(r,u,\bY,\bK)$. 
Given such a point, we consider the space $\mathbb{Y}^s_r$ of $C^s$ domain-dependent almost-complex structures on the fibre $\Sigma^\circ_r$, satisfying the conditions required to constitute part of a perturbation datum for that fibre. 
By \cite[Lemma 6.5]{Cieliebak2007}, in order to prove regularity, it suffices to show that there exists a non-empty open set $W \times U \subset \Sigma^\circ \times X \setminus V$, such there exists a point $z \in W$ with $u(z) \in U$, and the image of the linearization of the restriction map $\mathbb{P}^s  \to \mathbb{Y}^s_r$ contains all perturbations $\delta Y$ which are supported in $W \times U$. 

To arrange this, we first choose a neighbourhood $N$ of $r \in \EuR_{k,\ell}$, such that $\sigma(N) \cap N = \emptyset$ for all $\sigma \in \Sym(\ell)$. 
This is possible by Lemma \ref{lem:free}. 
We choose $W \subset \Sigma^\circ$ to be the complement of the thin regions and marked points, and $U$ to be equal to $\{h<c\}$. 
Note that $W \cap \partial \Sigma^\circ$ is non-empty, and gets mapped into $U$; thus there exists $z \in W$ with $u(z) \in U$, as required. 

Now, given a perturbation $\delta Y$ supported in $W \times U$, we explain how to construct a perturbation $\delta \bP = (\delta \bK, \delta \bY)$ such that the restriction of $\delta \bY$ to $\Sigma^\circ_r$ is $\delta Y$. 
First, we choose a bump function $\eta$ with support in $N$, so that $\eta(r) = 1$; then we consider the perturbation $\eta(n) \cdot \delta Y$ with support in $N \times W \times U$; then we set
\[ \delta \bY := \sum_{\sigma \in \Sym(\ell)} \sigma^* \left(\eta(n) \cdot \delta Y\right).\]
We claim that $\bY + \epsilon \cdot \delta \bY \in \mathbb{Y}^s$ for $\epsilon$ sufficiently small: it clearly lies in $\EuY_*$ for $\epsilon$  sufficiently small, because $\delta \bY$ is compactly supported so its $C^0$ norm is bounded; it extends to $\Sbar^\circ_{k,\ell}$ and satisfies the \textbf{(Consistent on discs)} and \textbf{(Constant on spheres)} hypotheses, because it has compact support in $\EuS^\circ_{k,\ell}$ and in particular vanishes on the boundary; it satisfies the \textbf{(Thin regions)} hypothesis because $W$ is disjoint from the thin regions, hence so is its orbit under $\Sym(\ell)$ (by equivariance of the thin regions); it satisfies the \textbf{(Equivariant)} hypothesis by construction. 
The image of $\partial/\partial \epsilon(\bY + \epsilon \cdot \delta \bY)$ under the linearization of the restriction map $\mathbb{Y}^s \to \mathbb{Y}^s_r$ is equal to $\delta Y$, because $\left(\eta(n) \cdot \delta Y\right)|_{\Sigma^\circ_r} = \delta Y$ by construction, but $\sigma^* \left(\eta(n) \cdot \delta Y\right)|_{\Sigma^\circ_r} = 0$ for $\sigma$ not equal to the identity, using the fact that $\sigma(N) \cap N = \emptyset$.

Having established regularity of the universal moduli space, we can apply Sard--Smale to the projection $\widetilde{\cM}^{q,p,*}_{\{\Gamma_i\}}(\by,A,\tang,\mathbb{P}^s)|_C \to \mathbb{P}^s$. 
We conclude that $\widetilde{\cM}^{q,p,*}_{\{\Gamma_i\}}(\by,A,\tang,\bP)|_C$ (and hence, by definition, $\cM^{q,p,*}_{\{\Gamma_i\}}(\by,A,\tang,\bP)|_C$) is regular for a comeagre subset of the set of choices of $C^s$ perturbation data $\mathbb{P}^s$, for each $C$. 
It follows by an argument due to Taubes (as in \cite[Proof of Theorem 3.5.1(ii)]{mcduffsalamon}) that $\cM^*_{\{\Gamma_i\}}(\by,A,\tang,\bP)|_C$ is regular for a comeagre subset of the set of choices of smooth perturbation data. Taking the intersection over all $C$, we conclude that $\cM^*_{\{\Gamma_i\}}(\by,A,\tang,\bP)$ is regular for a comeagre subset of the set of smooth perturbation data, as required.
\end{proof}

\begin{cor}\label{cor:regex}
There exists a regular universal choice of perturbation data.
\end{cor}
\begin{proof}
In fact, we can construct a universal choice of perturbation data such that \emph{all} moduli spaces $\cM^*_{\{\Gamma_i\}}(\by,A,\tang,\bP)$ are regular; such a choice is \emph{a fortiori} regular. 
The choice can be constructed inductively in $k+2\ell$. At each step, there is a nonempty space of (possibly non-regular) choices of $\bP_{k,\ell,\bL}$, by Lemma \ref{lem:pertinduct}. 
The subset of regular choices is the intersection of the subsets making $\cM^*_{\{\Gamma_i\}}(\by,A,\tang,\bP)$ regular for all compatible $\by,A,\tang,\{\Gamma_i\}$. 
This is a countable intersection of subsets, each of which is comeagre by Lemma \ref{lem:genreg}; thus it is itself a comeagre subset of a nonempty Baire space, in particular it is nonempty. 
\end{proof}

\begin{lem}\label{lem:bubtreedim}
If $\cM^*_{\{\Gamma_i\}}(\by,A,\tang,\bP)$ is regular, it has dimension
\begin{equation}
\label{eq:dimdisc}
 \mathrm{dim} = i(A_0) + k-2+ 2\ell - 2 \sum_{i =1}^\ell |\tang(i)| + \sum_{i \in I} \mathrm{dim}(\Gamma_i) - 2n + 2|K_i \cap K'_i|.
 \end{equation}
Here $|\tang(i)| = \sum_{q\in Q} \tang(i)_q$.
\end{lem}
\begin{proof}
The dimension count goes as follows:
\[  \underbrace{i(A_0) + k-2+ 2\ell - 2 \sum_{i =1}^\ell |\tang(i)|}_{\dim \cM(\by,A_0,\tang,\bP)}  +\underbrace{  \sum_{i \in I} \left(  \dim(\Gamma_i) + \dim \EuS_{k,l}^\circ \right) }_{\sum_i \dim \cM_{\Gamma_i}(\mathbf{Y}_{k,\ell})}  - \underbrace{\sum_{i\in I} \left( 2n -  2|K_i \cap K'_i| +  \dim \EuS_{k,l}^\circ  \right) }_{\mathrm{codim}\, \Delta} . \]
The formula follows from the assumption that $\ev$ is transverse to $\Delta$.
\end{proof}

\begin{lem}\label{lem:dimboundreg}
If $\tang = \tang^{\can}$ and $I=\emptyset$, then the dimension \eqref{eq:dimdisc} is equal to 
\[ \mathrm{dim}(\cM(\by,A,\bP)) = i(A) +k-2.\] 
In any other case, it is bounded above by $\mathrm{dim}(\cM(\by,A,\bP))-2$.
\end{lem}
\begin{proof}
Applying Lemmas \ref{lem:bubtreedim} and \ref{lem:dimineq} (we set $k=1$ in the latter, because by definition all bubble trees carry only a single marked point, namely the one at which they are attached to the disc), we find
\begin{align*}
\dim & \leq  i(A) + k - 2 + \sum_{i=1}^\ell 2(1-|\tang(i)|) + \sum_{i \in I} 2(|K_i \cap K'_i| - \max_{\alpha \in T_i} |K_\alpha| - 2)  \\
& \le \dim(\cM(\by,A,\bP)) + \sum_{i=1}^\ell 2(1-|\tang(i)| - 2I_i),
\end{align*}
where $I_i := 1$ if $i \in I$, $0$ if $i \notin I$. 
Because $I$ contains all $i$ such that $\tang(i) = 0$, we have $1-|\tang(i)| - 2 I_i \le 0$, with equality if and only if $|\tang(i)|=1$ and $I_i=0$. The result follows.
\end{proof}

\begin{cor}\label{cor:nospheres}
If $\bP$ is regular and $\mathrm{dim}(\cM(\by,A,\bP)) \le 1$, then $\cM_{\{\Gamma_i\}}(\by,A,\tang,\bP) = \emptyset$ unless $\tang=\tang^{\can}$ and $I=\emptyset$. 
\end{cor}
\begin{proof}
Observe that the subset $\cM^*_{\{\Gamma_i\}}(\by,A,\tang,\bP)$ of discs with simple bubble trees attached is empty by Lemma \ref{lem:dimboundreg}, since it is regular and has negative virtual dimension. 
On the other hand, one can always replace a non-simple bubble tree with a simple one (see \cite[Proposition 6.1.2]{mcduffsalamon}). The procedure involves replacing multiply-covered spheres by their image, deleting some spheres, and rearranging the marked points; in particular, it cannot increase the Chern number by Corollary \ref{cor:semipos}. It follows that the moduli spaces of discs with non-simple bubble trees attached are also empty.
\end{proof}

\begin{rmk}
One would expect a similar result to hold for $\cM(\by,A,\tang,\bP)$, if one changed the definition of the bubble trees to incorporate higher tangency conditions (otherwise the analogue of Lemma \ref{lem:dimboundreg} would not hold).
\end{rmk}

\subsection{Compactness}

Because we chose our perturbation data $\bP_{k,\ell}$ to extend smoothly to $\Sbar^\circ_{k,\ell}$, the moduli space $\cM(\by,A,\tang,\bP)$ of discs with tangency conditions admits a Gromov compactification $\Mbar^{\mathrm{tot}}(\by,A,\tang,\bP)$. Let 
\begin{equation}\label{eqn:Gromov closure}
\Mbar(\by,A,\tang,\bP) \subset \Mbar^{\mathrm{tot}}(\by,A,\tang,\bP)
\end{equation}
denote the closure of the subset $\cM(\by,A,\tang,\bP)$.  Write 
\begin{equation}\label{eqn:Gromov closure can tangency}
\Mbar(\by,A,\bP)
\end{equation} 
to mean $\Mbar(\by,A,\tang^{\mathsf{can}}, \bP)$.

\begin{lem}\label{lem:mustintersect}
If $u$ is a disc component of an element of $\Mbar(\by,A,\tang,\bP)$, then any point at which $u$ intersects $V$ is special, i.e., an internal marked point or node. 
\end{lem}
\begin{proof}
Follows from Lemma \ref{lem:posintCM}, as isolated intersection points persist under small perturbations.  
\end{proof}

Let us call a disc with $k+1$ boundary marked points and $\ell$ interior marked points `stable' if $k+2\ell \ge 2$, `semistable' if $k+2\ell \ge 1$, and `unstable' otherwise.

\begin{cor}\label{cor:nounstable}
Every disc component of an element of $\Mbar(\by,A,\bP)$ is semistable.
\end{cor}
\begin{proof}
Follows from Lemmas \ref{lem:mustintersect} and \ref{lem:intersectv}.
\end{proof}

\begin{lem}\label{lem:nospheresdim1}
Suppose $\bP$ is regular, and $\dim(\cM(\by,A,\bP)) \le 1$. 
Then elements of $\Mbar(\by,A,\bP)$ have no (constant or non-constant) sphere bubbles.
\end{lem}
\begin{proof}
Given an element $u \in \Mbar(\by,A,\bP)$, let us cut it along its boundary nodes. 
By Corollary \ref{cor:nounstable}, each disc component is semistable; therefore, by the \textbf{(Consistent on discs)} condition on $\bP$, each component defines an element $u_\alpha$ of some $\Mbar(\by_\alpha,A_\alpha,\bP)$, without boundary nodes. 
We have
\[ \sum_\alpha \dim(\cM(\by_\alpha,A_\alpha,\bP)) + \#\, (\text{boundary nodes}) = \dim(\cM(\by,A,\bP))  \]
by Lemma \ref{lem:dimboundreg}. It follows from the assumption $\dim(\cM(\by,A,\bP)) \leq 1$, then, that $\dim(\cM(\by_\alpha,A_\alpha,\bP) \le 1$ for each $\alpha$.  Thus we may assume, without loss of generality, that $u$ has no boundary nodes.

Let $u_\disc$ be the disc component of $u$. 
By Lemma \ref{lem:mustintersect}, every intersection point of $u_\disc$ with $V$ is a marked point or node. 
We attach the tangency data $\tang(i) := \iota(u_\disc,z_i)$ to each marked point or node $z_i$, so that $u_\disc$ defines an element of some $\cM(\by,A_0,\tang,\bP)$ by the \textbf{(Consistent on discs)} condition on $\bP$. 
By the \textbf{(Constant on spheres)} condition on $\bP$, forgetting the internal marked points on the sphere bubbles of $u$ and collapsing unstable components gives a set of bubble trees attached to the marked points of $u_\disc$. 
In order to check that it defines an element of some $\cM_{\{\Gamma_{i}\}}(\by,A,\tang,\bP)$, we need to check that the subset $I$ of all points with non-trivial bubble trees attached, contains all $i$ such that $\tang(i)=0$; or equivalently, if the bubble tree attached to the $i$th point is trivial, then $\tang(i) \neq 0$. 
Indeed, if the bubble tree attached to point $i$ becomes trivial after forgetting all marked points, it must originally have been constant; therefore $\tang(i)$ is greater than or equal to the sum of $\tang^{\can}(j)$ over all points $j$ on the bubble tree, by \cite[Lemma 7.2]{Cieliebak2007}. 
Since $\tang^{\can}(j) \neq 0$ for all $j$, this implies that $\tang(i) \neq 0$ as required. 

Thus we obtain an element of some $\cM_{\{\Gamma_{i}\}}(\by,A,\tang,\bP)$, so we must have $\tang=\tang^{\can}$ and $I=\emptyset$ by Corollary \ref{cor:nospheres}. 
Because $I=\emptyset$, we conclude that all sphere bubbles on $u$ become trivial after forgetting the marked points they contain, so they must be constant.  
In order for a constant tree of sphere bubbles to be stable, it must contain at least two marked points in addition to the point at which it is attached to the disc; however then the tangency vector of the point $i$ at which the tree of sphere bubbles is attached to the disc would have to be greater than or equal to a sum of $\tang^{\can}(j)$ over at least two points $j$, so could not possibly be equal to $\tang^{\can}(i)$. 
Therefore $u$ contains no sphere bubbles. 
\end{proof}

\section{The relative Fukaya category}\label{sec:def}

Our definition of the relative Fukaya category depends on the following data: 

\begin{enumerate}
\item \label{it:1}Closed manifold $X$; subdomain $W \subset X$; symplectic form $\omega$ on $X$; Liouville one-form $\theta$ on $W$; $\omega$-compatible almost complex structure $J_0$ such that $(W \subset X,J_0)$ admits a convex collar.
\item \label{it:2}System of divisors $V$; choice of cylindrical and strip-like ends and thick-thin decomposition; choice of regular Floer data; choice of regular perturbation data.
\end{enumerate}

We will address the dependence of the relative Fukaya category on the data \eqref{it:1} in Section \ref{sec:dep2}, and its dependence on the data \eqref{it:2} in Section \ref{sec:indep}. 
Pending those results, we denote the relative Fukaya category by $\fuk(W \subset X,\omega, \theta, V, J_0,\bN)$, where $\bN$ denotes the choice of cylindrical and strip-like ends, thick-thin decomposition, regular Floer data, and regular perturbation data $\bP$. 

The coefficient ring of the relative Fukaya category is $R:= \Z\power{\NE(V)}$, see Section \ref{sec:symp_rel_fuk}. 
The objects are Lagrangian branes in $W$, see Definition \ref{defn:obj}. 

We associate an `orientation line' $o_y$ to each Hamiltonian chord $y$ from $L_0$ to $L_1$ following \cite[Appendix B]{Sheridan2016} (which is based on \cite[Section 12f]{Seidel:FCPLT}). 
An orientation line is a $\Z/2$-graded free $R$-module of rank $1$.

\begin{defn}\label{defn:mor}
The $hom$-spaces in the relative Fukaya category are the direct sums of all orientation lines associated to Hamiltonian chords:
\[ hom^*(L_0,L_1) := \bigoplus_y o_y.\]
They are $\Z/2$-graded free $R$-modules.
\end{defn}

The $A_\infty$ structure maps are defined using the moduli spaces $\cM(\by,A,\bP)$, for a regular choice of perturbation data $\bP$.  
Each zero-dimensional moduli space $\cM(\by,A,\bP)$ is then finite (because compact). 
Each element $u$ of such a moduli space determines an isomorphism
\begin{equation}
\label{eq:signu}
\sigma_u: o_{y_1} \otimes \ldots \otimes o_{y_k} \xrightarrow{\sim} o_{y_0},
\end{equation}
in accordance with the conventions of \cite[Appendix B]{Sheridan2016}. 
Observe that $\Sym(\vv)$ acts on $\cM(\by,A,\bP)$ because of the \textbf{(Equivariant)} condition on our perturbation data; and the action is free by Lemma \ref{lem:free}. 
It is easy to see from the definition that $\sigma_{g \cdot u} = \sigma_u$ for any $g \in \Sym(\vv)$.\footnote{The key point is that the moduli spaces $\EuR_{k,\ell}$ are oriented in such a way that the fibres of the forgetful maps $\EuR_{k,\ell} \to \EuR_k$ carry their natural complex orientation (one can make sense of this even when $k \le 2$); $\Sym(\vv)$ acts trivially on the base and by a complex automorphism of the fibre, hence it preserves orientations.}
This allows us to make the following

\begin{defn}\label{def:ainf}
We define the $A_\infty$ structure map
\begin{align}
\nonumber \mu^k &: hom^*(L_0,L_1) \otimes_R \ldots \otimes_R hom^*(L_{k-1},L_k) \to hom^*(L_0,L_k),\\
\label{eq:mukdef}
 \mu^k &:= \sum_{u \in \cM(\by,A,\bP)/\Sym(\vv)} \nov^{[u]} \cdot \sigma_u,
 \end{align}
where the sum is over all $\Sym(\vv)$-orbits of zero-dimensional moduli spaces $\cM(\by,A,\bP)$.\end{defn}

Observe that $[u] \in \NE(V)$ by Lemma \ref{lem:posint}, so the sum lies in the coefficient ring $R = \Z\power{\NE(V)}$.
We also observe that $\mu^0$ is of order $\fm$ by definition: we at no point consider discs with a single boundary marked point and no interior marked point (we are able to do this by Corollary \ref{cor:nounstable}). 

Finally, $\mu^* \circ \mu^* = 0$ by the standard argument \cite[Section 12g]{Seidel:FCPLT}. 
The only boundary strata appearing in a one-dimensional moduli space $\Mbar(\by,A,\bP)$ consist of stable discs with two disc components and no spheres, by Lemma \ref{lem:nospheresdim1}. 
The moduli space has the structure of a one-dimensional manifold with boundary in a neighbourhood of these points, by a standard gluing argument.   The action of $\Sym(\vv)$ on such boundary strata remains free (cf. the two remarks following Lemma \ref{lem:free}), and so the moduli space of $\Sym(\vv)$-orbits is again a one-dimensional manifold with boundary.  The signed count of boundary points of a compact one-dimensional manifold with boundary is zero; and these boundary points correspond precisely to the terms in the expression $\mu^* \circ \mu^*$.

\begin{proof}[Proof of Proposition \ref{prop:5.1}:]
It is immediate from Lemma \ref{lem:maxprin} and the \textbf{(Maximum principle)} condition on our choice of perturbation data that
$$\fuk(W \subset X,\omega,\theta,V,J_0,\bN) \otimes_R R/\fm \simeq \fuk(W).$$ 
\end{proof}

\section{Dependence on perturbation data and system of divisors}\label{sec:indep}

For the purposes of this section, we fix the data $W \subset X, \omega, \theta, J_0$, and elide it from the notation for the relative Fukaya category.
We consider the dependence of the relative Fukaya category on the system of divisors $V$, and the data $\bN$. 
Note that the choice of $V$ imposes conditions on $\bN$ (but not vice-versa): we denote the set of $V$-compatible choices of $\bN$ by $\EuN(V)$.

Given a system of divisors $\{V_q\}_{q \in Q_1}$, a \emph{sub-system of divisors} is a subset $Q_0 \subset Q_1$ such that $\{V_q\}_{q \in Q_0}$ is also a system of divisors. 
Note that $\NE(V_1) \subset \NE(V_0)$, so there is a natural map $R(V_1) \to R(V_0)$.

\begin{prop}\label{prop:indep}
Suppose that $V_0 \subset V_1$ is a sub-system of divisors. 
Then for any $\bN^1 \in \EuN(V_1)$, and a non-empty subset of $\bN^0 \in \EuN(V_0)$, there exists a curved filtered $R(V_1)$-linear $A_\infty$ category $\EuC$ which admits strict curved filtered quasi-equivalences
\begin{equation}
\label{eq:embeds}
    \fuk(V_1,\bN^1)   \hooklongrightarrow \EuC \hooklongleftarrow \fuk(V_0,\bN^0)\otimes_{R(V_0)}R(V_1).
 \end{equation}
If $V_0=V_1$, then the the subset of the regular $\bN^0$ consists of all of $\EuN(V_0)$.
\end{prop}

Note that the case $V_0 = V_1$ of the Proposition shows that $\fuk(V,\bN)$ is independent of $\bN$ up to curved filtered quasi-equivalence; we denote this curved filtered quasi-equivalence class by $\fuk(V)$. 
The full statement of the proposition then yields a curved filtered quasi-equivalence
\[ \fuk(V_1) \simeq \fuk(V_0) \otimes_{R(V_0)} R(V_1).\]

\begin{rmk}
We do not establish any kind of uniqueness of the curved filtered quasi-equivalences involved: we have a `weak system' of (curved filtered quasi-)equivalent categories, in the terminology of \cite[Section 10a]{Seidel:FCPLT}.
\end{rmk} 

The basic idea behind the construction of $\EuC$ is to show that $\fuk(V_1,\bN^1)$ can be defined using the pullback of the data $\bN^0$ via the map forgetting all marked points labelled by $q \in Q_1 \setminus Q_0$. 
The difficulty lies in the fact that the pulled-back data will not be valid: e.g., the pulled-back perturbation data will never satisfy the conditions \textbf{(Thin regions)}, \textbf{(Equivariant)}, \textbf{(Consistent on discs)} required of a universal choice of perturbation data. 
Therefore, we enlarge the allowed class of perturbation data so that it includes the pullback of $\bN^0$, as well as the perturbation data $\bN^1$; then a doubled category construction allows us to interpolate between the two.

\subsection{Construction of the doubled category: domain data}

We give the definition of the doubled category $\EuC$ appearing in Proposition \ref{prop:indep}. 
The objects are Lagrangian branes $L^i$, equipped with a label $i=0$ or $1$. 
On the level of objects, the left embedding in \eqref{eq:embeds} send $L \mapsto L^0$, while the right embedding sends $L\mapsto L^1$.
A \emph{Lagrangian labelling} $\bL$ will be a tuple of objects $(L_0^{i_0},\ldots,L_k^{i_k})$ of $\EuC$. 
Forgetting the Lagrangians $L_j$, we obtain a \emph{0/1 labelling} $\bi(\bL) = (i_0,\ldots,i_k)$. 

\begin{defn}\label{def:forgmods}
Suppose that $2 \le k+2\ell$, $\ell_\for \le \ell$, $K_\pun \subset \{1,\ldots,k\}$ is a subset, and $\bi = (i_0,\ldots,i_k)$ is a 0/1 labelling, having the property that whenever $j \notin K_\pun$, we have $i_{j-1}=i_j = 0$. 
We define $\EuS_{k,K_\pun,\ell,\ell_\for,\bi} \to \EuR_{k,K_\pun,\ell,\ell_\for,\bi}$ to coincide with $\EuS_{k,\ell} \to \EuR_{k,\ell}$. 
We define $K^0_\pun \subset \{0,1,\ldots,k\}$ such that $K^0_\pun \cap \{1,\ldots,k\} = K_\pun$, and $0 \notin K^0_\pun$ if and only if $K_\pun$ is empty and $\ell=\ell_\for$. 
We define $\EuS^\circ_{k,K_\pun,\ell,\ell_\for,\bi}$ to be the complement of all boundary marked points labelled by $j \in K_\pun^0$. 
\end{defn}

We now explain how to `think about' Definition \ref{def:forgmods}, and preview the main construction. 
The moduli spaces $\EuR_{k,K_\pun,\ell,\ell_\for,\bi}$ parametrize discs with $k+1$ boundary marked points and $\ell$ interior marked points. 
We call the last $\ell_\for$ of the internal marked points `forgettable', and the remaining ones `unforgettable'. 
One should think of $\bi$ as a 0/1 labelling of the boundary intervals between the boundary marked points. 

We call the boundary marked points corresponding to $j \in K_\pun^0$ `puncturable', and the remaining ones `unpuncturable'. 
If a boundary point is unpuncturable, then the two boundary intervals on either side of it must be labeled $0$.
The output boundary marked point is unpuncturable if and only if all input boundary marked points are unpuncturable and all interior marked points are forgettable. 

We will consider pseudoholomorphic discs $u\in \cM(\by,A,\tang,\bP$), cf. (\ref{eqn: tang moduli}), whose domain is a fibre $\Sigma^\circ_r$ of the family $\EuS^\circ_{k,K_\pun,\ell,\ell_\for,\bi}$, with the following properties:
\begin{itemize}
\item
The disc $u$ maps unforgettable internal marked points to $V_0$, and forgettable ones to $V_q$ for $q \in Q_1 \setminus Q_0$. 
\item
The disc $u$ limits to Hamiltonian chords at the puncturable boundary marked points, which are removed from $\Sigma_r$ to obtain $\Sigma^\circ_r$. 
\item
We have required that unpuncturable boundary marked points $z_j$ always separate two boundary components carrying the label $0$. These boundary components will furthermore \emph{be labeled by the same Lagrangian}: $L_{j-1}=L_j$. Unpuncturable boundary points $z_j$ are not removed, and $u$ will extend smoothly across $z_j$. 
\item
Our perturbation data will be constructed to be compatible with the map forgetting the forgettable marked points, over certain subsets of the space of domains. It will similarly be compatible with the map forgetting the unpuncturable marked points.
\end{itemize}

\paragraph{The role of unpuncturable marked points.}
The $A_\infty$ structure maps in the doubled category $\EuC$ will be defined via pseudoholomorphic discs with only puncturable boundary marked points. 
Nevertheless, unpuncturable marked points play an unavoidable role in the construction: they are introduced to deal with boundary nodes which are not mapped to a Hamiltonian chord. 
We are forced to allow for such a possibility in order to construct a Gromov compactification of our moduli spaces, although we will subsequently exclude it \emph{a posteriori}, as we now explain. 

The following are two fundamental aspects of the construction of our moduli spaces of pseudoholomorphic curves:
\begin{itemize}
\item the perturbation data $\bP$ on $\EuS^\circ$ should extend smoothly over the boundary $\Sbar^\circ$, so that a Gromov compactification exists;
\item over moduli spaces all of whose boundary components are labelled `$0$', the perturbation data $\bP$ must be pulled back from the perturbation data $\bP^0$ from $\bN^0$, via the map $\ff$ which `forgets' all of the forgettable marked points.
\end{itemize}
Combining these two, we see that the perturbation data over the boundary $\partial \Sbar^\circ$ of a moduli space all of whose boundary components are labelled `$0$' must also be pulled back via the forgetful map. 
Now let us consider a stable disc in $\Rbar$ which has a disc component $\mathbb{D}$ with a single boundary marked point which is a node, and a single interior marked point which is forgettable. 
The forgetful map $\ff$ collapses $\mathbb{D}$ to a smooth boundary point $p$. 
In particular, the pulled back perturbation data $\ff^*\bP^0$ take the form $(\bY,\bK)$ where $\bK = 0$ over $\mathbb{D}$, and $\bY$ is constant equal to $\bY^0(p)$ over $\mathbb{D}$. 
This raises the following issues:
\begin{itemize}
\item The perturbation data over $\mathbb{D}$ cannot satisfy the \textbf{(Thin regions)} condition, as $\bK = 0$.
\item The perturbation data over $\mathbb{D}$ can only satisfy the \textbf{(Consistent on discs)} condition if $\bY^0$ is constant along the corresponding boundary component.
\end{itemize}
In particular, when considering pseudoholomorphic maps from this source, we are forced to consider the disc $\mathbb{D}$ as bubbling off at a node, rather than breaking along a Hamiltonian chord. 

In particular, we are forced to treat certain boundary marked points and components differently from others: the marked point described above is an example of an unpuncturable marked point (whereas in the previous sections we only ever considered puncturable marked points), and the disc $\mathbb{D}$ is an example of what we will call an `$\ff$-disclike' component in the sequel.

Although we are forced to introduce unpuncturable marked points in order to construct moduli spaces admitting Gromov compactifications, \emph{a posteriori}, the corresponding components of the Gromov compactification will actually be empty (Lemma \ref{lem:nounpunc}). For example, the disc $u(\mathbb{D})$ considered above must by non-constant, since it intersects $V_1$ at its interior marked point; but that implies that it must also intersect $V_0$ by Lemma \ref{lem:intersectv}, as $V_0$ is a system of divisors, and in particular must contain an additional, stabilizing marked point mapping to $V_0$. 

\paragraph{Deligne--Mumford compactification.}

Now we describe the structure of the Deligne--Mumford compactification of $\EuR_{k,K_\pun,\ell,\ell_\for,\bi}$. 
Of course we already have a description with strata isomorphic to products of $\EuR_{\ell_\alpha}$ and $\EuR_{k_\alpha,\ell_\alpha}$; the aim is to extend the $0/1$ labelling of boundary components, designation of interior marked points as forgettable/unforgettable, and designation of boundary marked points as puncturable/unpuncturable, to the components of a stable curve. 
This is necessary in order to specify the consistency conditions we require on our choices of perturbation data. We illustrate the following definitions in Figure \ref{fig:disk2}.
  
We define $\Sbar_{k,K_\pun,\ell,\ell_\for,\bi} \to \Rbar_{k,K_\pun,\ell,\ell_\for,\bi}$ to be equal to $\Sbar_{k,\ell} \to \Rbar_{k,\ell}$. 
For the following definitions, let $\Sigma_r$ be a fibre of $\Sbar_{k,K_\pun,\ell,\ell_\for,\bi} \to \Rbar_{k,K_\pun,\ell,\ell_\for,\bi}$.

\begin{defn}\label{def:whatsforg}
Let $z \in \Sigma_r$ be an interior marked point or node.  We define $z$ to be forgettable if it is a forgettable marked point, or if the tree of spheres lying above it contains only forgettable marked points; otherwise define it to be unforgettable.
\end{defn}

\begin{figure}
\begin{center}
\includegraphics[width=\textwidth]{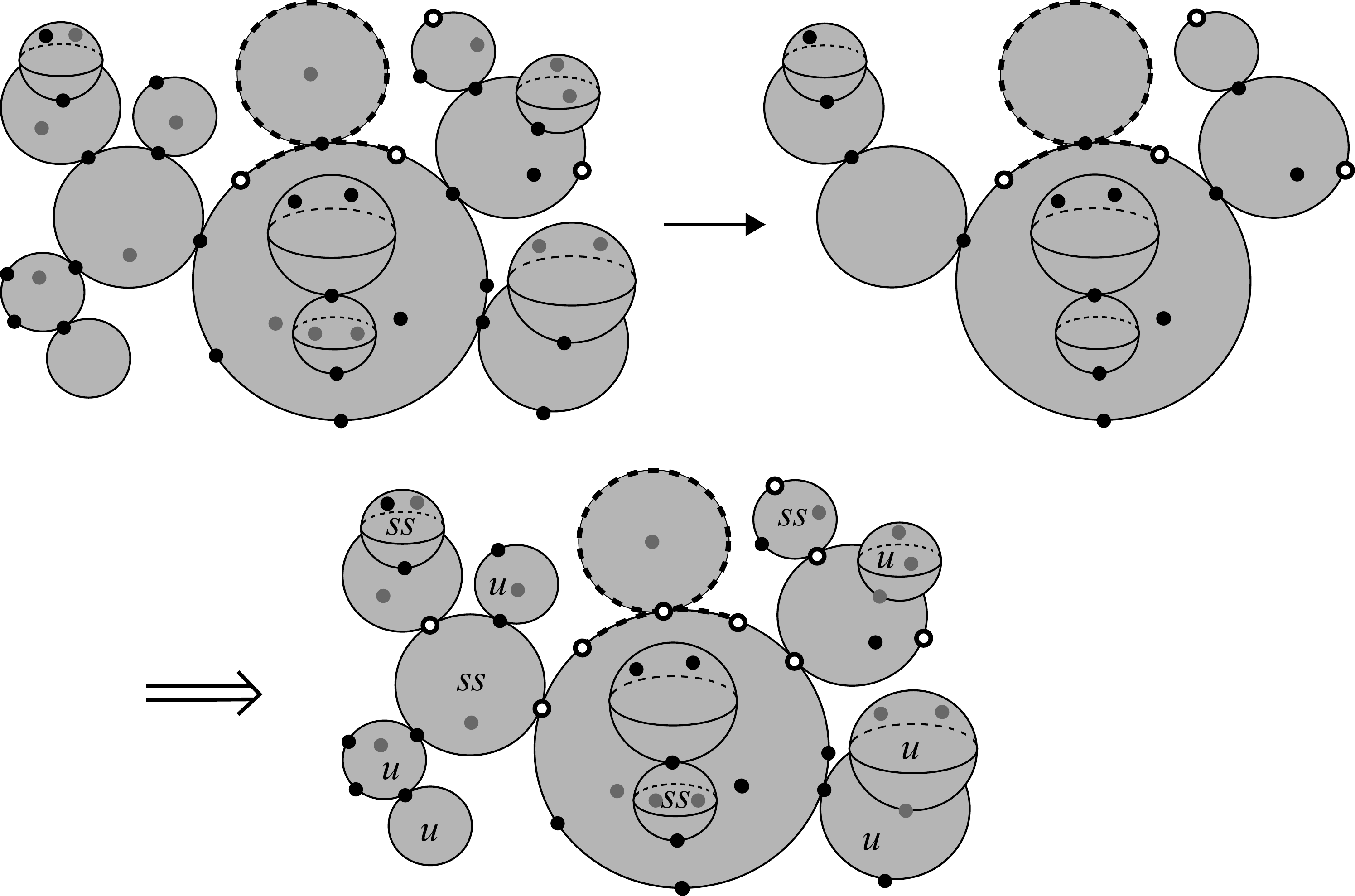}
\end{center}
\caption{\label{fig:disk2}The procedure for designating interior nodes as forgettable or unforgettable, and boundary nodes as puncturable or unpuncturable. On the top left we illustrate an element of $\Rbar_{k,K_\pun,\ell,\ell_\for,\bi}$. All nodes are drawn as a solid point; solid interior marked points are unforgettable; grey interior marked points are forgettable; open boundary marked points are puncturable; solid boundary marked points are unpuncturable; dashed boundary components are labelled `$1$'; all other boundary components are labelled `$0$'. On the top right, we illustrate the outcome of the $\ff$-process. On the bottom, we return to the original element of $\Rbar_{k,K_\pun,\ell,\ell_\for,\bi}$. We draw the interior marked points and nodes attached to a disc as solid if they correspond to an unforgettable marked point, and grey if they correspond to a forgettable marked point, in accordance with Definition \ref{def:whatsforg}. We label the components by `$u$' if they are unstable; `$ss$' if they are $\ff$-semistable; and leave them unlabelled if they are $\ff$-stable, in accordance with Definition \ref{def:fdisclike}. We draw the boundary marked points and nodes as open if they are puncturable, and solid if they are unpuncturable, in accordance with Definition \ref{def:whatspunc}.}
\end{figure}

Let us call the following procedure the `$\ff$-process'. Start with the stable curve $\Sigma_r$, and modify it as follows.

\begin{enumerate}
\item
Delete all forgettable interior marked points.

\item 
Delete all unpuncturable boundary marked points. 

\item
Collapse any discs with no internal special points, one boundary marked point, and whose unique boundary component is labelled `$0$'. The point to which the disc gets collapsed does not get marked.

\item
Collapse any spheres with no marked points and one node, which belong to a tree of spheres attached to a disc all of whose boundary components are labelled `$0$'. The point to which the sphere gets collapsed does not get marked.

\item 
Return to (3) and iterate.
\end{enumerate}

The iteration stops when there are no more spheres or discs to collapse. The output is a curve $\Sigma_r'$ without forgettable or unpuncturable marked points, which comes with a collapsing map $c\colon \Sigma_r \to \Sigma_r'$.

\begin{defn}\label{def:fdisclike}
The components of $\Sigma_r$ are partitioned into three types. 
Components which are collapsed in the $\ff$-process are called \emph{$\ff$-unstable}. Components whose image in the output curve $\Sigma_r'$ is either a disc with no internal marked points but exactly two boundary special points, separated by two boundary components both labelled `$0$', or a sphere with precisely two special points, belonging to a tree of spheres attached to a disc all of whose boundary components are labelled `$0$', are called \emph{$\ff$-semistable}. All other components are called \emph{$\ff$-stable}. 
\end{defn}

\begin{defn}\label{def:whatspunc}
Let $\zeta \in \Sigma_r$ be a boundary marked point or node. 
We define $\zeta$ to be unpuncturable if it is an unpuncturable marked point, or if it is a node lying on an $\ff$-unstable component; otherwise define it to be puncturable.
\end{defn}

As before, strata of $\Rbar_{k,K_\pun,\ell,\ell_\for,\bi}$ are indexed by certain decorated trees $\Gamma$. 
The conventions established in Definitions \ref{def:whatsforg} and \ref{def:whatspunc} (together with the obvious convention for determining boundary labellings $\bi_\alpha$ from $\bi$) determine isomorphisms
\begin{align}
\label{eq:Ridecomp} \EuR^\Gamma_{k,K_\pun,\ell,\ell_\for,\bi} &\simeq \prod_{\alpha \in V_\sph} \EuR_{\ell_\alpha} \times \prod_{\alpha \in V_\disc} \EuR_{k_\alpha,K_{\alpha,\pun},\ell_\alpha,\ell_{\alpha,\for},\bi_\alpha},\\
\label{eq:Sdecomp}
\EuS^\Gamma_{k,K_\pun,\ell,\ell_\for,\bi} & \simeq \left(\coprod_{\alpha \in V_\sph} \mathrm{pr}_\alpha^* \EuS_{\ell_\alpha} \sqcup \coprod_{\alpha \in V_\disc} \mathrm{pr}_\alpha^* \EuS_{k_\alpha,K_{\alpha,\pun},\ell_\alpha,\ell_{\alpha,\for},\bi_\alpha}\right)/\sim
\end{align}
(depending on an appropriate choice of ordering of the interior marked points on each component). 
It is straightforward to check that $(k_\alpha,K_{\alpha,\pun},\ell_\alpha,\ell_{\alpha,\for},\bi_\alpha)$ satisfies the conditions required by Definition \ref{def:forgmods} in order for the corresponding moduli spaces to be defined.

We define $\Sbar^\circ_{k,K_\pun,\ell,\ell_\for,\bi} \subset \Sbar_{k,K_\pun,\ell,\ell_\for,\bi}$ to be the complement of all puncturable points. 
One easily checks that
\[ \Sbar^\circ_{k,K_\pun,\ell,\ell_\for,\bi} \cap \mathrm{pr}_\alpha^* \EuS_{k,K_\pun,\ell,\ell_\for,\bi} = \mathrm{pr}_\alpha^* \EuS^\circ_{k,K_\pun,\ell,\ell_\for,\bi}\]
in \eqref{eq:Sdecomp}.

Now we need to make an appropriate choice of cylindrical and strip-like ends and thick/thin decomposition for our moduli spaces. 

We run into some new issues: firstly, the choice of cylindrical ends for $\Rbar_\ell$ coming from $\bN^0$ and $\bN^1$ need not be compatible. We resolve this by modifying the notion of a universal choice of cylindrical ends. Secondly, the pullback of a choice of cylindrical or strip-like end by a forgetful map may no longer be a cylindrical or strip-like end, as a `forgotten' marked point may lie inside the end. We resolve this by `shrinking' each end by a domain-dependent factor so that it becomes disjoint from all marked points. 

We now introduce the modified notion of a universal choice of cylindrical and strip-like ends. 
This consists of a choice of map $\epsilon_{\Gamma,\alpha,i}$ for each $k,K_\pun,\ell,\ell_\for,\bi$, tree $\Gamma$ indexing a stratum of $\Rbar_{k,K_\pun,\ell,\ell_\for,\bi}$, vertex $\alpha$ of $\Gamma$, and $i \in \{1,\ldots,k_\alpha\} \sqcup \{1,\ldots,\ell_\alpha\}$, where:
\begin{itemize}
\item if $i \in \{1,\ldots,\ell_\alpha\}$ corresponds to an interior point, then $\epsilon_{\Gamma,\alpha,i}$ is a smooth $S^1$-equivariant map
$$\mathbb{D} \times \Rbar^{\Gamma,(i)}_{k,K_\pun,\ell,\ell_\for,\bi} \to \mathrm{pr}_\alpha^* \Sbar_{\ell_\alpha}$$
such that for any $\xi \in \Rbar^{\Gamma,(i)}$, the (equivalence class under positive real scaling of the) differential of the map $\epsilon(-,\xi)$ at $0 \in \mathbb{D}$ coincides with $\xi$, representing a cylindrical end on each fibre of the universal family; 

\item if $i \in \{1,\ldots,k_\alpha\}$ corresponds to a boundary point, then $\epsilon_{\Gamma,\alpha,i}$ is a strip-like end for the $i$th boundary marked point of the component $\mathrm{pr}_\alpha^* \Sbar_{k_\alpha,\ell_\alpha}$ of the universal family;
\end{itemize}
We require that these choices restrict to a choice of cylindrical and strip-like ends on each fibre of $\EuS^\Gamma_{k,K_\pun,\ell,\ell_\for,\bi}$. 

Observe that $\Sym(\ell-\ell_\for) \times \Sym(\ell_\for) \subset \Sym(\ell)$ acts on $\Sbar^\circ_{k,K_\pun,\ell,\ell_\for,\bi} \to \Rbar^\circ_{k,K_\pun,\ell,\ell_\for,\bi}$, preserving the forgettability status of the marked points. 
We will require our universal choice of cylindrical and strip-like ends to be equivariant for this subgroup.

We now need to explain the relevant notion of consistency, which is slightly more elaborate than in Section \ref{sec:slends}. In words, we do \emph{not} make a universal choice of cylindrical ends for the moduli spaces of spheres $\Rbar_\ell$, we only make such a choice for the moduli spaces of discs $\Rbar_{k,\ell}$; we impose the maximal `consistency' of choices of cylindrical and strip-like ends which is well-defined in the absence of such a choice. 
More precisely, the notion of consistency has two parts: 
\begin{itemize} 
\item[]
\textbf{(Consistency with previous choices)} Let $\Gamma' \subset \Gamma$ be a subgraph, so that $\Rbar^{\Gamma'}_{k',K'_\pun,\ell',\ell'_\for,\bi'}$ is a factor of $\Rbar^\Gamma_{k,K_\pun,\ell,\ell_\for,\bi}$ under the product decomposition \eqref{eq:Ridecomp}. Precisely, this happens when $\Gamma'$ is connected, includes a disc vertex, and all edges of $\Gamma$ incident to a vertex of $\Gamma'$ are included in $\Gamma'$. We write $(\Gamma',\alpha,i) \subset (\Gamma,\alpha,i)$, for sets of data indexing cylindrical or strip-like ends, if $\Gamma' \subset \Gamma$ is as above, and $\alpha$ is contained in $\Gamma'$. 
In this situation, let $\mathrm{pr}_{\Gamma'}: \Rbar^\Gamma \to \Rbar^{\Gamma'}$ denote the projection. 
It is clear that $\mathrm{pr}_\alpha \circ \mathrm{pr}_{\Gamma'} = \mathrm{pr}_\alpha$. 
We require that
$$ \epsilon_{\Gamma,\alpha,i} = \left(\mathrm{pr}_{\Gamma'}\right)^* \epsilon_{\Gamma',\alpha,i}$$
in the case of a strip-like end, or
$$ \epsilon_{\Gamma,\alpha,i} = \left(\mathrm{pr}_{\Gamma'}^{(i)}\right)^* \epsilon_{\Gamma',\alpha,i}$$
in the case of a cylindrical end, where $\mathrm{pr}_{\Gamma'}^{(i)} : \Rbar^{\Gamma,(i)} \to \Rbar^{\Gamma',(i)}$ denotes the map induced by $\mathrm{pr}_{\Gamma'}$.

\item[]
\textbf{(Consistency with gluing)} For any $\Gamma$ indexing a stratum of $\Rbar_{k,K_\pun,\ell,\ell_\for,\bi}$, the choice of cylindrical and strip-like ends $\epsilon_{\Gamma,\alpha,i}$ over all $\alpha,i$ determines a gluing map 
$$\EuR^\Gamma_{k,K_\pun,\ell,\ell_\for,\bi} \times Gl \to \EuR_{k,K_\pun,\ell,\ell_\for,\bi},$$
where $Gl$ is an appropriate space of gluing parameters; furthermore, the glued surfaces inherit a choice of cylindrical and strip-like ends from the ends which do not get `glued'. We require that these agree with the choice of cylindrical and strip-like ends $\epsilon_{*,\alpha,i}$ where `$*$' is the graph indexing the top-dimensional stratum of $\Rbar_{k,K_\pun,\ell,\ell_\for,\bi}$ and $\alpha$ is its unique vertex, over a region where the gluing parameters are sufficiently small.
\end{itemize}

Next we need to explain the sense in which we require our choices of cylindrical and strip-like ends to be compatible with our previous choices from $\bN^0$ and $\bN^1$.

\begin{defn}
We say that a universal choice of cylindrical and strip-like ends is \emph{compatible with $\bN^1$} if whenever $K_\pun = \{1,\ldots,k\}$ and $\bi = (1,\ldots,1)$, the ends coincide with those from $\bN^1$, under the isomorphism $\Sbar^\circ_{k,K_\pun,\ell,\ell_\for,\bi} \simeq \Sbar^\circ_{k,\ell}$.
\end{defn}

The notion of compatibility with $\bN^0$ is more involved, and its definition requires some preparation.

\begin{defn}
Suppose that $2 \le |K_\pun|+2(\ell-\ell_\for)$ and $\bi = (0,\ldots,0)$. 
We define the map
\[ \ff: \Sbar_{k,K_\pun,\ell,\ell_\for,\bi} \to \Sbar_{|K_\pun|,\ell-\ell_\for}\]
which forgets all forgettable and unpuncturable marked points and stabilizes. 
That is, one applies the $\ff$-process  to obtain a collapsing map $c\colon \Sigma_r\to \Sigma_r'$, and then also collapses the images in $\Sigma_r'$ of the $\ff$-semistable  components to obtain the stabilized curve $\Sigma_r''$ with a collapsing map $c'\colon \Sigma_r' \to \Sigma_r''$; and $\ff|_{\Sigma_r} = c'\circ c$.
\end{defn}

\begin{lem}\label{lem:ffstrends}
We have
\[ \ff\left(\Sbar^\circ_{k,K_\pun,\ell,\ell_\for,\bi}\right) \subset \Sbar^\circ_{|K_\pun|,\ell-\ell_\for}.\]
On the other hand, let
\[\Sbar^{\ff\circ}_{k,K_\pun,\ell,\ell_\for,\bi} \subset \Sbar^\circ_{k,K_\pun,\ell,\ell_\for,\bi}\]
denote the complement of all $\ff$-semistable disc components and their attached trees of $\ff$-unstable components. Then 
\[ \ff^{-1}\left(\Sbar^\circ_{|K_\pun|,\ell-\ell_\for}\right) = \Sbar^{\ff\circ}_{k,K_\pun,\ell,\ell_\for,\bi}.\]
\end{lem}
\begin{proof}
This is straightforward from the definition of $\ff$.
\end{proof}

Now we explain the notion of `shrinking' an end: given $S \in \R_{\ge 0}$ and $z \in \mathbb{D}$ or $Z_\pm$, we define 
\begin{align*}
S \cdot z &= e^{-S} \cdot z \quad \text{if the domain is $z \in \mathbb{D}$};\\
S \cdot z &= (s\pm S,t) \quad \text{if the domain is $z = (s,t) \in Z_\pm$}.
\end{align*}

\begin{defn}\label{defn:N0compat}
We say that a universal choice of cylindrical and strip-like ends is \emph{compatible with $\bN^0$} if whenever $\bi=(0,\ldots,0)$ we have:
\begin{itemize}
\item If $\alpha$ is $\ff$-semistable, then the end $\epsilon_{\Gamma,\alpha,i}$ extends to a parametrization of the whole component $\Sigma_\alpha^\circ$ of $\Sigma^\circ_r$ (by $\R \times [0,1]$ if $\alpha$ is a disc, or $\CP^1$ if it is a sphere).
\item If $\alpha$ is $\ff$-stable and $i$ is unforgettable or puncturable, then we have an end $\epsilon_{\Gamma,\alpha,i}$ for $\Rbar^\Gamma_{k,K_\pun,\ell,\ell_\for,\bi}$ and a corresponding end $\epsilon_{\ff(\Gamma,\alpha,i)}$ for $\Rbar^{\ff(\Gamma)}_{|K_\pun|,\ell-\ell_\for}$, the latter coming from $\bN^0$. 
We require that these be compatible in the sense that
\begin{equation}\label{eq:ffstrip}
 \ff(\epsilon_{\Gamma,\alpha,i}(r,z)) = \epsilon_{\ff(\Gamma,\alpha,i)}(\ff(r),S_{\Gamma,\alpha,i}(r) \cdot z)
 \end{equation}
 for all $r \in \EuR^\Gamma_{k,K_\pun,\ell,\ell_\for,\bi}$, where $S_{\Gamma,\alpha,i}:\EuR^\Gamma_{k,K_\pun,\ell,\ell_\for,\bi}\to\R_{\ge 0}$ is a smooth function. (It follows from the first clause of Lemma \ref{lem:ffstrends} that the image of $\ff \circ \epsilon_{\Gamma,\alpha,i}$ lies in $\Sbar^\circ_{|K_\pun|,\ell-\ell_\for}$.)
\end{itemize}
\end{defn}

\begin{rmk}
Note that the smooth functions $S_{\Gamma,\alpha,i}(r)$ necessarily converge to infinity as $r$ approaches any boundary component of $\Rbar^\Gamma$ where a forgettable interior marked point of the component $\alpha$ bubbles off with the $i$th marked point. 
That is why we only impose the condition \eqref{eq:ffstrip} over $\EuR^\Gamma$, not $\Rbar^\Gamma$.
\end{rmk}

\begin{lem}\label{lem:mixendsex}
There exists a universal choice of cylindrical and strip-like ends for the moduli spaces $\Sbar_{k,K_\pun,\ell,\ell_\for,\bi}$, which is consistent, equivariant, and compatible with $\bN^0$ and $\bN^1$.
\end{lem}
\begin{proof}
Compatibility with $\bN^1$ uniquely determines the choice of ends on the moduli spaces labelled by $\bi = (1,\ldots,1)$, satisfying all the necessary requirements. 
Next we construct a universal choice of cylindrical and strip-like ends on the moduli spaces labelled by $\bi = (0,\ldots,0)$. 
Suppose we have made such a choice for all moduli spaces of dimension $<N$, and consider a moduli space of dimension $N$. 
Compatibility with $\bN^0$, together with a choice of the functions $S_{\Gamma,\alpha,i}$, determines the choice of ends associated to puncturable and unforgettable points on $\ff$-stable components. 
The functions $S_{\Gamma,\alpha,i}(r)$ must be chosen to be sufficiently large that the image of $\epsilon_{\ff(\Gamma,\alpha,i)}(\ff(r),S_{\Gamma,\alpha,i}(r) \cdot z)$ avoids the image under $\ff$ of all forgettable and unpuncturable marked points on $\Sigma_r$; they must be $\Sym(\ell - \ell_\for) \times \Sym(\ell_\for)$-equivariant; they must be chosen so that the corresponding ends satisfy the \textbf{(Consistency with previous choices)} condition; finally they must be chosen so that the corresponding ends satisfy the \textbf{(Consistency with gluing)} condition, which (in the case of gluing a $\ff$-semistable component) relies on our assumption that ends on $\ff$-semistable components extend to a parametrization of the component. 

Next we choose ends associated to puncturable and unforgettable points on $\ff$-semistable components. 
The requirement that they should extend to a parametrization of the whole component determines these up to `shrinking', and the procedure for choosing an appropriate shrinking is analogous to the case above. 
Next we choose ends for the unpuncturable and forgettable marked points; compatibility with $\bN^0$ doesn't impose any conditions on these, so the construction is strictly easier. 
This completes the inductive construction of ends on moduli spaces labelled by $\bi = (1,\ldots,1)$ and $\bi = (0,\ldots,0)$. 
It is then straightforward to extend to a choice of ends on moduli spaces with mixed labels, by a straightforward variation on the proof of Lemma \ref{lem:cons_ends}.
\end{proof}

We now introduce the appropriate notion of a thick-thin decomposition: it is an open set $\Sbar_{k,K_\pun,\ell_0,\ell_\for,\bi}^{thin} \subset \Sbar^\circ_{k,K_\pun,\ell_0,\ell_\for,\bi}$ whose intersection with each fibre consists of the strip-like ends associated to puncturable boundary marked points, together with a \emph{subset} of the gluing regions: namely, those formed by gluing two puncturable boundary points together (gluing regions formed by gluing two unpuncturable boundary points, or two interior points together do not get included in the thin region). 
As with the strip-like ends, the parametrizations of the gluing regions are required to be $\Sym(\ell-\ell_\for) \times \Sym(\ell_\for)$-equivariant. 
We also require that our thick-thin decomposition coincides with that from $\bN^1$, when $\bi=(1,\ldots,1)$, and that the image of the thin region under $\ff$ lies inside the thin region from $\bN^0$, when $\bi=(0,\ldots,0)$. 
By a straightforward extension of Lemma \ref{lem:mixendsex}, such a choice of thick-thin decomposition exists.
(We remark that the parametrizations of the gluing regions will automatically satisfy the analogue of \eqref{eq:ffstrip}, because they are constructed from the strip-like ends which have this property.)
  
\subsection{Construction of the doubled category: perturbation data}

Suppose now that we have made a choice of cylindrical and strip-like ends which is consistent, equivariant, and compatible with $\bN^0$ and $\bN^1$; and a corresponding choice of equivariant thick-thin decomposition.
We now choose Floer data for each pair of objects in $\EuC$, coinciding with those from $\bN^0$ on pairs of objects $(K^0,L^0)$, coinciding with those from $\bN^1$ on pairs of objects $(K^1,L^1)$, and chosen arbitrarily on pairs of objects with different labels. 
This allows us to define the morphism spaces in $\EuC$, precisely as in Definition \ref{defn:mor}. 
We have already defined the embeddings \eqref{eq:embeds} on the level of objects; it is clear from this convention how to define it on the level of morphisms.

Let $(k,K_\pun,\ell,\ell_\for,\bL)$ be such that
 \begin{itemize}
 \item $(k,K_\pun,\ell,\ell_\for,\bi)$ satisfies the conditions of Definition \ref{def:forgmods}, where $\bi = \bi(\bL)$;
 \item $K_\pun \neq \emptyset$;
 \item $L_{j-1} = L_j$ whenever $j \notin K_\pun$.
 \end{itemize}
Let $r \in \Rbar_{k,K_\pun,\ell,\ell_\for,\bi}$, $\Sigma_r$ the corresponding fibre of the universal family, $\Sigma^\circ_r$ the complement of all puncturable boundary marked points and nodes, and $\tilde{\Sigma}^\circ_r$ the normalization of $\Sigma^\circ_r$. 
We define:

\begin{defn}\label{def:mixpert_1}
A perturbation datum for $\Sigma^\circ_r$ equipped with boundary labelling $\bL$ consists of a pair $P = (Y,K)$ where 
$$Y \in C^\infty(\tilde{\Sigma}^\circ_r,\EuY_*(X,V_1)),\qquad K \in \Omega^1(\tilde{\Sigma}^\circ_r,\EuH),$$
satisfying
\begin{itemize}
\item []
\textbf{(Constant on spheres and $\ff$-unstable discs)} 
$Y$ is constant, and $K$ vanishes, on any spherical or $\ff$-unstable disc component of $\Sigma^\circ_r$; and furthermore, $Y(z_{\alpha \beta}) = Y(z_{\beta\alpha})$ for every sphere or unpuncturable disc edge $\alpha E \beta$;

\item [] \textbf{(Thin regions)} 
The restriction to the thin regions, and to $\ff$-semistable disc components (which are isomorphic to $\R \times [0,1]$), is given by the Floer data;\footnote{Note that the condition imposed on $\ff$-semistable disc components in Definition \ref{defn:N0compat} ensures that the requirements imposed over thin regions and $\ff$-semistable disc components are compatible.} 

\item[]
\textbf{(Boundary)} over each component $C$ of the boundary labelled by $L_C$, we have
\[ K(\xi)|_{L_C} =0 \quad \text{for all $\xi \in TC \subset T(\partial \Sigma^\circ_r)$}.\]

\item [] \textbf{(Maximum principle)} 
If $\ell=0$, then $Y \in \EuY^{\max{}}_* \subset \EuY_*$.   
\end{itemize}

\end{defn}

Notice that the perturbation data are defined at the unpuncturable marked points. 

Let $\Sigma^{\circ \circ}_r \subset \Sigma^\circ_r$ denote the complement of all marked points and nodes of $\Sigma_r$.  
Let $\Sbar^{\circ \circ} \subset \Sbar^{\circ}$ denote the union of subsets $\Sigma^{\circ \circ}_r \subset \Sigma^\circ_r$ over all fibres. 
Note that $\Sbar^{\circ \circ} \to \Rbar$ is a submersion, so the bundle of fibrewise one-forms $\Omega^1_{\Sbar^{\circ\circ}/\Rbar}$ makes sense.

\begin{defn}\label{def:mixpert}
A \emph{universal choice of perturbation data} consists of pairs 
\[ \bP_{k,K_\pun,\ell,\ell_\for,\bL} = (\bY_{k,K_\pun,\ell,\ell_\for,\bL},\bK_{k,K_\pun,\ell,\ell_\for,\bL})\]
 for all $(k,K_\pun,\ell,\ell_\for,\bL)$ such that
 \begin{itemize}
 \item $(k,K_\pun,\ell,\ell_\for,\bi)$ satisfies the conditions of Definition \ref{def:forgmods}, where $\bi = \bi(\bL)$;
 \item $K_\pun \neq \emptyset$;
 \item $L_{j-1} = L_j$ whenever $j \notin K_\pun$; 
 \end{itemize}
 where
\begin{align*}
 \bY_{k,K_\pun,\ell,\ell_\for,\bL} &\in C^\infty\left(\Sbar^\circ_{k,K_\pun,\ell,\ell_\for,\bi}, \EuY_*(X,V_1)\right), \\
 \bK_{k,K_\pun,\ell,\ell_\for,\bL} &\in \Omega^1_{\Sbar^{\circ\circ}_{k,K_\pun,\ell,\ell_\for,\bi}/\Rbar_{k,K_\pun,\ell,\ell_\for,\bi}}\left(\EuH\right).
 \end{align*}
The restriction to each fibre of the universal family should be a perturbation data for that fibre, in the sense of Definition \ref{def:mixpert_1}. 
Furthermore, the $\bP_{k,K_\pun,\ell,\ell_\for,\bL}$ should satisfy:
\begin{enumerate}

\item [] \textbf{(Equivariant)} 
$\bP_{k,K_\pun,\ell,\ell_\for,\bL}$ is $\Sym(\ell-\ell_\for) \times \Sym(\ell_\for)$-equivariant;

\item [] \textbf{(Consistent on $\ff$-stable discs)} 
The restriction of $\bP_{k,K_\pun,\ell,\ell_\for,\bL}$ to $\mathrm{pr}_\alpha^* \Sbar^\circ_{k_\alpha,K_{\alpha,\pun},\ell_\alpha,\ell_{\alpha,\for},\bL_\alpha}$ is equal to $\mathrm{pr}_\alpha^* \bP_{k_\alpha,K_{\alpha,\pun},\ell_\alpha,\ell_{\alpha,\for},\bL_\alpha}$, for any $\ff$-stable disc vertex $\alpha \in V_\disc$ of a tree $\Gamma$;

\item[] \textbf{(Compatibility with $\bP^1$)} 
In the case $\bi = (1,\ldots,1)$ (which determines $K_\pun = \{1,\ldots,k\}$ uniquely), we have
\[\bP_{k,K_\pun,\ell,\ell_\for,\bL} = \bP^1_{k,\ell,\bL'},\]
where $\bL'$ is obtained from $\bL$ by forgetting the $0/1$ labels;

\item[] \textbf{(Compatibility with $\bP^0$)} 
Let $\bi = (0,\ldots,0)$; then, under the map $\ff\colon \Sbar^{\ff\circ}_{k,K_\pun,\ell,\ell_\for,\bi}  \to  \Sbar^\circ_{|K_\pun|,\ell-\ell_\for}$ (cf. Lemma \ref{lem:ffstrends}) one has
\[\left.\bP_{k,K_\pun,\ell,\ell_\for,\bL}\right|_{\Sbar^{\ff\circ}_{k,K_\pun,\ell,\ell_\for,\bi}} = \ff^* \bP^0_{|K_\pun|,\ell-\ell_\for,\bL'},\]
where $\bL'$ is obtained from $\bL$ by merging all $0$-labelled Lagrangians separated by an unpuncturable marked point, then forgetting the $0/1$ labels.
\end{enumerate}
\end{defn}

\begin{rmk}
The perturbation scheme we use here is closely related to that developed by Woodward--Xu \cite{WX19} to prove an analogous result (namely, that Gromov--Witten invariants are independent of the choice of stabilizing divisor), but there is a small difference which may be illuminating to explain. 
Similarly to \emph{op. cit.}, we impose two types of requirements on the restriction of the perturbation data to a component of a nodal curve: consistency and constancy (together with a third, namely the \textbf{(Thin regions)} condition, which is specific to the setting of Fukaya categories). 
Our perturbation data would satisfy the analogue of `$\ff$-locality' (\emph{op. cit.}, Definition 2.1(b)) if we imposed consistency on all $\ff$-stable components, the \textbf{(Thin regions)} condition on $\ff$-semistable disc components, and constancy on the remaining components. 
This would be satisfied by any perturbation data pulled back via the forgetful map, so long as the perturbation data getting pulled back were consistent on all components. 
However, our perturbation data are not consistent on all components: rather, they satisfy the \textbf{(Constant on spheres)} condition from Definition \ref{def:perdat}. 
(We are able to impose this condition by our \textbf{(Semipositive$'$)} condition; it is part of how we achieve the \textbf{(Equivariant)} condition, which is what makes our counts of pseudoholomorphic curves be integers, rather than rational numbers.) 
For that reason, rather than the analogue of $\ff$-locality, we impose consistency only on the $\ff$-stable \emph{disc} components, the \textbf{(Thin regions)} condition on $\ff$-semistable disc components, and constancy on the remaining components (including \emph{all} spherical components). 
Lemma \ref{lem:pullbackdata} below establishes the key compatibility between Definitions \ref{defn:pert} and \ref{def:mixpert}.
\end{rmk}

\begin{defn}
We say that $\bP^0$ is \emph{$V_1$-compatible} if the almost-complex structure components $\bY^0$ satisfy
\[ \bY^0_{k,\ell} \in C^\infty(\Sbar^\circ_{k,\ell},\EuY_*(V_1)) \subset C^\infty(\Sbar^\circ_{k,\ell},\EuY_*(V_0)).\]
(Note that $\bP^0$ is always $V_0$-compatible.)
\end{defn}

\begin{lem}\label{lem:pullbackdata}
There exists a universal choice of perturbation data if and only if $\bP^0$ is $V_1$-compatible.
\end{lem}
\begin{proof}
The `only if' follows from the \textbf{(Compatibility with $\bP^0$)} condition, which implies that 
\[\ff^*\bY^0_{k,\ell,\bL} = \bY_{k,\{1,\ldots,k\},\ell,0,\bL'} \in C^\infty(\Sbar^\circ_{k,\ell},\EuY_*(V_1)),\]
 where $\bL'$ is obtained from $\bL$ by adding `$0$' labels to each Lagrangian.
 
To prove `if', we start by constructing the perturbation data over the moduli spaces $\Sbar_{k,K_\pun,\ell,\ell_\for,\bi}$ with $\bi=(0,\ldots,0)$. We claim that it is uniquely determined by $\bP^0$. 
Indeed, the \textbf{(Compatibility with $\bP^0$)} condition uniquely determines the perturbation data over $\Sbar^{\ff \circ}$; the \textbf{(Thin regions)} condition uniquely determines it over $\ff$-semistable disc components; and the \textbf{(Constant on spherical and $\ff$-disclike components)} uniquely determines it over the remaining components. 

We need to verify that the resulting perturbation data satisfy the relevant conditions of Definition \ref{def:mixpert}. We first observe that $\bY$ clearly maps to $\EuY_*(X,V_1)$, by hypothesis; \textbf{(Thin regions)} is satisfied by construction; \textbf{(Equivariant)} follows from the condition that $\bP^0$ be $\Sym(\ell - \ell_\for)$-equivariant, and the forgetful map is $\Sym(\ell_\for)$-equivariant; \textbf{(Consistent on $\ff$-stable discs)} follows from the \textbf{(Consistent on discs)} condition on $\bP^0$, by compatibility of the forgetful maps with inclusions of Deligne--Mumford strata; \textbf{(Maximum principle)} follows from the corresponding condition on $\bP^0$; \textbf{(Compatibility with $\bP^1$)} is vacuous; and \textbf{(Compatibility with $\bP^0$)} follows by construction. 
The least trivial to verify is \textbf{(Constant on spherical and $\ff$-disclike components)}, which requires us to split into two cases: spheres which get mapped to a sphere by the forgetful map, and spheres or discs which get contracted to a point by the forgetful map. In the former case, the condition is a consequence of the \textbf{(Constant on spheres)} condition on $\bP^0$; in the latter, it follows as the pullback of $\bY^0$ by a constant map is constant, and the pullback of $\bK^0$ by a constant map vanishes.

Next we construct the perturbation data over the moduli spaces with $\bi = (1,\ldots,1)$. These are uniquely determined by the \textbf{(Compatibility with $\bP^1$)} condition, and it is easy to verify that they satisfy the relevant conditions of Definition \ref{def:mixpert}. 
Finally we extend the perturbation data to moduli spaces with mixed boundary types, analogously to Lemma \ref{lem:pertinduct}.
\end{proof}

\subsection{Construction of the doubled category: pseudoholomorphic curves}

Recall that the local intersection numbers at $z$ of a holomorphic curve $u$ with the components of $V_1$ are recorded in a \emph{tangency vector} $\iota(u,z) \in (\Z_{\ge 0})^{Q_1}$. 
We define a tangency vector  to be \emph{forgettable} if it lies in $(\Z_{\ge 0})^{Q_1 \setminus Q_0}$, i.e. if $\iota(u,z)(q)=0$ when $q \in Q_0$. 

\begin{defn}
Given $A \in \pi_2^\num(\by)$ such that $A \cdot V_q \ge 0$ for all $q \in Q_1$, and $\ell_\for \le \ell$, a choice of \emph{tangency data} is a function $\tang:\{1,\ldots,\ell\} \to (\Z_{\ge 0})^{Q_1}$ such that $\sum_{i=1}^\ell \tang(i)_q = A \cdot V_q$, and furthermore, $\tang(i)$ is forgettable for all $i \ge \ell-\ell_\for+1$. 
\end{defn}

\begin{defn}
Suppose we are given a Lagrangian labelling $\bL$, a choice of Hamiltonian chords $\by$, an element $A \in \pi_2^\num(\by)$, $\ell_\for \le \ell$, a choice of tangency data $\tang$, and a universal choice of perturbation data $\bP$. 
If $k + 2\ell \ge 2$, we define $\cM(\by,A,\ell_\for,\tang,\bP)$ to be the moduli space of pairs $(r,u)$ where $r \in \EuR_{k,\{1,\ldots,k\},\ell,\ell_\for,\bi}$ ($\bi$ is the $0/1$-labelling corresponding to $\bL$) and $u: \Sigma^\circ_r \to X$ satisfies the pseudoholomorphic curve equation determined by the perturbation data $\bP_{k,\{1,\ldots,k\},\ell,\ell_\for,\bL}$, with tangency conditions at the marked points determined by $\tang$, and $[u]=A$. 
If $k+2\ell = 1$, so $(k,\ell) = (1,0)$ and $\by = (y_0,y_1)$, and we define $\cM(\by,A,0,\emptyset,\bP)$ to be the moduli space of Floer trajectories from $y_0$ to $y_1$, modulo translation.
\end{defn}

\begin{rmk}
Note that we are only defining moduli spaces with $K^0_\pun = \{0,1,\ldots,k\}$: i.e., we have not defined any moduli spaces with unpuncturable boundary marked points at this point.
\end{rmk}

\begin{rmk}
The moduli space $\cM(\by,A,\ell_\for,\tang,\bP)$ depends non-trivially on $\ell_\for$, the number of forgettable marked points. The difference can be seen most easily when thinking about the natural Gromov compactification of the moduli space. For example, suppose that $\ell = 1$ and $\tang = 0$. The difference between the cases $\ell_\for = 0$ and $\ell_\for = 1$ can be seen when the unique interior marked point approaches a boundary component labelled `$0$' and comes close to bubbling off inside a disc. 
If $\ell_\for = 0$, then a thin region develops where the perturbation data are required to be given by the Floer data, and the disc will break along a Hamiltonian chord; if $\ell_\for = 1$, then the marked point will instead bubble off inside a holomorphic disc (with vanishing Hamiltonian term in the Floer equation), connected by a node to the main component. 
\end{rmk}

\begin{lem}\label{lem:forgmod}
Suppose that the Lagrangian labelling $\bL$ has corresponding $0/1$-labelling $\bi = (0,\ldots,0)$, and $k+2(\ell-\ell_\for) \ge 1$. 
Define $\ff(\tang)$ to be the composition
\[ \{1,\ldots,\ell - \ell_\for\} \hookrightarrow \{1,\ldots,\ell\} \xrightarrow{\tang} (\Z_{\ge 0})^{Q_1} \xrightarrow{\mathsf{restriction}} (\Z_{\ge 0})^{Q_0}.\]
Then there is a well-defined map 
\begin{align*}
\ff: \cM(\by,A,\ell_\for,\tang,\bP) &\to \cM(\by,A,\ff(\tang),\bP^0), \qquad \text{which sends} \\
\ff(r,u) & := (\ff(r), u \circ \ff) \qquad\text{if $k+2(\ell-\ell_\for) \ge 2$.}
\end{align*}
In the remaining case $(k,\ell - \ell_\for) = (1,0)$, the map sends a pseudoholomorphic curve $(r,u)$ (or an equivalence class $[u]$ of Floer trajectories modulo translation, in the case $\ell = 0$) to the corresponding equivalence class $[u]$ of Floer trajectories modulo translation. 
The map is $\Sym(\ell - \ell_\for) \times \Sym(\ell_\for)$-equivariant (where $\Sym(\ell_\for)$ acts trivially on the target).
\end{lem}
\begin{proof}
The first thing to check is that $\cM(\by,A,\ff(\tang),\bP^0)$ is defined. 
This follows as $k+2(\ell-\ell_\for) \ge 1$ by hypothesis, and furthermore $\ff(\tang)$ is a valid choice of tangency data, by the condition that forgettable marked points carry forgettable tangency vectors. 

In the case $k+2(\ell-\ell_\for) \ge 2$, the well-definedness of the map is immediate from the \textbf{(Compatible with $\bP^0$)} condition on $\bP$. 
In the remaining case $(k,\ell-\ell_\for) = (1,0)$, the well-definedness follows from the \textbf{(Thin regions)} condition on $\bP$, as the domain is an $\ff$-semistable disc.
\end{proof}

\begin{lem}\label{lem:dimdiff}
The virtual dimensions of $ \cM(\by,A,\ell_\for,\tang,\bP)$ and $\cM(\by,A,\ff(\tang),\bP^0) $ are related by
\[ \dim \cM(\by,A,\tang,\bP) = \dim \cM(\by,A,\ff(\tang),\bP^0) + 2\ell_\for - 2\sum_{i=1}^\ell \sum_{q \in Q_1 \setminus Q_0} \tang(i)_q.\]
\end{lem}
\begin{proof}
As in (\ref{lem:dimifreg}), we have 
\begin{align*} 
\dim \cM(\by,A,\ell_\for,\tang,\bP) & = i(A) + k-2 + 2\ell - 2 \sum_{i=1}^\ell \sum_{q\in Q_1} \tang(i)_q , \\
\dim \cM(\by,A,\ff(\tang),\bP^0) & = i(A) + k-2 + 2(\ell - \ell_\for) - 2 \sum_{i=1}^{\ell - \ell_\for} \sum_{q\in Q_0} \tang(i)_q,
\end{align*}
while the forgettability condition is that $\tang(i)_q=0$ if $i>\ell-\ell_\for$ and $q\in Q_0$.
\end{proof}

\begin{lem}\label{lem:lift}
Let $(r,u) \in \cM(\by,A,\tang,\bP^0)$. 
Then there exists a number $\ell_\for$, choice of tangency data $\tang'$ satisfying $\ff(\tang') = \tang$, and element $(r',u') \in \cM(\by,A,\ell_\for,\tang',\bP)$, such that 
\begin{itemize}
\item $\ff(r',u') = (r,u)$;
\item $\tang'(i) \neq 0$ for all forgettable marked points $i$.
\end{itemize}
These choices are unique up to the action of $\Sym(\ell_\for)$ on the choice of tangency data.
\end{lem}
\begin{proof}
We add marked points to $u$ in the finite set of points where it intersects $V_1 \setminus V_0$; and we define $\tang'(i) := \iota(u,z_i)$ for each marked point.
\end{proof}

\begin{defn}
Suppose $\ell = A \cdot V$, $\vv: \{1,\ldots,\ell\} \to Q$ is a function with $|\vv^{-1}(q)| = A \cdot V_q$, and $\vv^{-1}(Q_1 \setminus Q_0) = \{\ell-\ell_\for+1,\ldots,\ell\}$. Then we can define a function $\tang^\can: \{1,\ldots,\ell\} \to \Z_{\ge 0}^{Q_1}$ by setting $\tang^\can(i)_q = 1$ if $\vv(i) = q$ and $0$ otherwise. We denote the resulting moduli space by 
\[ \cM(\by,A,\bP) := \cM(\by,A,\ell_\for,\tang^\can,\bP).\]
It is independent of the choice of $\vv$ up to the action of $\Sym(\ell-\ell_\for) \times \Sym(\ell_\for)$. 
We denote by $\Sym(\vv) \subset \Sym(\ell-\ell_\for) \times \Sym(\ell_\for)$ the subgroup preserving $\vv$, and observe that it acts freely on $\cM(\by,A,\bP)$.
\end{defn}

\subsection{Construction of the doubled category: transversality}

We introduce moduli spaces $\cM_{\{\Gamma_i\}}(\by,A,\ell_\for,\tang,\bP)$, consisting of pseudoholomorphic discs with holomorphic bubble trees attached, precisely as in Section \ref{sec:transv}. 
We similarly introduce the subspace $\cM^*_{\{\Gamma_i\}}(\by,A,\ell_\for,\tang,\bP)$ consisting of discs with simple bubble trees attached. 
The following Lemma is straightforward:

\begin{lem}\label{lem:forgmodbt}
Suppose that the Lagrangian labelling $\bL$ has $\bi(\bL) = (0,\ldots,0)$, and $k+2(\ell-\ell_\for) \ge 1$. 
The map $\ff$ from Lemma \ref{lem:forgmod} extends to a well-defined map 
\begin{align*}
\ff: \cM^*_{\{\Gamma_i\}}(\by,A,\ell_\for,\tang,\bP) &\to \cM^*_{\{\Gamma_i\}}(\by,A,\ff(\tang),\bP^0), \\
\ff\left((r,u),\{u_i\}_{i \in I}\right) & := \left(\ff(r,u), \{u_i\}_{i \in I}\right).
\end{align*}
\end{lem}

We also have:

\begin{lem}\label{lem:doublift}
Let $p \in \cM^*_{\{\Gamma_i\}}(\by,A,\tang,\bP^0)$. 
Then there exists a number $\ell_\for$, choice of tangency data $\tang'$ satisfying $\ff(\tang') = \tang$, and element $p' \in \cM^*_{\{\Gamma_i\}}(\by,A,\ell_\for,\tang',\ff^*\bP^0)$, such that 
\begin{itemize}
\item $\ff(p') = p$;
\item $\tang'(i) \neq 0$ for all forgettable marked points $i$.
\end{itemize}
Furthermore, these choices are unique up to the action of $\Sym(\ell_\for)$ on the tangency data, and we have
\[ \dim \left(\cM^*_{\{\Gamma_i\}}(\by,A,\ell_\for,\tang',\bP)\right) \le \dim \left( \cM^*_{\{\Gamma_i\}}(\by,A,\tang,\bP^0)\right). \]
\end{lem}
\begin{proof}
The first part follows easily from Lemma \ref{lem:lift}. 
For the second part, we note that Lemma \ref{lem:dimdiff} applies also to the moduli spaces of discs with bubble-trees:
 \[ \dim  \cM^*_{\{\Gamma_i\}}(\by,A,\ell_\for\tang' ,\bP) - \dim  \cM^*_{\{ \Gamma_i\} }(\by,A, \tang,\bP^0) =  2\ell_\for - 2\sum_{i=1}^\ell \sum_{q \in Q_1 \setminus Q_0} \tang(i)_q.\]
We combine this with the observation that
\begin{align*}
0 &\neq \tang'(i) \quad \text{for all $i \ge \ell-\ell_\for +1$} \\
0 &= \tang'(i)_q \quad \text{for all $i \ge \ell-\ell_\for+1$ and $q \in Q_0$} \\
\Rightarrow \ell_\for & \le \sum_{i=\ell-\ell_\for+1}^\ell \sum_{q \in Q_1 \setminus Q_0} \tang'(i)_q .
\end{align*}
\end{proof}

We introduce the notion of regularity of the moduli spaces $\cM^*_{\{\Gamma_i\}}(\by,A,\ell_\for,\tang,\bP)$ as before. 

\begin{lem}\label{lem:doubtransv}
Consider the setting of Lemma \ref{lem:genreg}, with the following modification: let us suppose that the previously-chosen perturbation data $\bP^0_{k',\ell',\bL'}$ are $V_1$-compatible. 
We consider the space of $\bP^0_{k,\ell,\bL}$ which satisfy the conditions required of a universal choice of perturbation data, and are furthermore $V_1$-compatible. 
The following subsets of this space are comeagre:
\begin{enumerate}
\item The set of $\bP^0_{k,\ell,\bL}$ such that $\cM(\by,A,\tang,\bP^0)$ is regular;
\item Given $\ell_\for$ and $\tang'$ such that $\ff(\tang') = \tang$, and a set of combinatorial types of bubbles trees $\{\Gamma_i\}$: the subset such that $\cM^*_{\{\Gamma_i\}}(\by,A,\ell_\for,\tang',\ff^*\bP^0)$ is regular. 
\end{enumerate}
In the case that $V_0=V_1$, the latter subset coincides with the subset such that $\cM^*_{\{\Gamma_i\}}(\by,A,\tang,\bP^0)$ is regular.
\end{lem}
\begin{proof}
The proof is a minor modification of that of Lemma \ref{lem:genreg}. 
\end{proof}

\begin{rmk}
Note that we can \emph{not} prove comeagreness of the set of $\bP^0_{k,\ell,\bL}$ such that $\cM^*_{\{\Gamma_i\}}(\by,A,\tang,\bP^0)$ is regular: the proof breaks down in the presence of bubble trees with components contained inside $V_q$ for some $q \in Q_1 \setminus Q_0$, due to our requirement that $\bP^0_{k,\ell,\bL}$ be $V_1$-compatible.
\end{rmk}

\begin{defn}
We say that a universal choice of perturbation data $\bP^0$ is $\ff$-regular if:
\begin{itemize}
\item Every moduli space $\cM(\by,A,\bP^0)$ of virtual dimension $\le 1$ is regular;
\item $\bP^0$ is $V_1$-compatible; 
\item Every moduli space $\cM(\by,A,\ff^*\bP^0)$ of virtual dimension $1$, with all boundary labels $0$, is regular;
\item Every moduli space $\cM^*_{\{\Gamma_i\}}(\by,A,\tang,\ff^*\bP^0)$ of negative virtual dimension, with all boundary labels $0$, is empty.
\end{itemize}
\end{defn}

\begin{lem}\label{lem:PmeansP0}
If $\bP^0$ is $\ff$-regular, then it is regular. 
\end{lem}
\begin{proof}
We must show that the moduli spaces $\cM^*_{\{\Gamma_i\}}(\by,A,\tang,\bP^0)$ having negative virtual dimension are empty. 
Suppose, for a contradiction, that $p$ is an element of such a moduli space. 
Then, in view of the ({\bf Compatibility with $\bP^0$}) condition on $\bP$, there exists $p' \in \cM^*_{\{\Gamma_i\}}(\by,A,\ell_\for,\tang',\bP)$ mapping to $p$ by Lemma \ref{lem:doublift}. 
However the virtual dimension of this moduli space is bounded above by that of $\cM^*_{\{\Gamma_i\}}(\by,A,\tang,\bP^0)$, which is negative by assumption: hence the moduli space is empty by $\ff$-regularity of $\bP^0$, giving a contradiction. 
\end{proof}

Recall that Definition \ref{defn:regularity} decreed that $\bP^0$ is regular if all moduli spaces $\cM^*_{\{\Gamma_i\}}(\by,A,\tang,\bP^0)$  of  negative virtual dimension and all moduli spaces $\cM(\by, A, \bP^0)$ of virtual dimension $\leq 1$ are regular.  Thus the moduli spaces $\cM(\by, A, \bP^0)$  directly involved in defining $\mu^k$ or in proving the $A_\infty$ equations are regular, while unwanted strata $\cM^*_{\{\Gamma_i\}}(\by,A,\tang,\bP^0)$ of the Gromov closure are empty. It was observed in the proof of Corollary \ref{cor:regex} that one can actually arrange that all moduli spaces $\cM^*_{\{\Gamma_i\}}(\by,A,\tang,\bP^0)$ are regular, regardless of dimension; but we did \emph{not} make this part of the definition of regularity for $\bP^0$. 
That is because had we done so, Lemma \ref{lem:PmeansP0} would not have held, for the reason mentioned after Lemma \ref{lem:doubtransv}. 

\begin{lem}\label{lem:fregex}
There exists an $\ff$-regular universal choice of perturbation data $\bP^0$. 
If $V_0 = V_1$, then $\bP^0$ is $\ff$-regular if and only if it is regular.
\end{lem}
\begin{proof}
An $\ff$-regular universal choice of perturbation data can be constructed inductively, using Lemma \ref{lem:doubtransv}. 
If $V_0=V_1$, then the definitions of $\ff$-regularity and regularity manifestly coincide.
\end{proof}

We now make a similar definition of regularity for $\bP$:
\begin{defn} 
We define $\bP$ to be regular if all moduli spaces $\cM^*_{\{\Gamma_i\}}(\by,A,\tang,\bP)$ having negative virtual dimension are regular (i.e., empty), and all moduli spaces $\cM(\by,A,\bP)$ having virtual dimension $\le 1$ are regular. 
\end{defn}

\begin{lem}
For any regular universal choice of perturbation data $\bP^1$, and any $\ff$-regular universal choice of perturbation data $\bP^0$, there exists a corresponding regular universal choice of perturbation data $\bP$. 
\end{lem}
\begin{proof}
The perturbation data $\bP^0$ (respectively $\bP^1$) determine the perturbation data $\bP$ on all moduli spaces with boundary labels $(0,\ldots,0)$ (respectively $(1,\ldots,1)$), and make the required moduli spaces of discs with bubble trees regular. 
We can extend these to perturbation data on moduli spaces with boundary labels of mixed types by Lemma \ref{lem:genreg}.
\end{proof}

Henceforth, we fix such a choice of $\bP^0,\bP^1,\bP$. 

\begin{lem}\label{lem:forgiso}
Suppose that $\cM(\by,A,\bP)$ has dimension $\le 1$. 
Then the map from Lemma \ref{lem:forgmod} defines an isomorphism
\[ \ff: \cM(\by,A,\bP)/\Sym(\ell - \ell_\for) \times \Sym(\ell_\for) \xrightarrow{\sim} \cM(\by,A,\bP^0)/\Sym(\ell-\ell_\for).\]
\end{lem}
\begin{proof}
We first check that the map is well-defined: i.e., that  $\ff(\tang^\can)$ is a canonical choice of tangency data and $k+2(\ell-\ell_\for) \ge 1$. 
The first part is clear. 
For the second part, suppose to the contrary that $(k,\ell-\ell_\for) = (0,0)$. 
Then because $\ff(\tang^\can)$ is a choice of tangency data, we must have $A \cdot V_q = 0$ for all $q \in Q_0$; since $V_0$ is a system of divisors, its components span $H^2(X,X \setminus D;\Q)$, so it follows that $A \cdot V_q = 0$ for all $q \in Q_1$ too. 
This implies that $\ell=\ell_\for=0$, so $(k,\ell) = (0,0)$, which is a contradiction as we assumed $\cM(\by,A,\bP)$ to be well-defined.

The uniqueness statement in Lemma \ref{lem:doublift} implies that it is injective, because $\tang^\can(i) \neq 0$ for all forgettable marked points $i$. 
Lemma \ref{lem:doublift} also implies surjectivity: if $\tang'$ denotes the tangency data for the lift provided by the Lemma, and $\tang' \neq \tang^\can$, then applying Lemma \ref{lem:dimdiff} as in the proof of Lemma \ref{lem:doublift} easily gives
\begin{align*}
 \dim \cM(\by,A,\tang',\bP) & = \dim(\by,A,\tang,\bP^0) + 2\ell_\for - 2 \sum_{i=1}^\ell \sum_{q \in Q_1 \setminus Q_0} \tang'(i)_q \\
 &\le 1 - 2 < 0.
 \end{align*}
Hence the moduli space $\cM(\by,A,\tang',\bP)$ is empty unless $\tang' = \tang^\can$; it follows that $\ff$ must be surjective, and is therefore an isomorphism.
\end{proof}

The following is straightforward from the definition:

\begin{lem}\label{lem:forg_sign}
Let $u$ be a regular element of a $0$-dimensional component of $\cM(\by,A,\bP)$. Then $\sigma_u = \sigma_{\ff(u)}$, where $\ff$ is the map from Lemma \ref{lem:forgmod}, and $\sigma_u$ and $\sigma_{\ff(u)}$ denote the isomorphisms \eqref{eq:signu}. 
\end{lem}

\subsection{Construction of the doubled category: compactness}

As before, we let $\Mbar(\by,A,\bP)$ denote the closure of $\cM(\by,A,\bP)$ inside its Gromov compactification; this exists because our perturbation data extend smoothly to the Deligne--Mumford compactification.

\begin{lem}\label{lem:nounpunc}
Elements of $\Mbar(\by,A,\bP)$ contain no unpuncturable points.
\end{lem}
\begin{proof}
Since elements of $\cM(\by,A,\bP)$ have no unpuncturable points, such a point could only develop if an $\ff$-unstable disc bubbled off in the Gromov limit. 
The proof that this cannot happen is the same as that of Corollary \ref{cor:nounstable}: $V_0$ is a system of divisors, hence would intersect such a disc in a stabilizing unforgettable marked point. 
\end{proof}

The proof of the following is the same as that of Lemma \ref{lem:nospheresdim1}:

\begin{lem}
Suppose that $\dim \cM(\by,A,\bP) \le 1$. 
Then elements of $\Mbar(\by,A,\bP)$ contain no sphere bubbles. 
\end{lem}

\subsection{Construction of the doubled category}

We have defined the objects and morphism spaces of $\EuC$. 
We define the $A_\infty$ structure maps by counting elements of $\cM(\by,A,\bP)/\Sym(\vv)$, precisely as in Definition \ref{def:ainf}. 
The proof that they satisfy the $A_\infty$ equations is the same: the terms are in one-to-one correspondence with boundary points of the compact one-manifolds $\Mbar(\by,A,\bP)$ (here we use Lemma \ref{lem:nounpunc} to show that no other types of boundary points appear). 
The subcategory consisting of objects $L^1$ is strictly isomorphic to $\fuk(V_1,\bP^1)$, by the \textbf{(Compatible with $\bP^1$)} condition on $\bP$; and the subcategory consisting of objects $L^0$ is strictly isomorphic to $\fuk(V_0,\bP^0) \otimes_{R(V_0)} R(V_1)$, by Lemma \ref{lem:forgiso} together with Lemma \ref{lem:forg_sign}.

The embeddings of \eqref{eq:embeds} are quasi-equivalences modulo $\fm$, as observed in \cite[Section 10a]{Seidel:FCPLT}; hence they are curved filtered quasi-equivalences. 
The set of $\ff$-regular perturbation data $\bP^0$ is a non-empty subset of the set of regular perturbation data, which coincides with the full set of regular perturbation data in the case $V_0 = V_1$, by Lemmas \ref{lem:PmeansP0} and \ref{lem:fregex}.  
This completes the proof of Proposition \ref{prop:indep}.

\section{Dependence on symplectic form and Liouville subdomain}\label{sec:dep2}

For the purposes of this section, we fix a closed manifold $X$, almost complex structure $J_0$, and a finite set $V = \{V_q\}_{q \in Q}$ of transversely-intersecting $J_0$-almost complex closed submanifolds of codimension $2$. 
We consider the dependence of the relative Fukaya category $\fuk(W \subset X,\omega,\theta,V,J_0,\bN)$ on the data $(W \subset X, \omega, \theta)$ (where $(W \subset X,J_0)$ admits a convex collar, and $V$ is a system of divisors for $W \subset X$).

We will abbreviate the notation for the relative Fukaya category to $\fuk(W \subset X,\omega,\theta)$ for the purposes of this section. 
Our main result is:

\begin{prop}[cf. Lemma 4.12 of \cite{Sheridan2016}]\label{prop:indep2}
Suppose we have a smoothly-varying family of data $(W_t \subset X,\omega_t,\theta_t)_{t \in [0,1]}$; and suppose that we also have a smoothly-varying family of convex collars $(U_t,h_t)$ for $(W_t \subset X,J_0)$. 
Then there is a curved filtered quasi-equivalence
$$\fuk(W_0 \subset X,\omega_0,\theta_0) \simeq \fuk(W_1 \subset X,\omega_1,\theta_1).$$
\end{prop} 

The rough idea of the proof will be to apply Moser's lemma to the family of exact symplectic manifolds $(W_t,\omega_t = d\theta_t)$. 

Now, let 
$$\EuY_{\delta} := \{Y \in \EuY: \|Y\|_{C^0}<\delta\}.$$

\begin{lem}\label{lem:C1}
Let $\psi:X \to X$ be a diffeomorphism. 
For $\delta>0$ sufficiently small, and $\psi$ sufficiently $C^1$-close to the identity and preserving $V$, there exists a map $\bar{\psi}:\EuY_{\delta} \to \EuY_*$ satisfying
$$\psi_* J_Y = J_{\bar{\psi}(Y)}.$$
\end{lem}
\begin{proof}
Recalling that $J_Y := J_0 \exp(Y)$, we find that $\bar{\psi}$ should satisfy
\begin{align*}
\psi_*(J_0) \exp(\psi_*Y) &= J_0 \exp(\bar{\psi}(Y))\\
\iff \bar{\psi}(Y) &= \log \left(J_0^{-1} \psi_*(J_0) \exp(\psi_*Y)\right).
\end{align*}
It is clear that for $\|Y\|_{C^0}$ sufficiently small and $\psi$ sufficiently $C^1$-close to the identity, $J_0^{-1} \psi_*(J_0) \exp(\psi_*Y)$ can be made arbitrarily $C^0$-close to $\id$, so that its logarithm is well-defined and arbitrarily $C^0$-close to $0$. 
In particular, we may ensure that $\bar{\psi}(Y)$ is well-defined and lies in $\EuY_*$.
\end{proof}

\begin{lem}\label{lem:psi_fcq}
Suppose we have data $(W_i \subset X, \omega_i,\theta_i)_{i \in \{0,1\}}$ as above, and a diffeomorphism $\psi:X \to X$, such that:
\begin{itemize}
\item $\psi|_{X \setminus W_1} = \id$;
\item $\psi|_{W_0}$ defines an exact symplectic embedding of $W_0$ into $W_1$, which induces an isomorphism between the respective Liouville completions;
\item $\psi$ is sufficiently $C^1$-close to the identity that there exists a constant $\delta>0$ and map $\bar{\psi}:\EuY_\delta \to \EuY_*$ as in Lemma \ref{lem:C1}.
\end{itemize}
Then there exists a curved filtered quasi-equivalence
$$\fuk(W_0 \subset X,\omega_0,\theta_0) \hookrightarrow \fuk(W_1 \subset X,\omega_1,\theta_1).$$
\end{lem}
\begin{proof}
First, the fact that $\psi|_{W_0}$ defines an exact symplectic embedding into $W_1$ implies that it induces a map from the objects of $\fuk(W_0 \subset X,\omega_0,\theta_0)$ to those of $\fuk(W_1 \subset X,\omega_1,\theta_1)$. 

We choose Floer and perturbation data for $\fuk(W_0 \subset X_0,\omega_0,\theta_0)$ so that the almost complex structure part belongs to $\EuY_\delta \subset \EuY_*$. 
(Because this is an open set, we may still find regular Floer and perturbation data inside it.)

We now explain why, by pushing forward by $\psi$, we obtain valid Floer and perturbation data for $\fuk(W_1 \subset X_1,\omega_1,\theta_1)$. 
Here the pushforward acts by $\bar{\psi}$ on the almost complex structure part $Y$, and $\psi_*$ on the Hamiltonian part $K$ of the Floer and perturbation data. 
Note that $\bar{\psi}(Y)$ is well-defined, by Lemma \ref{lem:C1}, and it is clear that it satisfies the conditions of Definition \ref{defn:pert}: the only one that requires a moment's thought is the \textbf{(Maximum principle)}, for which one must check that $\bar{\psi}(\Ymax_{0,\delta}) \subset \Ymax_{1,*}$. 
For this, suppose that $Y \in \Ymax_{0,\delta}$. 
Because $Y = 0$ outside $W_0$, which is contained in $W_1$ (this follows from $\psi(W_0) \subset W_1$ and $\psi|_{X \setminus W_1} = \id$), we have $J_Y = J_0$ outside $W_1$; we also have $\psi = \id$ in this region, so $\psi_*J_Y = J_0$ and hence $\bar{\psi}(Y) = 0$ outside $W_1$ as required. 
For the Hamiltonian parts $K_0$, $K_1$ of the data, we recall that $K_i$ belongs to $\EuH_i$, the set of smooth functions supported in $W_i$. 
Again it is clear that $\psi_*(\EuH_0) \subset \EuH_1$, because $W_0 \subset W_1$ and $\psi = \id$ outside $W_1$. 

We now claim that the map $u \mapsto \psi \circ u$ defines an isomorphism between all moduli spaces of pseudoholomorphic curves involved in the definition of the relative Fukaya category. 
This is because the map respects all of the data going into the pseudoholomorphic curve equation. 
The only part which requires some thought is the term $\nu(\xi)$ in the pseudoholomorphic curve equation, which is the Hamiltonian vector field associated to the Hamiltonian $K(\xi)$: this involves the symplectic form $\omega$, which need \emph{not} be respected by $\psi$. 
However, $K(\xi)$ is supported in $W_0$, and $\psi|_{W_0}$ is a symplectic embedding by assumption, so indeed this term in the pseudoholomorphic curve equation is respected by $\psi$.

It follows that we have a strict embedding 
$$\fuk(W_0 \subset X,\omega_0,\theta_0) \hookrightarrow \fuk(W_1 \subset X,\omega_1,\theta_1).$$ 
In order to check that this embedding is a curved filtered quasi-equivalence, it remains to observe that the resulting embedding $\fuk(W_0) \hookrightarrow \fuk(W_1)$ is a quasi-equivalence by \cite[Proposition 10.5]{Seidel:FCPLT} (using our assumption that the embedding of $W_0$ into $W_1$ induces an isomorphism between their completions).
\end{proof}

\begin{proof}[Proof of Proposition \ref{prop:indep2}]
It clearly suffices to prove that for any $T \in [0,1]$, there exists $\epsilon>0$ such that the categories $\fuk(W_t \subset X,\omega_t,\theta_t)_{t \in (T-\epsilon,T+\epsilon)}$ are all curved filtered quasi-equivalent. 

To do this, let us choose a smoothly-varying family of convex collars $(U_t,h_t)$, and translate the functions $h_t$ so that they are all fibrations over the interval $[0,d_t)$ for some $d_t \in \R_{>0}$ depending smoothly on $t \in [0,1]$. 
Then we can define new and larger Liouville domains $W^+_t := W_t \cup h_t^{-1}([0,d_t/2])$, which still admit convex collars (namely, $U^+_t:=h_t^{-1}([d_t/2,d_t))$ with $h^+_t := h_t|_{U^+_t}$). 
Note that $\fuk(W^+_t \subset X,\omega_t,\theta_t) \simeq \fuk(W_t \subset X,\omega_t,\theta_t)$ by the case $\psi=\id$ of Lemma \ref{lem:psi_fcq}. 

We now aim to use Lemma \ref{lem:psi_fcq} to produce curved filtered quasi-equivalences $$\fuk(W_T \subset X,\omega_T,\theta_T) \simeq \fuk(W^+_t \subset X,\omega_t,\theta_t)$$
for all $t \in (T-\epsilon,T+\epsilon)$. 
We produce the diffeomorphisms $\psi_t$, needed to apply the Lemma, by Moser's trick.
 
Precisely, let $\tilde{v}_t$ denote the vector field on $W^+_t$ whose contraction with $\omega_t$ is equal to $d \theta_t/dt$. 
Let $\beta:X \to \R$ be a smooth function which is equal to $1$ over some region containing $W_T$ in its interior, and equal to $0$ over some region containing $X \setminus W_T^+$ in its interior. 
Shrinking $\epsilon>0$ if necessary, we may assume that $\beta=0$ over some region containing $X \setminus W_t^+$, for all $t \in (T-\epsilon,T+\epsilon)$. 
Then we may define the smooth, $t$-dependent vector field
$$
v_t = \left\{\begin{array}{ll}
				\beta \cdot \tilde{v}_t & \text{over $W_t^+$}\\
				0 & \text{elsewhere.}
\end{array}\right.
$$  
Let $\psi_t:X \to X$ denote the flow of the time-dependent vector field $v_t$, with initial condition $\psi_T = \id$. 

We now check that $\psi_t$ satisfies the hypotheses of Lemma \ref{lem:psi_fcq} for all $t \in (T-\epsilon,T+\epsilon)$, for $\epsilon>0$ sufficiently small:
\begin{itemize}
\item we have $\psi_t|_{X \setminus W_t^+} = \id$ by construction, as $v_t \equiv 0$ over this region. 
\item $\psi_t(W_T)$ remains within the region where $\beta = 1$ for all $t \in (T-\epsilon,T+\epsilon)$, for $\epsilon>0$ sufficiently small.
When this is the case, Moser's lemma implies that $\psi_t$ defines an exact symplectomorphism from $W_T$ onto its image in $W_t^+$, for all $t \in (T-\epsilon,T+\epsilon)$; and it is standard that this embedding defines an isomorphism between the respective Liouville completions, cf. \cite[Lemma 7.2]{Seidel:FCPLT}. 
\item note that $\psi_t$ converges smoothly to $\psi_T = \id$ as $t \to T$, so for $\epsilon>0$ sufficiently small we may assume that $\psi_t$ is arbitrarily $C^1$-close to $\id$ for $t \in (T-\epsilon,T+\epsilon)$. 
\end{itemize}
Thus we may apply Lemma \ref{lem:psi_fcq}, and the proof is complete.
\end{proof}

\section{Algebro-geometric relative Fukaya category}\label{sec:a_geom}

Let $(X,D,\Nef)$ be as in Section \ref{sec:alg_geom}. In this section describe how to construct, from $(X,D,\Nef)$, the data $(W \subset X,\omega,\theta,V,J_0,\bN)$ needed to construct the relative Fukaya category. 

\subsection{Relative K\"ahler form, Liouville subdomain, convex collar}\label{sec:aux_data}

We equip $X$ with the integrable complex structure $J_0$.
 
We recall the notion of a \emph{relative K\"ahler form} on $(X,D)$ from \cite[Definition 3.2]{Sheridan2016}. 
It is a K\"ahler form $\omega$ on $X$, equipped with a K\"ahler potential $h:X \setminus D \to \R$ having a prescribed form near $D$.  
The K\"ahler potential $h$ determines a Liouville one-form $\theta := -d^c h$, so that any regular sublevel set $\{h \le c\}$ is a Liouville domain. 

A relative K\"ahler form determines a cohomology class $\kappa = [\omega;\theta] \in H^2(X,X \setminus D;\R) \simeq \R^P$.  
Such a cohomology class is represented by a relative K\"ahler form if and only if the corresponding $\R$-divisor $\sum_p a_p \cdot D_p$ is ample, and has all $a_p>0$ \cite[Lemma 3.3]{Sheridan2016}. 
In particular, by the \textbf{(Ample)} condition on $\Nef$, there exists a relative K\"ahler form whose cohomology class $\kappa$ lies in the interior of $\Nef$.

We choose such a relative K\"ahler form, and equip $X$ with the corresponding K\"ahler form $\omega$. 
The prescribed form of $h$ near $D$ ensures that $h$ is proper, and also that the critical points of $h$ form a compact subset of $X \setminus D$. 
We set $W = \{h \le c\}$ for some $c$ larger than all critical values of $h$, and equip $W$ with the primitive $\theta$, making it into a Liouville subdomain.

If we choose $c' > c$, then we have a convex collar $(U,h)$ for $(W \subset X,J_0)$, where $U = h^{-1}([c,c'))$. 
(This is a convex collar because $\theta = -dh \circ J_0$ by definition, and $dh$ is a submersion because $h$ has no critical values in $[c,c')$.)

The space of relative K\"ahler forms is open and convex, hence contractible; and the space of choices of $c$, for a given relative K\"ahler form, takes the form $(C,\infty)$ where $C$ is the largest critical value of $h$. 
From this one can easily see that the space of parameters on which $W \subset X,\omega,\theta$ depend is path-connected; moreover, we can connect any two choices by a smooth path, which comes with a smoothly-varying family of convex collars. 

\begin{rmk}
The particular space of choices of $W \subset X,\omega,\theta$ described here has the advantage of being obviously path-connected, but of course more general choices exist.
\end{rmk}

\subsection{System of divisors}\label{sec:sod}

It is straightforward to show that $W \subset X \setminus D$ is a deformation retract. Thus we have an isomorphism $H^2(X,W) \simeq H^2(X,X \setminus D) \simeq \Z^P$. 

Because $\Nef$ is rational polyhedral and satisfies the \textbf{(Semiample)} condition, there exists a system of divisors $V$ for $(W \subset X,J_0)$ such that $\Nef(V) = \Nef$, by Bertini's theorem (see \cite[Lemma 3.8]{Sheridan2016}). 
Furthermore, the system of divisors can be chosen to lie in an arbitrarily small neighbourhood of $D$. 
We choose such a system of divisors $V$. 

\subsection{Definition}

After choosing the data $\bN$ consisting of cylindrical and strip-like ends, thick-thin decomposition, Floer and perturbation data, we have the complete set of data needed to define the algebro-geometric relative Fukaya category:
$$ \fuk(X,D,\Nef) := \fuk(W \subset X,\omega,\theta,V,J_0,\bN).$$
To justify the notation, we must show that this category is independent of the auxiliary choices made in its construction. 

\begin{prop}\label{prop:indep3}
The category $\fuk(W \subset X,\omega,\theta,V,J_0,\bN)$ is independent of the choices of $W,\omega,\theta,V,\bN$ made above, up to curved filtered quasi-equivalence.
\end{prop}
\begin{proof}
Independence of $\bN$ was established in Proposition \ref{prop:indep} (the case $V_0 = V_1$). 

To establish independence of $V$, we first consider the case that $V$ and $V'$ are two systems of divisors which are transverse, so that $V \cup V'$ is also a system of divisors. 
Then Proposition \ref{prop:indep} yields curved filtered quasi-equivalences
$$\fuk(V) \simeq \fuk(V \cup V') \simeq \fuk(V')$$
as required. 
If $V$ is not transverse to $V'$, then we may introduce an intermediate system of divisors $V''$ which is transverse to both $V$ and $V'$, by Bertini's theorem; then we have
$$\fuk(V) \simeq \fuk(V'') \simeq \fuk(V')$$
as required.

Now recall that any two choices of $W \subset X,\omega,\theta$ can be connected by a smooth path $(W_t \subset X,\omega_t,\theta_t)_{t \in [0,1]}$, together with a smoothly-varying family of convex collars $(U_t,h_t)_{t \in [0,1]}$.  
We may choose a system of divisors $V$ which lies in the complement of the compact set $\cup_{t \in [0,1]} W_t$, by Bertini's theorem; then Proposition \ref{prop:indep2} yields a curved filtered quasi-equivalence between the relative Fukaya categories at the endpoints. 
\end{proof}

\begin{prop}
If $\Nef_0 \subset \Nef_1$, then there is a curved filtered quasi-equivalence
$$ \fuk(X,D,\Nef_1) \otimes_{R(\Nef_1)}R(\Nef_0) \simeq \fuk(X,D,\Nef_0).$$
\end{prop}
\begin{proof}
We first choose a system of divisors $V_0$ such that $\Nef(V_0) = \Nef_0$, then enlarge it to a system of divisors $V_1$ such that $\Nef(V_1) = \Nef_1$ (which is possible by Bertini's theorem). 
The result then follows from Proposition \ref{prop:indep}.
\end{proof}

\subsection{Other properties of the relative Fukaya category}\label{sec:other_ass}

We remark that \cite[Assumptions 5.2 and 5.3]{Sheridan2016} are also satisfied by $\fuk(X,D,\Nef)$. The arguments are straightforward variations on the proof of Proposition \ref{prop:indep}, so we omit them. (For the purposes of the argument, one can treat a single component of $D$ as if it were a system of divisors, by \cite[Lemma 4.9]{Sheridan2016}.)

\bibliographystyle{amsalpha}
\bibliography{./mybib}

\end{document}